\documentclass{amsart}
\usepackage{amsfonts,amssymb,stmaryrd,amscd,amsmath,latexsym,amsbsy,
graphicx,xypic}
\xyoption{all}
\numberwithin{equation}{section}
\theoremstyle{plain}

\newtheorem{thm}{Theorem}[section]

\newtheorem{prop}[thm]{Proposition}

\newtheorem{defi}[thm]{Definition}
\newtheorem{lem}[thm]{Lemma}
\newtheorem{cor}[thm]{Corollary}
\newtheorem{eg}[thm]{Example}

\theoremstyle{remark}
\newtheorem{rema}[thm]{Remark}

\title[Quantum affine KZ equations]
{Quantum affine Knizhnik-Zamolodchikov equations and quantum
spherical functions, I}
\author{Jasper V. Stokman}
\address{KdV Institute for Mathematics, University of Amsterdam,
Science Park 904, 1098 XH Amsterdam, The Netherlands.}
\email{j.v.stokman@uva.nl}
\begin{document}
\begin{abstract}
Cherednik's quantum affine Knizhnik-Zamolodchikov equations associated to an
affine Hecke algebra module $M$ form a holonomic system of $q$-difference
equations acting on $M$-valued functions on a complex torus $T$. In this paper
the quantum affine Knizhnik-Zamolodchikov equations are related to 
the Che\-red\-nik-Mac\-do\-nald theory when $M$ is induced from
a character of a standard parabolic subalgebra
of the affine Hecke algebra.
We set up correspondences between solutions of the quantum affine KZ equations
and, on the one hand, 
solutions to the spectral problem of the
Cherednik-Dunkl $q$-difference reflection
operators (generalizing work of
Kasatani and Takeyama) and, on the other hand, solutions to the spectral
problem of the 
Cherednik-Macdonald $q$-difference operators (generalizing work of Cherednik).
The correspondences are applicable 
to all relevant spaces of functions on $T$ and for all
parameter values,
including the cases that $q$ and/or 
the Hecke algebra parameters are roots of unity.
\end{abstract}
\maketitle
\setcounter{tocdepth}{1}
\tableofcontents

\section{Introduction}
This is the first part of a sequel of papers reporting on the
analysis of Cherednik's \cite{CQKZ, CInd} 
quantum affine Knizhnik-Zamolodchikov (KZ) equations and
their applications to quantum harmonic analysis and integrable systems. 

Cherednik's \cite{CQKZ, CInd}
quantum affine KZ equations associated to
an affine Hecke algebra module $M$ form a holonomic system of first
order $q$-difference equations (an integrable $q$-connection)
acting on $M$-valued functions on a complex
torus $T$; typically 
meromorphic or Laurent polynomial solutions are considered. 
For $0<|q|<1$ and for $M$ a principal
series module 
(i.e. obtained from inducing a character of the minimal
standard parabolic subalgebra of the affine Hecke algebra), 
Cherednik \cite{CInd} studied a correspondence between meromorphic
solutions of the associated quantum affine KZ equations on the one hand,
and common meromorphic
eigenfunctions of the Cherednik-Macdonald $q$-difference operators
on the other hand. 
One of the objectives of the present paper is to determine explicit
conditions on the induction data of the principal series module
to ensure the bijectivity of this correspondence.

In addition we generalize and refine 
Cherednik's correspondence, allowing modules $M$ that are 
induced from a character of
a standard parabolic subalgebra of the affine Hecke algebra,
allowing arbitrary classes of functions on $T$, and allowing
all values of the quantum parameter $q$ and of 
the Hecke algebra parameters $k$
(including roots of unity). 
The main properties of this generalized and refined correspondence
are stated in Theorem \ref{CMcorr}.

The classical analogue of Cherednik's correspondence is due to 
Matsuo \cite{Mat} and Cherednik \cite{CInt}. 
It was pursued further by Opdam \cite{O} and by 
Cherednik and Ma \cite{CM}.
A substantial part of the present work is inspired by Opdam's \cite[\S 3]{O}
approach to the classical correspondence, in which Cherednik's
trigonometric analogues of Dunkl's
differential-reflection operators naturally
come into play. In the present quantum setup, Opdam's approach 
suggests a natural
intermediate stage of the correspondence in which solutions of
quantum affine KZ equations are related to common eigenfunctions
of the Cherednik-Dunkl commuting $q$-difference reflection operators.
Such a correspondence has indeed recently been
established for root system of type A
by Kasatani and Takeyama \cite{KT}. We extend these results
to arbitrary root systems in Theorem \ref{mainthmY}.

For $q=1$, for root system of
type A,
and for $M$ induced from a character a maximal
standard parabolic subalgebra, the $q$-connection
matrices of the quantum KZ equations
are interpolants of the transfer matrix of an inhomogeneous
spin chain. It leads to the possibility to construct particular
eigenstates for XXZ spin chains from solutions of quantum affine KZ
equations.
In the case that the Hecke algebra parameter is a third root
of unity this approach is explored extensively in the context of
the Razumov-Stroganov conjectures
(see, e.g., \cite{RS, Pas, dFZ, KT, KP}). The correspondences 
investigated in the present paper, in case that the underlying
root system is of classical type, are expected 
to be useful in the
analysis of recent generalizations \cite{dGPS, dGP} of the
Razumov-Stroganov conjectures. 

Part II of the present paper
will be devoted to the interplay between quantum affine KZ equations
and quantum harmonic analysis. 
The crucial starting point will be the fact that 
for $0<|q|<1$, asymptotically
free solutions of the quantum KZ equations 
can be constructed which map, under the correspondence,
to $q$-analogues of the Harish-Chandra series (see \cite{vMS, vM}).
The results in the present paper then lead to explicit conditions on the
spectral parameters to ensure 
that the $q$-analogues of the Harish-Chandra series become
a basis of the meromorphic solution space of the spectral problem of the 
Cherednik-Macdonald $q$-difference operators.
Combined with the recent results of Cherednik 
\cite{CWhit} on $q$-analogues of the Harish-Chandra $c$-function,
this will accumulate in the derivation of 
an explicit expansion of Cherednik's \cite{CMehta}
quantum spherical function in
terms of the $q$-analogues of the Harish-Chandra series
(generalizing Harish-Chandra's $c$-function expansion
of the spherical function).\\
\noindent
{\bf Conventions}\\
Lots of results come in two forms: a ``$+$''-version, related to
symmetric theory, and a ``$-$''-version, related to antisymmetric
theory. We formulate both versions at the same time, labeling the objects
by $\pm$. An equality like $\pm a=\mp b$ should thus be read as
$a=-b$ and $-a=b$; it will always
be clear from context which of the two
equalities should be seen as the $+$-version and which as the 
$-$-version.\\ 
\noindent
{\bf Acknowledgments}\\
The author is supported by the Netherlands
Organization for Scientific Research (NWO) in the VIDI-project
``Symmetry and modularity in exactly solvable models''.
The author thanks Ivan Cherednik,
Eric Opdam and Paul Zinn-Justin for valuable discussions.

\section{The classes of affine Hecke algebra modules}
In this section we recall the definition of the affine Hecke algebra.
In addition, we discuss the affine Hecke algebra modules
obtained by inducing a suitable character of a standard parabolic
subalgebra.
\subsection{Weyl groups and parabolic subgroups}

Let $R_0\subset V_0$ be a finite, crystallographic, reduced, irreducible
root system in an Euclidean space $\bigl(V_0,\langle\cdot,\cdot\rangle\bigr)$ 
of dimension $n$.
We normalize the roots in such a way that long roots have squared
length $2$. The Weyl group of $R_0$ is denoted by $W_0$.
We fix a basis $\Delta_0=\{\alpha_1,\ldots,\alpha_n\}$ 
of $R_0$ once and for all. Let $R_0^{\pm}$, $\varphi$, 
$\{s_i\}_{i=1}^n$, $w_0$ be the associated positive/negative
roots, longest root, simple reflections and longest Weyl group
element, respectively. The length of a Weyl group element
$w\in W_0$ is $l(w):=\#(R_0^+\cap w^{-1}R_0^-)$. The associated
Bruhat ordering on $W_0$ is denoted by $\leq$.

Unless specified explicitly
otherwise, $I$ will always stand for a fixed subset of $\{1,\ldots,n\}$.
We write $W_{0,I}$ for the subgroup of $W_0$ generated by
the simple reflections $s_i$ ($i\in I$). Then 
$R_0^I=R_0\cap\textup{span}_{\mathbb{Z}}\{\alpha_i\}_{i\in I}$
is a root system in $V_{0,I}:=\bigoplus_{i\in I}\mathbb{R}\alpha_i$
with Weyl group isomorphic to $W_{0,I}$.
Furthermore, $\{\alpha_i\}_{i\in I}$ is a basis of $R_0^I$.
We write $R_0^{I,\pm}$ for the associated set of positive and negative
roots, respectively. The length function on $W_{0,I}$
coincides with the restriction of the length function $l$ to
$W_{0,I}$. 

Set
\begin{equation*}
\begin{split}
W_0^I:=&\{w\in W_0 \,\, | \,\, l(ws_i)>l(w)\quad \forall\, i\in I\}\\
=&\{w\in W_0 \,\, | \,\, w(R_0^{I,+})\subseteq R_0^+\}.
\end{split}
\end{equation*}
It is a complete set of representatives of the coset space
$W_0/W_{0,I}$. Furthermore,
\[
l(uv)=l(u)+l(v),\qquad \forall\, u\in W_0^I,\, \forall\, v\in W_{0,I}.
\]
The elements of $W_0^I$ are called the minimal coset representatives 
of $W_0/W_{0,I}$. When decomposing an element $w\in W_0$ as a product
of a minimal coset representative and an element of $W_{0,I}$, we use the
notation
\[
w=\overline{w}\underline{w}\qquad (\overline{w}\in W_0^I,\, 
\underline{w}\in W_{0,I})
\]
(although $\overline{w}$ and $\underline{w}$ depends on the
choice of $I$, we suppress this from the notations).
We will frequently make use of the following elementary
lemma (see \cite[Lemma 2.1]{D}).
\begin{lem}\label{ABC}
Fix $1\leq i\leq n$. Define 
\begin{equation*}
\begin{split}
A_i&=\{w\in W_0^I \,\, | \,\, l(s_iw)=l(w)-1\},\\
B_i&=\{w\in W_0^I \,\, | \,\, l(s_iw)=l(w)+1\,\, \& \,\, s_iw\in W_0^I\},\\
C_i&=\{w\in W_0^I \,\, | \,\, l(s_iw)=l(w)+1\,\, \& \,\, s_iw\not\in W_0^I\}.
\end{split}
\end{equation*}
Then 
\begin{enumerate}
\item[{\bf (i)}]
$W_0^I=A_i\cup B_i\cup C_i$ (disjoint union).
\item[{\bf (ii)}] The map $w\mapsto s_iw$ defines an involution
of $A_i\cup B_i$. It maps $A_i$ onto $B_i$.
\item[{\bf (ii)}] For $w\in C_i$ there exists a unique $i_w\in I$
such that $s_iw=ws_{i_w}$. In particular, $\overline{s_iw}=w$.
\end{enumerate} 
\end{lem}
We record here a technical lemma, which we will be needing at a later stage.
\begin{lem}\label{Blift}
Suppose $w\in W_0^I$ and $w\not\in B_i$ for all $1\leq i\leq n$.
Then $w=\overline{w_0}$.
\end{lem}
\begin{proof}
Suppose to the contrary that $w\not=\overline{w_0}$.
Then $w\underline{w_0}\not=w_0$, hence there exists
an $i\in\{1,\ldots,n\}$ such that 
\begin{equation}\label{Blifteq}
l(s_iw\underline{w_0})=
l(w\underline{w_0})+1=l(w)+l(\underline{w_0})+1.
\end{equation}
If $s_iw\in W_0^I$ then it follows from \eqref{Blifteq}
that $l(s_iw)=l(w)+1$,
contradicting the fact that $w\not\in B_i$. If $s_iw\not\in W_0^I$
then $w\in C_i$, hence
$s_iw=ws_{i_w}$ with $i_w\in I$. Consequently
\[
l(s_iw\underline{w_0})=l(ws_{i_w}\underline{w_0})=
l(w)+l(s_{i_w}\underline{w_0})=
l(w)+l(\underline{w_0})-1,
\]
which contradicts \eqref{Blifteq}.
\end{proof}

\subsection{Affine root systems and affine Weyl groups}
We recall in this subsection some facts on
twisted affine root systems. We stick as much as possible
to the treatment in \cite[Chpt. 1\& 2]{M}.

Identify $V:=\mathbb{R}c\oplus V_0$
with the real vector space of affine
linear, real valued functionals on $V_0$ by interpreting $rc+v$
as the functional $v^\prime\mapsto r+\langle v,v^\prime\rangle$.
Let $D: V\rightarrow V_0$ be the projection onto $V_0$ along the direct
sum decomposition $V=\mathbb{R}c\oplus V_0$. 
We extend the scalar product $\langle \cdot,\cdot\rangle$ to 
a symmetric bilinear form on $V$ by requiring $D$ to be form preserving.
The correponding
semi-norm on $V$ is denoted by $\|\cdot\|$. 
We write $a^\vee=2a/\|a\|^2$ for a vector $a\in V$ satisfying $\|a\|\not=0$.

The twisted affine root system associated to $R_0$ is the set
\[R:=\{a^\vee \,\, | \,\, a\in\mathbb{Z}c+R_0\}\subset V.
\]
Let $W^a\subset \textup{O}(V)$ (with $\textup{O}(V)$ the group of invertible
linear endomorphisms of $V$ preserving the symmetric bilinear form
$\langle \cdot,\cdot\rangle$)
be the subgroup generated by the involutions
\[s_a: v\mapsto v-\langle a,v\rangle a^\vee,\qquad v\in V
\]
for all $a\in R$. It is called the affine Weyl group of $R$.

The affine Weyl group $W^a$ admits two important alternative descriptions,
namely as the semidirect product group $W_0\ltimes Q^\vee$ associated to 
the ($W_0$-invariant) coroot lattice $Q^\vee=
\textup{span}_{\mathbb{Z}}\{\alpha^\vee \, | \, \alpha\in R_0\}$
of $V_0$, and as a Coxeter group. Concretely, denote by 
$\tau: V_0\rightarrow \textup{O}(V)$ the semi-norm preserving
action of $V_0$ on $V$ given
by 
\[\tau(v)(rc+v^\prime)=(r-\langle v,v^\prime\rangle)c+v^\prime
\]
for $v,v^\prime\in V_0$ and $r\in\mathbb{R}$.
Then $w\tau(v)w^{-1}=\tau(wv)$ for $w\in W_0$ and $v\in V_0$.
Furthermore, $s_{a}=s_{\alpha}\tau(r\alpha^\vee)$ for 
$a=(rc+\alpha)^\vee$ ($r\in\mathbb{Z}$, $\alpha\in R_0$) and
$W_0\ltimes Q^\vee\simeq W^a$ by $(w,\lambda)\mapsto w\tau(\lambda)$.
It follows from this description of $W^a$ that $R$
is $W^a$-invariant. 

In fact, $R$ 
is invariant under the action of the extended affine Weyl
group 
\[
W:=\{w\tau(\lambda)\}_{w\in W_0, \lambda\in P^\vee}
\simeq W_0\ltimes P^\vee,
\] 
where $P^\vee:=\{\lambda\in V_0 \,\, | \,\, \langle \lambda,\alpha\rangle
\in\mathbb{Z} \,\, \forall\, \alpha\in Q\}$ 
is the coweight lattice of $R_0$ in $V_0$. 
We write $\{\varpi_i^\vee\}_{i=1}^n\subset P^\vee$ 
for the fundamental coweights with respect
to the ordered basis $\Delta_0$ of $R_0$ (they are characterized by
$\langle \varpi_i^\vee,\alpha_j\rangle=\delta_{i,j}$ for all $1\leq j\leq n$).

Observe that the affine Weyl group
$W^a$ is a normal subgroup of $W$ with finite abelian quotient group
$W/W^a\simeq P^\vee/Q^\vee$.

We recall now the presentation of $W^a$ as a Coxeter group. 
An ordered basis $\Delta$ of the twisted affine root system $R^\vee$ 
is given by
\begin{equation}\label{basis}
\{a_0,a_1,\ldots,a_n\}:=\{c-\varphi^\vee,\alpha_1^\vee,\ldots,\alpha_n^\vee\}
\end{equation}
(recall 
that, by our convention that long roots in $R_0$ have squared length $2$,
we have $\varphi^\vee=\varphi$). Let $R^{\pm}$ be the associated sets of
positive/negative roots. The corresponding simple
reflections are denoted by 
$S:=\{s_j:=s_{a_j}\}_{j=0}^n$ (since for $1\leq i\leq n$,
$s_i|_{V_0}\in W_0$ is the simple reflection associated to the basis element
$\alpha_i\in R_0$ as defined in the previous subsection, 
there is no conflict in notation). 
Note that $s_0=s_{\varphi}\tau(-\varphi^\vee)$.
Then $(W^a,S)$ is a Coxeter group with associated set of
simple reflections $S$. The defining relations of $W^a$ in terms
of $S$ are
\begin{equation}\label{braidrelations}
\begin{split}
s_is_js_i\cdots&=s_js_is_j\cdots\qquad (m_{ij} \textup{ terms on both sides}),\\
s_{i}^2&=1
\end{split}
\end{equation}
for $0\leq i,j\leq n$ with, for the first identity, the additional
requirements that $i\not=j$ and that $s_is_j\in W^a$
has finite order $m_{ij}$.

The length $l(w)$ of $w\in W$ with respect to the choice $\Delta$
of positive roots of $R$ is defined by
\begin{equation}\label{length}
l(w):=\#\bigl(R^+\cap w^{-1}R^-),\qquad w\in W.
\end{equation}
Its restriction to $W_0$ coincides with the length function of $W_0$
as considered in the previous subsection.
The subset $\Omega:=\{w\in W \,\, | \,\, l(w)=0\}$ is a subgroup of $W$, 
isomorphic to $P^\vee/Q^\vee$, and $W\simeq \Omega\ltimes W^a$.
In fact, $\Omega$ permutes the simple roots of $R$.
Hence an element $\omega\in\Omega$ gives rise to
a permutation of the index set $\{0,\ldots,n\}$ of the simple
reflections of $R$, which we again
denote by $\omega$. Consequently
$\omega(a_i)=a_{\omega(i)}$ and $\omega s_i\omega^{-1}=s_{\omega(i)}$
for $0\leq i\leq n$ and $\omega\in\Omega$.

\subsection{The affine Hecke algebra}
The results on the affine Hecke algebra which we recall in this
section to fix notations, are well known. We match the notations
to \cite[Chpt. 4]{M} as much as possible.

A multiplicity function on $R$ is a map $k: R\rightarrow \mathbb{C}^\times:=
\mathbb{C}\setminus\{0\}$, denoted by $a\mapsto k_a$ ($a\in R$),
which satisfies $k_{w(a)}=k_a$ for all $w\in W$ and $a\in R$. {}From now on, 
$k$ will always stand for a multiplicity function on $R$.
We write $k_j:=k_{a_j}$ for $0\leq j\leq n$. We emphasize that we are {\it not}
assuming $k_j$ to be generic, in particular, it may be a root of unity.

Note that
$k_a=k_{Da}$, hence it is uniquely
determined by its restriction $k|_{R_0^\vee}$ to a multiplicity function
of the underlying finite root system $R_0^\vee$.
In particular, the value $k_a$ only depends on the seminorm
$\|a\|$ of $a\in R$, hence $k$ takes on at most two different values.

The affine Hecke algebra $H^a(k)$ is the unital, complex associative algebra
generated by $T_i$ ($0\leq i\leq n$) with defining relations
\begin{equation}\label{braidrelations2}
\begin{split}
&T_iT_jT_i\cdots=T_jT_iT_j\cdots\qquad (m_{ij} \textup{ terms on both sides}),\\
&(T_{i}-k_{i})(T_{i}+k_{i}^{-1})=0
\end{split}
\end{equation}
for $0\leq i,j\leq n$ with, for the first identity,
the additional requirements that $i\not=j$ and that $s_is_j\in W^a$
has finite order $m_{ij}$.

Let $w\in W^a$ and fix a reduced expression
$w=s_{i_1}s_{i_2}\cdots s_{i_{l(w)}}$ ($0\leq i_j\leq n$).
Then $T_w:=T_{i_1}T_{i_2}\cdots T_{i_{l(w)}}\in H(k)$ is well defined
(it is by definition the unit of $H(k)$ if $w=e$ is the
identity element of $W$).
The $T_w$ ($w\in W^a$) form a $\mathbb{C}$-linear basis of $H(k)$.
The subalgebra $H_0(k)$ of $H^a(k)$ generated by $T_i$ ($1\leq i\leq n$)
is called the finite Hecke algebra. The $T_w$ ($w\in W_0$) form a 
$\mathbb{C}$-linear basis
of $H_0(k)$.

The finite abelian subgroup $\Omega$ of $W$ acts by algebra automorphisms
on $H^a(k)$ by $\omega(T_j)=T_{\omega(j)}$ for all $0\leq j\leq n$
($\omega\in \Omega$). The {\it extended} affine Hecke algebra is the
corresponding smashed product algebra $H(k):=H^a(k)\#\Omega$.
Recall that, as a complex vector space, $H(k)\simeq H^a(k)\otimes_{\mathbb{C}}
\mathbb{C}[\Omega]$ with $\mathbb{C}[\Omega]$ the complex
group algebra of $\Omega$. The algebra structure of $H(k)$ is 
then characterized by $(h\otimes\omega)(h^\prime\otimes\omega^\prime)=
h\omega(h^\prime)\otimes\omega\omega^\prime$ for $h,h^\prime\in H^a(k)$
and $\omega,\omega^\prime\in\Omega$.

A reduced expression of an extended affine Weyl group element
$w\in W$ is an expression of $w$ of the
form $w=s_{i_1}s_{i_2}\cdots s_{i_{l(w)}}\omega$
for some $0\leq i_j\leq n$ and for some $\omega\in\Omega$.
Then $T_w:=T_{i_1}T_{i_2}\cdots T_{i_{l(w)}}\otimes\omega\in H(k)$
is well defined, it reduces to the previous definition of $T_w$
in case $w\in W^a$, and $\{T_w\}_{w\in W}$ is a $\mathbb{C}$-linear
basis of $H(k)$.

Consider the complex torus
\begin{equation}\label{torus}
T:=\textup{Hom}_{\mathbb{Z}}(P^\vee,\mathbb{C}^\times)
\end{equation}
of group homomorphisms $P^\vee\rightarrow \mathbb{C}^\times$.
We write $t^\lambda\in\mathbb{C}^\times$ ($t\in T$, $\lambda\in P^\vee$)
for the evaluation of $t$ at $\lambda$.
For $z\in \mathbb{C}^\times$ and $\alpha\in Q$ with $Q\subset V_0$
the root lattice of $R_0$, we write $z^\alpha\in T$ for the group
homomorphism $P^\vee\ni\lambda\mapsto
z^{\langle\lambda,\alpha\rangle}$. 

The Weyl group $W_0$ acts on $P^\vee$, By transposition,
it also acts on $T$. Note that for $\alpha\in R_0$ and
$t\in T$,
\begin{equation}\label{saction}
s_\alpha t=z^{\alpha}t\quad \textup{with }\,\,
z=t^{-\alpha^\vee}\in\mathbb{C}^\times.
\end{equation}

Let $\mathbb{C}[T]$ be the
algebra of regular functions on $T$. We write 
$e^\lambda$ ($\lambda\in P^\vee$) for the canonical $\mathbb{C}$-basis
of $\mathbb{C}[T]$, where $e^\lambda$ stands for the regular function 
$T\ni t\mapsto t^\lambda$.
Note that $\mathbb{C}[T]$ is isomorphic to the group
algebra $\mathbb{C}[P^\vee]$ of the coweight lattice $P^\vee$.
 
The action of
$W_0$ on $\mathbb{C}[T]$, contragredient to the $W_0$-action
on $T$, satisfies $w(e^\lambda):=e^{w\lambda}$
for $w\in W_0$ and $\lambda\in P^\vee$. It is an action 
by algebra automorphisms on $\mathbb{C}[T]$, 
hence it extends uniquely to an action by field
automorphisms on the quotient field $\mathbb{C}(T)$ of $\mathbb{C}[T]$.
In addition, by transposition of the $W_0$-action on $T$, the finite
Weyl group $W_0$ acts on
the algebra $\mathcal{O}(T)$ of analytic functions on $T$
(respectively the field $\mathcal{M}(T)$ of meromorphic functions
on $T$) by algebra (respectively field) automorphisms.
It results in the following
inclusion of $W_0$-module algebras,
\[
\mathbb{C}[T]\subset\mathbb{C}(T)\subset \mathcal{O}(T)\subset
\mathcal{M}(T).
\]

For all $\alpha^\vee\in R_0^\vee$ define 
$c_{\alpha^\vee}^k\in\mathbb{C}(T)$ by
\begin{equation}\label{calpha}
c_{\alpha^\vee}^k(t):=\frac{k_{\alpha^\vee}^{-1}-
k_{\alpha^\vee}t^{\alpha^\vee}}{1-t^{\alpha^\vee}}.
\end{equation}
We furthermore write $c_i^k=c_{\alpha_i^\vee}^k$
for $1\leq i\leq n$.

Let $P_+^\vee=\{\lambda\in P^\vee \,\, | \,\,
\langle \lambda,\alpha\rangle\in\mathbb{Z}_{\geq 0}\}$ be the
cone of dominant coweights and set
\[Y^\lambda:=T_{\tau(\lambda)}\in H(k).
\]
Then $P_+^\vee\ni\lambda\mapsto Y^\lambda\in H(k)^\times$ is a
morphism of semigroups. It has a unique extension to a group
homomorphism $P^\vee\rightarrow H(k)^\times$, also denoted by
$\lambda\mapsto Y^\lambda$ ($\lambda\in P^\vee$).
Denote by $\mathcal{A}_Y^k$ the commutative subalgebra of 
$H(k)$ generated by the $Y^\lambda$ ($\lambda\in P^\vee$).
The following theorem is due to Bernstein and Zelevinsky
(for a proof, see \cite{L} or \cite{M}).
\begin{thm}\label{Hchar}
{\bf (i)} The surjective algebra map $\mathbb{C}[T]\rightarrow
\mathcal{A}_Y^k$
mapping $e^\lambda$ to $Y^\lambda$ for $\lambda\in P^\vee$,
is an isomorphism. We write $f(Y)$ for the element in $\mathcal{A}_Y^k$ 
corresponding to $f\in\mathbb{C}[T]$ under this isomorphism.\\
{\bf (ii)} For all $f\in\mathbb{C}[T]$ and $1\leq i\leq n$,
\begin{equation}\label{crossrelation}
f(Y)T_i=T_i(s_if)(Y)+(c_i^k(Y^{-1})-k_i)((s_if)(Y)-f(Y)),
\end{equation}
where $(c_i^k(Y^{-1})-k_i)((s_if)(Y)-f(Y))$ is the element of $\mathcal{A}_Y^k$
corresponding, under the isomorphism of {\bf (i)}, to the regular function
$(c_i^k(t^{-1})-k_i)((s_if)(t)-f(t))$ in $t\in T$.\\
{\bf (iii)} The multiplication map defines an isomorphism
$\mathcal{A}_Y^k\otimes_{\mathbb{C}}H_0(k)\overset{\sim}{\longrightarrow}
H(k)$ of vector spaces.\\ 
{\bf (iv)} 
The cross relations \eqref{crossrelation}
characterize the algebraic structure of $H(k)$
in terms of the algebras $\mathcal{A}_Y^k$ and $H_0(k)$.
\end{thm}

Recall that $I$ is a fixed subset of $\{1,\ldots,n\}$.
Write $H_I(k)$ for the unital subalgebra of $H(k)$ generated
by $\mathcal{A}_Y^k$ and the $T_i$ ($i\in I$). The subalgebra
$H_{0,I}(k)$ generated by the $T_i$ ($i\in I$) has 
as complex linear basis $\{T_w\}_{w\in W_{0,I}}$. 
Then $H_I(k)\simeq \mathbb{C}_Y[T]\otimes_{\mathbb{C}}H_{0,I}(k)$
as vector spaces (by the multiplication map). Theorem \ref{Hchar}
holds true for $H_I(k)$ with the role of $H_0(k)$ replaced by $H_{0,I}(k)$.
We call $H_{0,I}(k)$ and $H_I(k)$ standard parabolic subalgebras
of $H_0(k)$ and $H(k)$, respectively.
Note that $H_0(k)$ (respectively $H(k)$) is a
free right $H_{0,I}(k)$-module (respectively $H_I(k)$-module) with
basis $\{T_w\}_{w\in W_0^I}$. 

We write $H^a_I(k)$ for the unital complex subalgebra of $H^a(k)$ generated by
$Y^{\lambda}$ ($\lambda\in Q^\vee$) and $T_i$ ($i\in I$). The following
technical lemma will be convenient at a later stage.
\begin{lem}\label{generateI}
The standard parabolic algebra $H^a_I(k)$ is algebraically generated
by $T_i$ ($i\in I$) and $Y^{\pm w^{-1}(\varphi^\vee)}$ ($w\in \widetilde{W}_0^I$),
where $\widetilde{W}_0^I$ is a complete set of representatives of the coset
space $W_0/W_{0,I}$.
\end{lem}
\begin{proof}
It is clear that $H^a_I(k)$ is algebraically generated by
$T_i$ ($i\in I$) and $Y^{\pm w^{-1}(\varphi^\vee)}$ ($w\in W_0$).

For $\alpha\in R_0$ we write $\sigma(\alpha)=1$ if $\alpha\in R_0^+$
and $\sigma(\alpha)=-1$ if $\alpha\in R_0^-$. Then 
\begin{equation}\label{YT0}
Y^{w^{-1}(\varphi^\vee)}=T_w^{-1}T_0^{\sigma(w^{-1}\varphi)}T_{s_{\varphi}w}
\end{equation}
and
\begin{equation}\label{TWI}
T_{ws_i}=T_wT_i^{\sigma(w\alpha_i)}
\end{equation}
in $H^a(k)$ for all $w\in W_0$ and $i\in\{1,\ldots,n\}$, see
\cite[(3.3.6)]{M} and \cite[(3.1.7)]{M}. 

Let $i\in I$ and $w\in W_0$. If $\sigma(w^{-1}\varphi)\not=
\sigma(s_iw^{-1}\varphi)$ then $w^{-1}\varphi=\alpha_i$ or $=-\alpha_i$,
hence 
\[Y^{s_iw^{-1}(\varphi^\vee)}=Y^{-w^{-1}(\varphi)}.
\]
If $\sigma(w^{-1}\varphi)=\sigma(s_iw^{-1}\varphi)$ then 
\begin{equation*}
\begin{split}
Y^{s_iw^{-1}(\varphi^\vee)}&=
T_i^{-\sigma(w\alpha_i)}\bigl(T_w^{-1}T_0^{\sigma(w^{-1}\varphi)}
T_{s_{\varphi}w}\bigr)T_i^{\sigma(s_{\varphi}w\alpha_i)}\\
&=T_i^{-\sigma(w\alpha_i)}Y^{w^{-1}(\varphi^\vee)}T_i^{\sigma(s_{\varphi}w\alpha_i)}.
\end{split}
\end{equation*}
This shows that 
the $Y^{\pm w^{-1}(\varphi^\vee)}$ with $w\in \widetilde{W}_0^I$, 
together with the $T_i$ ($i\in I$), already form
algebraic generators of $H^a_I(k)$.
\end{proof}

\subsection{The affine Hecke algebra modules}

We have two characters (algebra maps) $\epsilon_{\pm}^k: H(k)\rightarrow
\mathbb{C}$, characterized by $\epsilon_\pm^k(T_j)=
\pm k_j^{\pm 1}$ 
($0\leq j\leq n$) and $\epsilon_{\pm}^k(T_{\omega})=1$ ($\omega\in\Omega$). 
The character $\epsilon^k_+$ (respectively $\epsilon^k_-$) is called
the trivial (respectively Steinberg) character of $H(k)$. {}From e.g.
\cite[\S 2.4]{M},
\begin{equation}\label{sstar}
\epsilon_{\pm}^k(T_{\tau(\lambda)})=\prod_{\alpha\in R_0^+}
\bigl(\pm k_{\alpha^\vee}\bigr)^{\pm\langle\lambda,\alpha\rangle},
\qquad \lambda\in P_+^\vee.
\end{equation}
Set
\begin{equation}\label{deltapm}
\delta_{\pm}^k:=
\prod_{\alpha\in R_0^+}\bigl(\pm k_{\alpha^\vee}\bigr)^{\pm\alpha}\in T,
\end{equation}
i.e. it is the element of $T=\textup{Hom}_{\mathbb{Z}}(P^\vee,\mathbb{C}^\times)$ 
mapping the coweight
$\lambda\in P^\vee$ to 
$\prod_{\alpha\in R_0^+}\bigl(\pm k_{\alpha^\vee}\bigr)^{\pm
\langle\lambda,\alpha\rangle}$. Then it follows from \eqref{sstar} that
\begin{equation}\label{trivY}
\epsilon_{\pm}^k(f(Y))=f\bigl(\delta_{\pm}^k\bigr),
\qquad f\in \mathbb{C}[T].
\end{equation}

More generally, we will consider $H(k)$-modules induced from
characters of a standard parabolic subalgebra $H_I(k)$. The characters
will be parametrized by elements of
\begin{equation}\label{TI}
T_I^k:=\{\gamma\in T \,\, | \,\, \gamma^{\alpha_i^\vee}=k_i^2\quad 
\forall\, i\in I\}.
\end{equation}
Note that $\delta_{\pm}^k\in T_{[1,n]}^{k^{\pm 1}}$, where
$[1,n]:=\{1,\ldots,n\}$ and $k^{-1}$ is the multiplicity function
$R\ni a\mapsto k_a^{-1}$.

Observe that
\begin{equation}\label{W0Iact}
w\gamma=\gamma \prod_{\alpha\in R_0^{I,+}\cap wR_0^{I,-}}k_{\alpha^\vee}^{-2\alpha}
\qquad (\gamma\in T_I^k,\,\, w\in W_{0,I}),
\end{equation}
hence in particular
\begin{equation}\label{W0Iactcons}
\underline{w_0}\gamma=\rho_I^k\gamma\qquad (\gamma\in T_I^k)
\end{equation}
where
\begin{equation}\label{rhoI}
\rho_I^k:=\prod_{\alpha\in R_0^{I,+}}k_{\alpha^\vee}^{-2\alpha}\in T.
\end{equation}

\begin{lem}\label{chilemma}
Let 
$\gamma\in T_I^{k^{\pm 1}}$.\\
{\bf (i)}
There exists a unique character $\chi_{\gamma}^{k,\pm,I}: 
H_I(k)\rightarrow\mathbb{C}$
satisfying $\chi_\gamma^{k,\pm,I}(f(Y))=f(\gamma)$ and 
$\chi_\gamma^{k,\pm,I}(T_i)=\pm k_i^{\pm 1}$
for $f\in\mathbb{C}[T]$ and $i\in I$.\\
{\bf (ii)} The left $H(k)$-module
\[
M^{k,\pm,I}(\gamma):=\textup{Ind}_{H_I(k)}^{H(k)}\bigl(\chi_\gamma^{k,\pm,I}\bigr)
\]
has complex linear basis
\[
v_w^{k,\pm,I}(\gamma):=T_w\otimes_{H_I(k),\chi_\gamma^{k,\pm,I}}1,\qquad w\in W_0^I.
\]
\end{lem}
\begin{proof}
It suffices to show that
$\chi_\gamma^{k,\pm,I}$ preserves the cross relation \eqref{crossrelation}
for $f\in\mathbb{C}[T]$ and $i\in I$. The cross relation is
explicitly given by
\begin{equation}\label{crossexp}
f(Y)T_i=T_i(s_if)(Y)+(k_i^{-1}-k_i)\left(\frac{(s_if)(Y)-f(Y)}
{1-Y^{-\alpha_i^\vee}}\right).
\end{equation}
Fix $i\in I$ and $f\in\mathbb{C}[T]$.
If $k_i^2=1$ then \eqref{crossexp} reduces to $f(Y)T_i=T_i(s_if)(Y)$,
which is indeed respected by $\chi_\gamma^{k,\pm,I}$ since 
$s_i\gamma=\gamma$ for $\gamma\in T_I^{k^{\pm 1}}$ by \eqref{W0Iact}.
If $k_i^2\not=1$, then \eqref{crossexp} is respected by 
$\chi_{\gamma}^{k,\pm,I}$ since $\gamma^{\alpha_i^\vee}=k_i^{\pm 2}$ for
$\gamma\in T_I^{k^{\pm 1}}$ and
\[\pm k_i^{\pm 1}f(\gamma)=\pm k_i^{\pm 1}f(s_i\gamma)+(k_i^{-1}-k_i)
\left(\frac{f(s_i\gamma)-f(\gamma)}{1-k_i^{\mp 2}}\right).
\]  
\end{proof}
The modules $M^{k}(\gamma):=M^{k,\pm,\emptyset}(\gamma)$ ($\gamma\in T$) 
are called the principal series modules of $H(k)$.
We write $v_w^k(\gamma):=v_w^{k,\pm,\emptyset}(\gamma)$ ($w\in W_0$) 
for the corresponding standard basis elements.

\begin{eg}
In view of \eqref{trivY},
$M^{k,\pm,[1,n]}\bigl(\delta_{\pm}^k\bigr)$
is the one-dimensional $H(k)$-module characterized by the
algebra map $\epsilon_{\pm}^k: H(k)\rightarrow \mathbb{C}$.
\end{eg}

By Lemma \ref{ABC}, 
the action of the finite Hecke algebra $H_0(k)$ on the standard
basis $\{v_w^{k,\pm,I}(\gamma)\}_{w\in W_0^I}$ of
$M^{k,\pm,I}(\gamma)$ is given by
\begin{equation}\label{vaction}
T_iv_w^{k,\pm,I}(\gamma)=
\begin{cases}
(k_i-k_i^{-1})v_w^{k,\pm,I}(\gamma)+v_{s_iw}^{k,\pm,I}(\gamma)
,\quad &\hbox{ if } w\in A_i,\\
v_{s_iw}^{k,\pm,I}(\gamma), 
&\hbox{ if } w\in B_i,\\
\pm k_i^{\pm 1}v_{w}^{k,\pm,I}(\gamma),
&\hbox{ if } w\in C_i.
\end{cases}
\end{equation}
Furthermore, $f(Y)v_e^{k,\pm,I}(\gamma)=f(\gamma)v_e^{k,\pm,I}(\gamma)$
for all $f\in\mathbb{C}[T]$.
\subsection{Intertwiners}\label{Intersection}

By a well known result of Bernstein, the center $Z(H(k))$ of the affine
Hecke algebra equals $\mathcal{A}_Y^{k,W_0}$ (the subalgebra
of $H(k)$ consisting of elements $f(Y)$ with $f\in \mathbb{C}[T]^{W_0}$).

Let $M$ be a finite dimensional left $H(k)$-module and $\gamma\in T$.
We write
\[M^{W_0\gamma}:=\{m\in M \,\, | \,\, f(Y)m=f(\gamma)m\quad \forall\, 
f\in\mathbb{C}[T]^{W_0}\},
\]
which is a $H(k)$-submodule of $M$.  We furthermore write
\[M_\gamma:=\{m\in M \,\, | \,\, f(Y)m=f(\gamma)m\quad \forall\,
f\in\mathbb{C}[T]\}.
\]
\begin{defi}
Let $M$ be a finite dimensional left $H(k)$-module.\\
{\bf (i)} $M$ is said to have central character
$W_0\gamma\in T/W_0$ if $M=M^{W_0\gamma}$.\\
{\bf (ii)} We say that $M$ is calibrated if
$M=\bigoplus_{\gamma\in T}M_\gamma$.
\end{defi}
Note that the $H(k)$-module $M^{k,\pm,I}(\gamma)$ ($\gamma\in T_I^{k^{\pm 1}}$)
has central character $W_0\gamma$. 

We now determine the conditions on $\gamma\in T_I^{k^{\pm 1}}$ 
that ensure that $M^{k,\pm,I}(\gamma)$ is calibrated
using the intertwiners of $H(k)$. 
In the following theorem we collect the definitions
and the basic properties of the intertwiners (cf., e.g., \cite{Ma,Kat},
\cite[\S 2.2]{Op}).
\begin{thm}\label{intertwinerthm}
For $1\leq i\leq n$ set
\[
I_i(k):=T_i(1-Y^{\alpha_i^\vee})+(k_i-k_i^{-1})Y^{\alpha_i^\vee}
\in H(k).
\]
There exists unique elements $I_w(k)\in H(k)$ ($w\in W_0$)
satisfying
\[
I_w(k)=I_{i_1}(k)I_{i_2}(k)\cdots I_{i_r}(k)
\]
if $w=s_{i_1}s_{i_2}\cdots s_{i_r}\in W_0$ is a reduced expression
($1\leq i_j\leq n$). Furthermore,
\[
I_i(k)^2=(k_i-k_i^{-1}Y^{\alpha_i^\vee})
(k_i-k_i^{-1}Y^{-\alpha_i^\vee})
\] 
for $1\leq i\leq n$ and in $H(k)$,
\[
I_w(k)f(Y)=(wf)(Y)I_w(k)\qquad (w\in W_0,\, f\in\mathbb{C}[T]).
\]
\end{thm}
Using \eqref{crossrelation}, we have the alternative expression
\begin{equation}\label{Iialt}
I_i(k)=(1-Y^{-\alpha_i^\vee})T_i+k_i^{-1}-k_i,\qquad 1\leq i\leq n
\end{equation}
for the intertwiners of the affine Hecke algebra $H(k)$.
\begin{cor}\label{productY}
For $w\in W_0$ we have
\[I_w(k)I_{w^{-1}}(k)=d_w(Y)(wd_{w^{-1}})(Y)
\]
in $H(k)$, with $d_w\in\mathbb{C}[T]$ given by
\[
d_w(t)=\prod_{\alpha\in R_0^+\cap wR_0^-}(k_{\alpha^\vee}-
k_{\alpha^\vee}^{-1}t^{\alpha^\vee}).
\]
\end{cor}
For $\gamma\in T_I^{k^{\pm 1}}$ and $w\in W_0^I$ we set
\begin{equation}\label{bbasis}
b_w^{k,\pm,I}(\gamma):=I_w(k)v_e^{k,\pm,I}(\gamma)\in M^{k,\pm,I}(\gamma)_{w\gamma}.
\end{equation}
\begin{rema}
If $w\in W_0\setminus W_0^I$ and $\gamma\in T_I^{k^{\pm 1}}$
then $I_w(k)v_e^{k,\pm,I}(\gamma)=0$.
To prove this it suffices to note that $I_i(k)v_e^{k,\pm,I}(\gamma)=0$
for $i\in I$, which is immediate from the definition of $I_i(k)$
and the fact that $\gamma^{\alpha_i^\vee}=k_i^{\pm 2}$.
\end{rema}

\begin{lem}\label{expb}
Let $\gamma\in T_I^{k^{\pm 1}}$. For $w\in W_0^I$ we have
\[
b_w^{k,\pm,I}(\gamma)=\bigl(\prod_{\alpha\in (R_0^+\setminus R_0^{I,+})\cap w^{-1}R_0^-}
(1-\gamma^{\alpha^\vee})\bigr)v_w^{k,\pm,I}(\gamma)+
\sum_{u\in W_0^I: u<w} a_u^{\pm}v_u^{k,\pm,I}(\gamma)
\]
for some $a_u^{\pm}\in\mathbb{C}$.
\end{lem}
\begin{proof}
By induction on $l(w)$ we have for $w\in W_0$,
\[
I_w(k)=T_w\prod_{\alpha\in R_0^+\cap w^{-1}R_0^-}(1-Y^{\alpha^\vee})
+\sum_{u\in W_0: u<w}T_uf_u(Y)
\]
in $H(k)$, for some $f_u\in\mathbb{C}[T]$.
Thus, for $w\in W_0^I$,
\[
b_w^{k,\pm,I}(\gamma)=\bigl(\prod_{\alpha\in R_0^+\cap w^{-1}R_0^-}
(1-\gamma^{\alpha^\vee})\bigr)v_w^{k,\pm,I}(\gamma)
+\sum_{u\in W_0: u<w}f_u(\gamma)T_uv_e^{k,\pm,I}(\gamma).
\]
Since $w(R_0^{I,+})\subseteq R_0^+$ for $w\in W_0^I$, the product over
$\alpha$ is in fact 
a product over the set $(R_0^+\setminus R_0^{I,+})\cap w^{-1}(R_0^-)$.
Furthermore, if $u\in W_0$, $w\in W_0^I$ and $u<w$, then 
$\overline{u}\leq u=\overline{u}\underline{u}<w$
and 
\[
T_uv_e^{k,\pm,I}(\gamma)=T_{\overline{u}}T_{\underline{u}}v_e^{k,\pm,I}(\gamma)=
\epsilon_{\pm}^k(T_{\underline{u}})v_{\overline{u}}^{k,\pm,I}(\gamma).
\]
This completes the proof.
\end{proof}
\begin{prop}\label{cal}
If $\gamma\in T_I^{k^{\pm 1}}$ satisfies $\gamma^{\alpha^\vee}\not=1$
for all $\alpha\in R_0^+\setminus R_0^{I,+}$, then
\begin{enumerate}
\item[{\bf (i)}] $w\gamma=w^\prime\gamma$ for $w,w^\prime\in W_0^I$
if and only if $w=w^\prime$.
\item[{\bf (ii)}] $M^{k,\pm,I}(\gamma)$ is calibrated.
\item[{\bf (iii)}] $M^{k,\pm,I}(\gamma)=\bigoplus_{w\in W_0^I}
M^{k,\pm,I}(\gamma)_{w\gamma}$ and
$\textup{Dim}_{\mathbb{C}}\bigl(M^{k,\pm,I}(\gamma)_{w\gamma}\bigr)=1$
for all $w\in W_0^I$.
\item[{\bf (iv)}] $M^{k,\pm,I}(\gamma)_{w\gamma}=\mathbb{C}b_w^{k,\pm,I}(\gamma)$
for all $w\in W_0^I$.
\end{enumerate}
\end{prop}
\begin{proof}
{\bf (i)} The fixpoint subgroup
\[
W_{0,\gamma}:=\{w\in W_0 \,\, | \,\, w\gamma=\gamma\}
\] 
of $\gamma$
is generated by the reflections $s_\alpha$ ($\alpha\in R_0^+$)
it contains (see \cite{St}). 
For $\alpha\in R_0^+$ we have $s_\alpha\in W_{0,\gamma}$
if and only if $\gamma^{\alpha^\vee}=1$. By the assumption that
$\gamma^{\alpha^\vee}\not=1$ for $\alpha\in R_0^+\setminus R_0^{I,+}$,
we conclude that $s_\alpha\in W_{0,\gamma}$
($\alpha\in R_0^+$) implies that $\alpha\in R_0^{I,+}$. Hence
$W_{0,\gamma}\subset W_{0,I}$. The result now follows immediately.\\
{\bf (ii)} follows from {\bf (iii)}.\\
{\bf (iii)}\&{\bf (iv)} 
The previous
lemma shows that $b_w^{k,\pm,I}(\gamma)\in M^{k,\pm,I}(\gamma)_{w\gamma}$
is nonzero for all $w\in W_0^I$. By a dimension count and {\bf (i)}
we get 
$M^{k,\pm,I}(\gamma)=\bigoplus_{w\in W_0^I}
M^{k,\pm,I}(\gamma)_{w\gamma}$ and
$M^{k,\pm,I}(\gamma)_{w\gamma}=\mathbb{C}b_w^{k,\pm,I}(\gamma)$ for all
$w\in W_0^I$.
\end{proof}

\begin{lem}\label{baction}
Let $\gamma\in T_I^{k^{\pm 1}}$ such that $\gamma^{\alpha^\vee}\not=1$
for all $\alpha\in R_0^+\setminus R_0^{I,+}$. Let $1\leq i\leq n$.
If $w\in A_i$ then
\begin{equation}\label{Aformula}
\begin{split}
T_ib_w^{k,\pm,I}(\gamma)&=
\frac{(k_i-k_i^{-1}\gamma^{w^{-1}(\alpha_i)^\vee})
(k_i-k_i^{-1}\gamma^{-w^{-1}(\alpha_i)^\vee})}
{(1-\gamma^{w^{-1}(\alpha_i)^\vee})}b_{s_iw}^{k,\pm,I}(\gamma)\\
&+\frac{(k_i^{-1}-k_i)\gamma^{w^{-1}(\alpha_i)^\vee}}
{(1-\gamma^{w^{-1}(\alpha_i)^\vee})}b_w^{k,\pm,I}(\gamma).
\end{split}
\end{equation}
If $w\in B_i$ then
\begin{equation}\label{Bformula}
T_ib_w^{k,\pm,I}(\gamma)=
\frac{1}{(1-\gamma^{w^{-1}(\alpha_i)^\vee})}b_{s_iw}^{k,\pm,I}(\gamma)\\
+\frac{(k_i^{-1}-k_i)\gamma^{w^{-1}(\alpha_i)^\vee}}
{(1-\gamma^{w^{-1}(\alpha_i)^\vee})}b_w^{k,\pm,I}(\gamma).
\end{equation}
If $w\in C_i$ then $T_ib_w^{k,\pm,I}(\gamma)=\pm k_i^{\pm 1}b_w^{k,\pm,I}(\gamma)$.
\end{lem}
\begin{proof}
If $w\in A_i$ then $l(s_iw)=l(w)-1$ hence $w^{-1}(\alpha_i)\in R_0^-$.
Furthermore, $w^{-1}(\alpha_i)\not\in R_0^{I,-}$ since $w\in W_0^I$.
Consequently $\gamma^{w^{-1}(\alpha_i)^\vee}\not=1$, hence \eqref{Aformula}
makes sense.

Similarly, if $w\in B_i$ then $w^{-1}(\alpha_i)\in R_0^+$ since
$l(s_iw)=l(w)+1$ and $w^{-1}(\alpha_i)\not\in R_0^{I,+}$ since
$s_iw\in W_0^I$. Consequently $\gamma^{w^{-1}(\alpha_i)^\vee}\not=1$, 
and \eqref{Bformula} makes sense.

By Theorem \ref{intertwinerthm} we have for $w\in W_0^I$,
\begin{equation}\label{tooo}
\begin{split}
(1-\gamma^{w^{-1}(\alpha_i)^\vee})T_ib_w^{k,\pm,I}(\gamma)&=
T_i(1-Y^{\alpha_i^\vee})I_w(k)v_e^{k,\pm,I}(\gamma)\\
&=(I_i(k)I_w(k)+(k_i^{-1}-k_i)Y^{\alpha_i^\vee}I_w(k))v_e^{k,\pm,I}(\gamma).
\end{split}
\end{equation}
If $w\in A_i$ then 
$I_i(k)I_w(k)=I_i(k)^2I_{s_iw}(k)=d_{s_i}(Y)(s_id_{s_i})(Y)I_{s_iw}(k)$
and \eqref{Aformula} follows from \eqref{tooo}, 
Theorem \ref{intertwiner}, and the fact that $s_iw\in W_0^I$.
If $w\in B_i$ then $I_i(k)I_w(k)=I_{s_iw}(k)$ and $s_iw\in W_0^I$, 
hence \eqref{Bformula} follows from \eqref{tooo}.
It remains to show that $T_ib_w^{k,\pm,I}(\gamma)=
\pm k_i^{\pm 1}b_w^{k,\pm,I}(\gamma)$
if $w\in C_i$. 

Fix $w\in C_i$. Then $w^{-1}(\alpha_i)=\alpha_{i_w}$ 
for a unique $i_w\in I$ (cf. Lemma \ref{ABC}). In particular,
$\gamma^{w^{-1}(\alpha_i)^\vee}=
\gamma^{\alpha_{i_w}^\vee}=k_{i_w}^{\pm 2}=k_i^{\pm 2}$.

Since $l(s_iw)=l(w)+1$ it follows from \eqref{tooo} that
\[
(1-\gamma^{w^{-1}(\alpha_i)^\vee})
T_ib_w^{k,\pm,I}(\gamma)=I_{s_iw}(k)v_e^{k,\pm,I}(\gamma)+
(k_i^{-1}-k_i)\gamma^{w^{-1}(\alpha_i)^\vee}b_w^{k,\pm,I}(\gamma).
\]
Hence, using $\gamma^{w^{-1}(\alpha_i)^\vee}=k_i^{\pm 2}$,
$I_{s_iw}(k)=I_w(k)I_{i_w}(k)$ and $I_j(k)v_e^{k,\pm,I}(\gamma)=0$
for $j\in I$,
\[
(1-k_i^{\pm 2})T_ib_w^{k,\pm,I}(\gamma)=
(k_i^{-1}-k_i)k_i^{\pm 2}b_w^{k,\pm,I}(\gamma).
\]
Consequently $T_ib_w^{k,\pm,I}(\gamma)=\pm k_i^{\pm 1}b_w^{k,\pm,I}(\gamma)$ 
if $k_i^2\not=1$.

It remains to consider the case that $w\in C_i$ and that $k_i^2=1$.
Then $f(Y)T_i=T_i(s_if)(Y)$ in $H(k)$ for all $f\in\mathbb{C}[T]$ 
by \eqref{crossexp}
and $s_{i_w}\gamma=\gamma$ since $\gamma\in T_I^{k^{\pm 1}}$, hence 
\[f(Y)T_ib_w^{k,\pm,I}(\gamma)=T_if(s_iw\gamma)b_w^{k,\pm,I}(\gamma)=
f(ws_{i_w}\gamma)T_ib_w^{k,\pm,I}(\gamma)=
f(w\gamma)T_ib_w^{k,\pm,I}(\gamma)
\]
for all $f\in\mathbb{C}[T]$.
This shows that $T_ib_w^{k,\pm,I}(\gamma)=c_i^{\pm}b_w^{k,\pm,I}(\gamma)$ for some
$c_i^{\pm}\in\mathbb{C}$. 
Note that
\[T_iv_w^{k,\pm,I}(\gamma)=T_{s_iw}v_e^{k,\pm,I}(\gamma)=
T_wT_{i_w}v_e^{k,\pm,I}(\gamma)=\pm k_i^{\pm 1}v_w^{k,\pm,I}(\gamma),
\]
hence by Lemma \ref{expb},
\[
T_ib_w^{k,\pm,I}(\gamma)=
\pm k_i^{\pm 1}\lambda_w
v_w^{k,\pm,I}(\gamma)+\sum_{u\in W_0^I: u<w}a_u^{\pm}T_iv_u^{k,\pm,I}(\gamma)
\]
for some $a_u^{\pm}\in\mathbb{C}$
and with $\lambda_w:=
\prod_{\alpha\in R_0^+\setminus R_0^{I,+}\cap w^{-1}R_0^-}
(1-\gamma^{\alpha^\vee})\not=0$. Let $u\in W_0^I$ with $u<w$.
Since $w\in C_i$, hence $s_iw\not\in W_0^I$, we have 
$\overline{s_iu}\not=w$.
Furthermore, $T_iv_u^{k,\pm,I}\in
\textup{span}_{\mathbb{C}}\{v_u^{k,\pm,I}(\gamma),
v_{\overline{s_iu}}^{k,\pm,I}(\gamma)\}$ by \eqref{vaction}. Hence
\[
T_ib_w^{k,\pm,I}(\gamma)=
\pm k_i^{\pm 1}\lambda_w
v_w^{k,\pm,I}(\gamma)+\sum_{u\in W_0^I\setminus\{w\}}a_u^{\prime\,\pm}
v_u^{k,\pm,I}(\gamma)
\]
for some $a_u^{\prime\,\pm}\in\mathbb{C}$.
On the other hand,
\begin{equation*}
\begin{split}
T_ib_w^{k,\pm,I}(\gamma)&=c_i^{\pm}b_w^{k,\pm,I}(\gamma)\\
&=
c_i^{\pm}\lambda_w
v_w^{k,\pm,I}(\gamma)+\sum_{u\in W_0^I: u<w}c_i^{\pm}a_u^{\pm}v_u^{k,\pm,I}(\gamma).
\end{split}
\end{equation*}
Comparing the coefficient of $v_w^{k,\pm,I}(\gamma)$
in these two expressions of $T_ib_w^{k,\pm,I}(\gamma)$, we
conclude that $c_i^{\pm}=\pm k_i^{\pm  1}$. 
Hence $T_ib_w^{k,\pm,I}(\gamma)=\pm k_i^{\pm 1}b_w^{k,\pm,I}(\gamma)$ for 
$w\in C_i$.
\end{proof}

\subsection{(Anti)spherical vectors}\label{spherical}
\begin{defi}
Let $M$ be a left $H(k)$-module. We call 
\[
M^{\pm}:=\{m\in M \,\, | \,\, hm=\epsilon_{\pm}^k(h)m
\quad \forall\, h\in H_0(k)\}
\]
the space of spherical ($+$), respectively antispherical ($-$),
elements in $M$.
\end{defi}
We also write for $M^{I,\pm}$ for the space of vectors $m\in M$
satisfying $hm=\epsilon_{\pm}^k(h)m$ for all $h\in H_{0,I}(k)$,
so that $M^{\pm}=M^{[1,n],\pm}$.
Define
\begin{equation}\label{almostidempotent}
C_{\pm}^I(k):=\sum_{w\in W_0^I}\epsilon_{\pm}^k(T_w)T_w\in H_0(k)
\end{equation}
and write $C_{\pm}(k)=C_{\pm}^\emptyset(k)$.
\begin{lem}\label{tttt}
Let $M$ be a left $H(k)$-module. Then the action of $C_{\pm}^I(k)$ on
$M^{I,\pm}$ defines a linear map
\[
C_{\pm}^I(k): M^{I,\pm}\rightarrow M^{\pm}.
\]
\end{lem}
\begin{proof}
This follows by a direct computation using Lemma \ref{ABC} 
and \eqref{vaction}.
\end{proof}
\begin{rema}
If $M$ is a left $H(k)$-module and $m\in M^{I,\pm}$
then 
\[C_\pm(k)m=P_I(k^{\pm 1})(C_{\pm}^I(k)m)
\] 
with
\begin{equation}\label{PI}
P_I(k):=\sum_{w\in W_{0,I}}\epsilon_+^k(T_w)^2
\end{equation}
the Poincar{\'e} polynomial (see \cite{Mac}) of $W_{0,I}$.
\end{rema}

Of special interest for the present paper
is the case that $M=M^{k,\pm,I}(\gamma)$ with 
$\gamma\in T_I^{k^{\pm 1}}$.
Then we have $v_e^{k,\pm,I}(\gamma)\in M^{k,\pm,I}(\gamma)^{I,\pm}$, 
hence we obtain
a spherical ($+$), respectively antispherical ($-$), vector
\begin{equation}\label{sphericalvector}
\begin{split}
1_\gamma^{k,\pm,I}:=&C_{\pm}^I(k)v_e^{k,\pm,I}(\gamma)\\
=&\sum_{w\in W_0^I}\epsilon_{\pm}^k(T_w)v_w^{k,\pm,I}(\gamma)
\in M^{k,\pm,I}(\gamma)^{\pm}.
\end{split}
\end{equation}
Clearly $1_\gamma^{k,\pm,I}\not=0$.
\begin{lem}\label{1spher}
For all $\gamma\in T_I^{k^{\pm 1}}$ we have $M^{k,\pm,I}(\gamma)^{\pm}=
\mathbb{C}1_\gamma^{k,\pm,I}$.
\end{lem}
\begin{proof}
Let $m\in M^{k,\pm,I}(\gamma)^{\pm}$ and write  
$m=\sum_{w\in W_0^I}a_wv_w^{k,\pm,I}(\gamma)$ with
$a_w\in\mathbb{C}$ ($w\in W_0^I$). Then $a_w=\pm k_i^{\pm 1}a_{s_iw}$ 
for $w\in A_i$ in view of Lemma \ref{ABC} and \eqref{vaction}.
Let $w=s_{i_1}s_{i_2}\cdots s_{i_r}\in W_0^I$ be a reduced expression.
Then $s_{i_j}s_{i_{j+1}}\cdots s_{i_r}\in A_{i_j}$ for $1\leq j\leq r$.
Consequently $a_w=a_e\epsilon_{\pm}^k(T_w)$ for all $w\in W_0^I$.
Thus $m=a_e1_\gamma^{k,\pm,I}$.
\end{proof}

Recall from Proposition \ref{cal} that $\{b_w^{k,\pm,I}(\gamma)\}_{w\in W_0^I}$
is a complex linear basis of $M^{k,\pm,I}(\gamma)$ if 
$\gamma\in T_I^{k^{\pm 1}}$
satisfies $\gamma^{\alpha^\vee}\not=1$ for all 
$\alpha\in R_0^+\setminus R_0^{I,+}$.
The special case $I=\emptyset$ of the following theorem goes back
to Kato, see \cite[Prop. 1.20]{Kat} (in the p-adic group case, i.e.
for constant multiplicity function $k$, see Casselman \cite{Ca}).
\begin{thm}\label{1exp}
Suppose that $\gamma\in T_I^{k^{\pm 1}}$ and $\gamma^{\alpha^\vee}\not=1$
for all $\alpha\in R_0^+\setminus R_0^{I,+}$.
Let $M$ be a left $H(k)$-module and $m\in M$ satisfying
$hm=\chi_\gamma^{k,\pm,I}(h)m$ for all $h\in H_I(k)$. Then
\begin{equation}\label{Cm}
\begin{split}
C_{\pm}^I(k)m&=\Bigl(\prod_{\alpha\in R_0^+\setminus R_0^{I,+}}
\frac{\pm k_{\alpha^\vee}^{\pm 1}}{1-\gamma^{\alpha^\vee}}\Bigr)\\
&\times\sum_{w\in W_0^I}\left(
\prod_{\alpha\in (R_0^+\setminus R_0^{I,+})\cap w^{-1}(R_0^+)}
(\pm k_{\alpha^\vee}^{\mp 1}\mp 
k_{\alpha^\vee}^{\pm 1}\gamma^{\alpha^\vee})\right)I_w(k)m.
\end{split}
\end{equation}
In particular, by taking $M=M^{k,\pm,I}(\gamma)$ and $m=v_e^{k,\pm,I}(\gamma)$,  
\begin{equation}\label{1explicitformula}
\begin{split}
1_\gamma^{k,\pm,I}&=\Bigl(\prod_{\alpha\in R_0^+\setminus R_0^{I,+}}
\frac{\pm k_{\alpha^\vee}^{\pm 1}}{1-\gamma^{\alpha^\vee}}\Bigr)\\
&\times\sum_{w\in W_0^I}\left(
\prod_{\alpha\in (R_0^+\setminus R_0^{I,+})\cap w^{-1}(R_0^+)}
(\pm k_{\alpha^\vee}^{\mp 1}\mp k_{\alpha^\vee}^{\pm 1}\gamma^{\alpha^\vee})\right)
b_w^{k,\pm,I}(\gamma)
\end{split}
\end{equation}
in $M^{k,\pm,I}(\gamma)$.
\end{thm}
\begin{proof}
If $m\in M$ satisfies $hm=\chi_\gamma^{k,\pm,I}(h)m$ for all
$h\in H_I(k)$ then there exists a unique $H(k)$-linear map
$M^{k,\pm,I}(\gamma)\rightarrow M$ mapping $v_e^{k,\pm,I}(\gamma)$
onto $m$. Hence it suffices to prove
\eqref{1explicitformula}.

Let $\gamma\in T_I^{k^{\pm 1}}$ satisfying $\gamma^{\alpha^\vee}\not=1$ for all
$\alpha\in R_0^+\setminus R_0^{I,+}$. 
Then
we have a unique expansion $1_{\gamma}^{k,\pm,I}=
\sum_{w\in W_0^I}a_w^{\pm}b_w^{k,\pm,I}(\gamma)$
in $M^{k,\pm,I}(\gamma)$
for some $a_w^{\pm}\in\mathbb{C}$. The fact that 
$T_i1_\gamma^{k,\pm,I}=\pm k_i^{\pm 1}1_\gamma^{k,I}$
for $1\leq i\leq n$ implies, in view of the previous lemma 
and Lemma \ref{ABC}, 
the following
recursion relation for the $a_w^{\pm}$:
\[
\frac{(k_i^{-1}-k_i)\gamma^{w^{-1}(\alpha_i)^\vee}}
{(1-\gamma^{w^{-1}(\alpha_i)^\vee})}a_w^{\pm}+
\frac{1}{(1-\gamma^{-w^{-1}(\alpha_i)^\vee})}a_{s_iw}^{\pm}=
\pm k_i^{\pm 1}a_w^{\pm}
\]
if $1\leq i\leq n$ and $w\in A_i$. Rewriting the recursion  
relation gives
\[a_{s_iw}^{\pm}=(\pm k_i^{\mp 1}\mp k_i^{\pm 1}
\gamma^{-w^{-1}(\alpha_i)^\vee})a_w^{\pm}
\]
if $1\leq i\leq n$ and $w\in A_i$.

Let $w\in W_0^I$. By Lemma \ref{Blift} there exist
$1\leq i_j\leq n$ such that 
$\overline{w_0}=s_{i_1}s_{i_2}\cdots s_{i_r}w$
and $l(\overline{w_0})=l(w)+r$.
Then for $1\leq j\leq r$,
\[
s_{i_j}s_{i_{j+1}}\cdots s_{i_r}w=s_{i_{j-1}}s_{i_{j-2}}\cdots 
s_{i_1}\overline{w_0}\in A_{i_j}.
\]
Hence 
\begin{equation*}
\begin{split}
a_{w}^{\pm}&=a_{s_{i_r}(s_{i_r}w)}^{\pm}\\
&=a_{s_{i_r}w}^{\pm}(\pm k_{i_r}^{\mp 1}\mp 
k_{i_r}^{\pm 1}\gamma^{w^{-1}(\alpha_{i_r})^\vee})\\
&=\cdots=a_{\overline{w_0}}^{\pm}
\prod_{\alpha\in (R_0^+\setminus R_0^{I,+})\cap w^{-1}(R_0^+)}
(\pm k_{\alpha^\vee}^{\mp 1}\mp k_{\alpha^\vee}^{\pm 1}\gamma^{\alpha^\vee}),
\end{split}
\end{equation*}
where we have used that 
\begin{equation*}
\begin{split}
w(R_0^+\setminus R_0^{I,+})\cap R_0^+&=
w(R_0^+\setminus R_0^{I,+}\cup R_0^{I,-})\cap R_0^+\\
&=w\overline{w_0}^{-1}(R_0^-)\cap R_0^+\\
&=(\overline{w_0}w^{-1})^{-1}R_0^-\cap R_0^+\\
&=\{\alpha_{i_r},s_{i_r}(\alpha_{i_{r-1}}),\ldots,
s_{i_r}s_{i_{r-1}}\cdots s_{i_2}(\alpha_{i_1})\}.
\end{split}
\end{equation*}
It thus remains to show that
\begin{equation}\label{lris}
a_{\overline{w_0}}^{\pm}=\prod_{\alpha\in R_0^+\setminus R_0^{I,+}}
\frac{\pm k_{\alpha^\vee}^{\pm 1}}{1-\gamma^{\alpha^\vee}}=
\epsilon_{\pm}^k(T_{\overline{w_0}})
\prod_{\alpha\in R_0^+\setminus R_0^{I,+}}
\frac{1}{1-\gamma^{\alpha^\vee}}.
\end{equation}
By Lemma \ref{expb} and the fact that
$R_0^+\cap \overline{w_0}^{-1}(R_0^-)=R_0^+\setminus R_0^{I,+}$,
we have
\[\sum_{w\in W_0^I}a_w^{\pm}b_w^{k,\pm,I}(\gamma)=
a_{\overline{w_0}}^{\pm}\Bigl(
\prod_{\alpha\in R_0^+\setminus R_0^{I,+}}(1-\gamma^{\alpha^\vee})\Bigr)
v_{\overline{w_0}}^{k,\pm,I}(\gamma)+\sum_{u<\overline{w_0}}
c_u^{\pm}v_u^{k,\pm,I}(\gamma)
\]
for certain $c_u^{\pm}\in\mathbb{C}$. On the other hand, by definition
this is
equal to $1_\gamma^{k,\pm,I}=
\sum_{w\in W_0^I}\epsilon_{\pm}^k(T_w)v_w^{k,\pm,I}(\gamma)$.
Comparing the coefficient of $v_{\overline{w_0}}^{k,\pm,I}(\gamma)$ 
in both
expressions gives \eqref{lris}.
\end{proof}

\section{Double affine Hecke algebras and 
quantum affine KZ equations}
\subsection{The double affine Hecke algebra}
In this subsection we recall the definition 
of Cherednik's double affine Hecke algebra and we state
some of its fundamentel properties.
To keep the technicalities to a minimum, we only
discuss the {\it twisted case} (in the sense that it
is naturally associated to reduced twisted affine Lie
algebras). 

Let $m$ be the positive integer such that $m\langle P^\vee,P^\vee\rangle=
\mathbb{Z}$. In the remainder of the paper we fix arbitrary 
multiplicity function
$k$ and arbitrary $q^{\frac{1}{m}}\in\mathbb{C}^\times$
(it is allowed to be a root of unity), unless explicitly specified
otherwise. We write 
\[q^r:=\bigl(q^{\frac{1}{m}}\bigr)^{mr},\qquad r\in\frac{1}{m}\mathbb{Z}.
\]
For $\lambda\in P^\vee$ define torus elements
$q^\lambda\in T=\textup{Hom}_{\mathbb{Z}}\bigl(P^\vee,\mathbb{C}^\times\bigr)$ 
by $P^\vee\ni\mu\mapsto q^{\langle\lambda,\mu\rangle}$.
We extend the $W_0$-action on $T$ to a $q^{\frac{1}{m}}$-dependent action of 
the extended affine Weyl group $W\simeq W_0\ltimes P^\vee$ on $T$ by
$\tau(\lambda)_qt=q^\lambda t$ for $\lambda\in P^\vee$
and $t\in T$. 

The corresponding contragredient action of $W$ on $\mathcal{M}(T)$
is explicitly given by 
\begin{equation}\label{qaction}
\begin{split}
(w_qf)(t)&=f(w^{-1}t),\qquad w\in W_0,\\
(\tau(\lambda)_qf)(t)&=f(q^{-\lambda}t),\qquad \lambda\in P^\vee.
\end{split}
\end{equation}
In particular, on the monomial basis $e^\mu$ ($\mu\in P^\vee$)
of $\mathbb{C}[T]$ the $W$-action takes on the form
\begin{equation*}
\begin{split}
w_q(e^\mu)&=e^{w\mu},\qquad\qquad\,\,\, w\in W_0,\\
\tau(\lambda)_q(e^\mu)&=
q^{-\langle\lambda,\mu\rangle}e^\mu,\qquad \lambda\in P^\vee.
\end{split}
\end{equation*}
We extend the definition of the monomials $e^\mu$ ($\mu\in P^\vee$) 
in such a way that the latter formulas can be 
captured in terms of a $W$-action
on the exponents of the generalized monomials. 
Consider the $W$-invariant subset
\[
\widehat{P}^\vee:=\frac{1}{m}\mathbb{Z}c+P^\vee
\]
of $V$. Note that $\widehat{P}^\vee$ contains the affine root system $R$. 
We set for
$rc+\lambda\in \widehat{P}^\vee$ ($r\in\frac{1}{m}\mathbb{Z}$, 
$\lambda\in P^\vee$),
\[
e_q^{rc+\lambda}:=q^{r}e^\lambda\in\mathbb{C}[T].
\]
Then it is an easy verification that
\[
w_q\bigl(e_q^{\hat{\mu}}\bigr)=e_q^{w\hat{\mu}}
\qquad (w\in W,\,\, \hat{\mu}\in\widehat{P}^\vee).
\]
We write $t_q^{\hat{\mu}}$ for the evaluation of $e_q^{\hat{\mu}}$
at $t\in T$.
We write $\mathbb{C}(T)\#_qW$ for the associated
smashed product algebra (note that it depends on the choice
$q^{\frac{1}{m}}$ of the $m$th root of $q$). 
It thus is $\mathbb{C}(T)\otimes_{\mathbb{C}}
\mathbb{C}[W]$ as complex vector space, with the canonical 
embeddings of $\mathbb{C}(T)$ and $\mathbb{C}[W]$ algebra maps,
and with cross relations governed by \eqref{qaction}:
$w\cdot p=(w_qp)\cdot w$ ($w\in W$, $p\in\mathbb{C}(T)$).

Note that $\mathbb{C}(T)\#_qW$ acts canonically on 
$\mathbb{C}(T)$ as $q$-difference reflection operators with coefficients
from $\mathbb{C}(T)$. 
This action is faithful unless $q^{\frac{1}{m}}$ is a root of unity.
Despite this fact, it is convenient to think of
$\mathbb{C}(T)\#_qW$ as the 
algebra of $q$-difference reflection
operators with coefficients from $\mathbb{C}(T)$. 

Define $c_{a}^{k,q}\in\mathbb{C}(T)$ by
\begin{equation}
c_a^{k,q}:=\frac{k_a^{-1}-k_ae_q^a}{1-e_q^a}\qquad (a\in R).
\end{equation}
It coincides with the definition \eqref{calpha} of $c_a^k$ 
when $a\in R_0^\vee$.
We denote $c_j^{k,q}=c_{a_j}^{k,q}$ for $0\leq j\leq n$.
The following
result is essentially due to Cherednik. 
The only difference is that we allow $q^{\frac{1}{m}}$
to be a root of unity.
\begin{thm}
There exists a unique injective unital algebra homomorphism
$\pi^{k,q}: H(k)\rightarrow \mathbb{C}(T)\#_qW$ satisfying
\begin{equation*}
\begin{split}
\pi^{k,q}(T_j)&=k_j+c_j^{k,q}(s_j-1),\\
\pi^{k,q}(T_{\omega})&=\omega
\end{split}
\end{equation*}
for $0\leq j\leq n$ and $\omega\in\Omega$.
\end{thm}
\begin{proof}
The (by now standard) arguments showing that the above formulas give 
rise to a unique algebra homomorphism $\pi^{k,q}: H(k)\rightarrow 
\mathbb{C}(T)\#_qW$ are valid without restrictions
on $k$ and $q^{\frac{1}{m}}$ (see, e.g., \cite[\S 4.3]{M}). 
It remains to show that 
$\pi^{k,q}$ is injective. This follows from a simple modification
of the proof of \cite[(4.3.11)]{M}, working in $\mathbb{C}(T)\#_qW$
instead of in $\textup{End}_{\mathbb{C}}\bigl(\mathbb{C}(T)\bigr)$
(the latter being the image space for the representation map
associated to the canonical action of $\mathbb{C}(T)\#_qW$ on $\mathbb{C}(T)$).
\end{proof}
The following result is due to Cherednik \cite[Theorem 2.1]{CDFT}. 
\begin{thm}\label{CherThm}
Up to isomorphism, there exists a unique complex, unital, 
associative algebra $\mathbb{H}(k,q)$ satisfying the following properties.
\begin{enumerate}
\item[{\bf (i)}] $\mathbb{C}[T]$
and $H(k)$ are subalgebras of $\mathbb{H}(k)$,
\item[{\bf (ii)}] the multiplication map defines a linear
isomorphism $\mathbb{C}[T]\otimes_{\mathbb{C}}H(k)\overset{\sim}{\longrightarrow}
\mathbb{H}(k,q)$,
\item[{\bf (iii)}] for all $f\in\mathbb{C}[T]$, $0\leq j\leq n$
and $\omega\in\Omega$ we have in $\mathbb{H}(k,q)$,
\begin{equation}\label{crossHH}
\begin{split}
T_jf&=(s_{j,q}f)T_j +(c_j^{k,q}-k_j)((s_{j,q}f)-f),\\
\omega f&=(\omega_qf)\omega.
\end{split}
\end{equation}
\end{enumerate}
\end{thm}
\begin{proof}
The modification of the proof of \cite[(4.3.11)]{M} (see 
the proof of the previous theorem) shows that
$\{e^\lambda \pi^{k,q}(T_w)\}_{\lambda\in P^\vee,w\in W}$ is $\mathbb{C}$-linear
independent in $\mathbb{C}(T)\#_qW$ (also if $q^{\frac{1}{m}}$ is a root of 
unity). Consequently $\mathbb{H}(k,q)$ can be realized as the subalgebra
of $\mathbb{C}(T)\#_qW$ generated by $\mathbb{C}[T]$ and $\pi^{k,q}(H(k))$.
\end{proof}
The algebra $\mathbb{H}(k,q)$ is called the double affine
Hecke algebra (note that it depends on the choice
$q^{\frac{1}{m}}$ of the $m$th root of $q$).

Since $\mathbb{C}[T]^\times:=\mathbb{C}[T]\setminus\{0\}
\subset \mathbb{H}(k,q)$ is a left 
Ore set, we can form the corresponding left 
localized double affine Hecke algebra $\mathbb{H}_{loc}(k,q)$. 
Theorem \ref{CherThm} is valid for $\mathbb{H}_{loc}(k,q)$
with the role of $\mathbb{C}[T]$ replaced by $\mathbb{C}(T)$; we will
call it the localized version of Theorem \ref{CherThm}.
By (the proof of) Theorem \ref{CherThm}, the algebra homomorphism
$\pi^{k,q}: H(k)\rightarrow \mathbb{C}(T)\#_qW$ uniquely
extends to an injective algebra homomorphism
\[\mathbb{H}_{loc}(k,q)\rightarrow \mathbb{C}(T)\#_qW
\]
mapping $f\in\mathbb{C}(T)\subset\mathbb{H}_{loc}(k,q)$
to $f$ viewed as element in $\mathbb{C}(T)\#_qW$.
The resulting algebra homomorphism will again be denoted
by $\pi^{k,q}: \mathbb{H}_{loc}(k,q)\rightarrow \mathbb{C}(T)\#_qW$.
\begin{rema}
Composing $\pi^{k,q}$ with the algebra map $\mathbb{C}(T)\#_qW\rightarrow
\textup{End}_{\mathbb{C}}\bigl(\mathbb{C}(T)\bigr)$ arising from the canonical
action of $\mathbb{C}(T)\#_qW$ on $\mathbb{C}(T)$ as $q$-difference reflection
operators, turns $\mathbb{C}(T)$ into 
a left $\mathbb{H}_{loc}(k,q)$-module. Restricting the action to the
double affine Hecke algebra $\mathbb{H}(k,q)$, 
the algebra $\mathbb{C}[T]$ of regular functions on $T$ becomes 
a $\mathbb{H}(k,q)$-invariant subspace of $\mathbb{C}(T)$. 
The resulting left $\mathbb{H}(k,q)$-module $\mathbb{C}[T]$
is Cherednik's basic representation of $\mathbb{H}(k,q)$.
It is faithful unless $q^{\frac{1}{m}}$ is a root of unity.
\end{rema}
Observe that $\pi^{k,q}: \mathbb{H}_{loc}(k,q)\rightarrow \mathbb{C}(T)\#_qW$
is in fact an algebra isomorphism.
The pre-images of the $s_j\in\mathbb{C}(T)\#_qW$
are given by
\begin{equation}\label{intertwiner}
\bigl(\pi^{k,q}\bigr)^{-1}(s_j)=\bigl(c_j^{k,q}\bigr)^{-1}
\bigl(T_j-k_j+c_j^{k,q}\bigr)\in\mathbb{H}_{loc}(k,q),
\qquad 0\leq j\leq n. 
\end{equation}
They are called the normalized ($X$-)intertwiners of
the localized double affine Hecke algebra.

\subsection{Algebras of $H(k)$-valued $q$-difference reflection
operators}

\begin{defi}
Denote $\mathbb{C}_\sigma^{q}[T]$ 
(respectively $\mathbb{C}_\nabla^{k,q}[T]$) for the subalgebra
of $\mathbb{C}(T)$
generated by $\mathbb{C}[T]$ and $(1-e_q^a)^{-1}$ for all
$a\in R$ (respectively $\mathbb{C}[T]$ and $(k_a^{-1}-k_ae_q^a)^{-1}$
for all $a\in R$).
\end{defi}
Note that $\mathbb{C}_{\sigma}^q[T]$ (respectively $\mathbb{C}_\nabla^{k,q}[T]$)
is a $W$-module subalgebra of $\mathbb{C}(T)$ 
with respect to the action \eqref{qaction}, and it contains
$c_a^{k,q}$ (respectively $(c_a^{k,q})^{-1}$) for all $a\in R$.
The possible singularities of 
$f\in\mathbb{C}_\sigma^q[T]$, respectively $f\in\mathbb{C}_\nabla^{k,q}[T]$, 
are at 
\[
\mathcal{S}_\sigma^q:=\{t\in T \,\, | \,\,
t_q^a=1 \,\, \textup{ for some } a\in R \},
\]
respectively at
\[
\mathcal{S}_\nabla^{k,q}:=\{t\in T \,\, | \,\, 
t_q^a=k_a^{2}\,\, \textup{ for some } a\in R \}.
\]
Note that $\mathcal{S}_\nabla^{k,q}=\mathcal{S}_\nabla^{k^{-1},q}$
and that $T_I^{k}\subseteq \mathcal{S}_{\nabla}^{k,q}$ if $I\not=\emptyset$.

We can now form the smashed product algebras
$\mathbb{C}_\sigma^{q}[T]\#_qW$ and $\mathbb{C}_\nabla^{k,q}[T]\#_qW$,
which are subalgebras of $\mathbb{C}(T)\#_qW$. 
The algebras of $H(k)$-valued $q$-difference reflection algebras
with coefficients from $\mathbb{C}_\sigma^q[T]$, $\mathbb{C}_\nabla^{k,q}[T]$
and $\mathbb{C}(T)$ are
\begin{equation}\label{A}
\begin{split}
\mathcal{A}_\sigma^{k,q}&:=\mathbb{C}_\sigma^q[T]\#_qW\otimes_{\mathbb{C}}H(k),\\
\mathcal{A}_\nabla^{k,q}&:=\mathbb{C}_\nabla^{k,q}[T]\#_qW\otimes_{\mathbb{C}}H(k).
\end{split}
\end{equation}
and
\[\mathcal{A}^{k,q}:=\mathbb{C}(T)\#_qW\otimes_{\mathbb{C}}H(k),
\]
respectively.
We will identify the algebras
$\mathbb{C}(T)\#_qW$ and $H(k)$ with their canonical images in
$\mathcal{A}^{k,q}$ (and similarly in case of $\mathcal{A}_\sigma^{k,q}$
and $\mathcal{A}_\nabla^{k,q}$). The following statement is essentially
a reformulation of \cite[Theorem 2.3]{CInd}. 
\begin{prop}\label{prop1}
There exists a unique algebra homomorphism
$\sigma^{k,q}: \mathbb{H}_{loc}(k,q)\rightarrow \mathcal{A}^{k,q}$
satisfying
\begin{equation}\label{sigmaMAP}
\begin{split}
\sigma^{k,q}(f)&=f,\\
\sigma^{k,q}(T_j)&=s_{j}T_j+(c_j^{k,q}-k_j)(s_{j}-1),\\
\sigma^{k,q}(T_{\omega})&=\omega T_{\omega}
\end{split}
\end{equation}
for $f\in\mathbb{C}(T)$, $0\leq j\leq n$ and $\omega\in\Omega$.
Furthermore, $\sigma^{k,q}\bigl(\mathbb{H}(k,q)\bigr)\subseteq
\mathcal{A}^{k,q}_\sigma$.
\end{prop}
\begin{proof}
A direct verification shows 
that the assignments \eqref{sigmaMAP} respect 
the cross relations \eqref{crossHH}, as well as the relations
$T_\omega T_jT_{\omega^{-1}}=T_{\omega(j)}$ in $H(k)\subset\mathbb{H}(k,q)$
($\omega\in\Omega$ and $0\leq j\leq n$).
It thus remains to show that the $\sigma^{q,k}(T_j)\in\mathcal{A}^{k,q}$
($0\leq j\leq n$) from \eqref{sigmaMAP} satisfy the defining relations
\eqref{braidrelations2} of $H^a(k)$, for which it 
suffices to provide a proof if $q^{\frac{1}{m}}$ is not a root of unity.

Suppose that $q^{\frac{1}{m}}$ is not a root of unity.
Consider the $\mathbb{H}_{loc}(k,q)$-module
\[
\textup{Ind}_{H(k)}^{\mathbb{H}_{loc}(k,q)}\bigl(H(k)\bigr),
\]
where $H(k)$ is considered as left $H(k)$-module by left multiplication.
By {\bf (i)} of the localized version
of Theorem \ref{CherThm} it is isomorphic to 
$\mathbb{C}(T)\otimes_{\mathbb{C}}H(k)$ as a complex vector space. 
Denote the resulting representation map by
\[\sigma: \mathbb{H}_{loc}(k,q)\rightarrow \textup{End}_{\mathbb{C}}
\bigl(\mathbb{C}(T)\otimes_{\mathbb{C}}H(k)\bigr).
\]
Since $q^{\frac{1}{m}}$ is not a root of unity, the formula
\[
(pw\otimes h): r\otimes h^\prime\mapsto p(w_qr)\otimes hh^\prime
\]
for $p,r\in\mathbb{C}(T)$, $w\in W$ and $h,h^\prime\in H(k)$
defines an algebra embedding 
\[\mathcal{A}^{k,q}\hookrightarrow
\textup{End}_{\mathbb{C}}
\bigl(\mathbb{C}(T)\otimes_{\mathbb{C}}H(k)\bigr).
\]
We identify $\mathcal{A}^{k,q}$ with its image in
$\textup{End}_{\mathbb{C}}\bigl(\mathbb{C}(T)\otimes_{\mathbb{C}}H(k)\bigr)$.
By a direct computation using {\bf (ii)} and
{\bf (iii)} of the localized version of 
Theorem \ref{CherThm}, it follows that $\sigma(f)$, $\sigma(T_j)$
and $\sigma(T_\omega)$ for $f\in\mathbb{C}(T)$, $0\leq j\leq n$
and $\omega\in\Omega$ are given by \eqref{sigmaMAP}.
In particular, they lie in the subalgebra $\mathcal{A}^{k,q}$.
Thus $\sigma$ is an algebra homomorphism $\sigma: 
\mathbb{H}_{loc}(k,q)\rightarrow
\mathcal{A}^{k,q}$, satisfying \eqref{sigmaMAP}. 

The last statement of the proposition is immediate.
\end{proof}
Note that $\pi^{k,q}=\bigl(\textup{id}\otimes \epsilon_+^k\bigr)\circ
\sigma^{k,q}$, where  
$\pi^{k,q}: \mathbb{H}_{loc}(k,q)\rightarrow
\mathbb{C}(T)\#_qW$ is the algebra isomorphism as defined in
the previous subsection. 
In particular, $\pi^{k,q}\bigl(\mathbb{H}(k,q)\bigr)\subseteq
\mathbb{C}_\sigma^q[T]\#_qW$.
\begin{rema}
Let $-k^{-1}$ be the multiplicity function on $R$ that takes the value
$-k_a^{-1}$ on $a\in R$.
Since
\[c_a^{k,q}(t)-k_a=c_a^{-k^{-1},q}(t)+k_a^{-1}\quad 
(a\in R),
\]
we have a unique algebra isomorphism
${}^\dagger: \mathbb{H}_{loc}(k,q)\overset{\sim}{\longrightarrow} 
\mathbb{H}_{loc}(-k^{-1},q)$  satisfying $T_j^\dagger=T_j$
($0\leq j\leq n$), $T_\omega^\dagger=T_\omega$ ($\omega\in\Omega$) and
$f^\dagger=f$ ($f\in\mathbb{C}(T)$). 
Then
$\pi^{-k^{-1},q}\circ 
{}^\dagger=(\textup{id}\otimes\epsilon_-^k)\circ\sigma^{k,q}$.
\end{rema}

\begin{cor}\label{cor1}
Fix a multiplicity function $k$ on $R$ and fix 
$q^{\frac{1}{m}}\in\mathbb{C}^\times$.
There exists a unique algebra homomorphism
$\nabla^{k,q}: \mathbb{C}(T)\#_qW\rightarrow \mathcal{A}^{k,q}$
satisfying
\begin{equation}\label{nablarel}
\begin{split}
\nabla^{k,q}(f)&=f,\\
\nabla^{k,q}(s_j)&=(c_j^{k,q})^{-1}s_{j}T_j+
\frac{c_j^{k,q}-k_j}{c_j^{k,q}}s_{j},\\
\nabla^{k,q}(\omega)&=\omega T_{\omega}
\end{split}
\end{equation}
for $f\in\mathbb{C}(T)$, $0\leq j\leq n$ and $\omega\in\Omega$.
Furthermore, $\nabla^{k,q}(\mathbb{C}[T]\#_qW)\subseteq
\mathcal{A}_{\nabla}^{k,q}$.
\end{cor}
\begin{proof}
Consider the algebra homomorphism
\[
\nabla^{k,q}:=\sigma^{k,q}\circ\bigl(\pi^{k,q}\bigr)^{-1}:
\mathbb{C}(T)\#_qW\rightarrow \mathcal{A}^{k,q}.
\]
A direct computation using \eqref{intertwiner} 
shows that $\nabla^{k,q}$ satisfies \eqref{nablarel}.
The second statement is immediate.
\end{proof}
The algebra homomorphism $\nabla^{k,q}$ is the key ingredient
in the definition of Cherednik's \cite{CInd} quantum affine
Knizhnik-Zamolodhikov (KZ) equations. We discuss this in detail
in subsection \ref{qaKZ}.

\subsection{Characterizations of spaces of invariants}

For a complex, unital associative algebra $A$ we denote by
$\textup{Mod}_A$ the category of left $A$-modules.

Proposition \ref{prop1} gives rise to a covariant functor
$F_\sigma: \textup{Mod}_{\mathcal{A}_{\sigma}^{k,q}}\rightarrow
\textup{Mod}_{H(k)}$ in the following way. If $M$ is 
a left $\mathcal{A}_{\sigma}^{k,q}$-module $M$, then
$F_\sigma(M)$ is the
vector space $M$ with $H(k)$-module structure defined by
\[h\cdot m:=\sigma^{k,q}(h)m\qquad (h\in H(k)\subset \mathbb{H}(k,q),
\,\,\, m\in M).
\]
Similarly, Corollary \ref{cor1} gives rise to a covariant functor
$F_\nabla: \textup{Mod}_{\mathcal{A}_{\nabla}^{k,q}}\rightarrow
\textup{Mod}_{\mathbb{C}[W]}$. In this case $F_\nabla(M)$, 
for a left $\mathcal{A}_{\nabla}^{k,q}$-module $M$, 
is the vector space $M$ with $\mathbb{C}[W]$-module structure
given by
\[w\cdot m:=\nabla^{k,q}(w)m\qquad (w\in W,\,\, m\in M).
\]
\begin{rema}
Since $\mathbb{C}_\sigma^{q}[T]\#_qW$ and $H(k)$ are
mutually commuting subalgebras of $\mathcal{A}_{\sigma}^{k,q}$,
both
$\pi^{k^{-1},q}(H(k^{-1}))\subseteq \mathbb{C}_\sigma^q[T]\#_qW$
and $H(k)$ act on a left
$\mathcal{A}_{\sigma,\nabla}^{k,q}$-module $M$, and these actions
commute. It is important to carefully distinguish between these
two commuting actions. 
\end{rema}
The next aim is to relate certain invariant subspaces of $F_\sigma(M)$
and of $F_\nabla(M)$ in case $M$ is a left module over 
the subalgebra $\mathcal{A}_{\sigma,\nabla}^{k,q}$ of $\mathcal{A}^{k,q}$,
generated by $\mathcal{A}_\sigma^{k,q}$ and $\mathcal{A}_\nabla^{k,q}$.
We first need to introduce some more notations.

Write $\mathbb{C}^{k,q}_{\sigma,\nabla}[T]\subseteq \mathbb{C}(T)$
for the subalgebra generated by $\mathbb{C}[T]$, $(1-e_q^a)^{-1}$ and
$(k_a^{-1}-k_ae_q^a)^{-1}$ for all $a\in R$. Note that 
$\mathbb{C}_{\sigma,\nabla}^{k,q}[T]=
\mathbb{C}_{\sigma,\nabla}^{k^{-1},q}[T]$.
Then
\[\mathbb{C}[T]\subseteq
\mathbb{C}_\sigma^q[T], \mathbb{C}_\nabla^{k,q}[T]\subseteq
\mathbb{C}^{k,q}_{\sigma,\nabla}[T]\subseteq \mathbb{C}(T)
\]
as $W$-module algebras. 
The elements $c_a^{k,q}$ ($a\in R$) are invertible 
in $\mathbb{C}^{k,q}_{\sigma,\nabla}[T]$. Let
$\mathbb{C}_{\sigma,\nabla}^{k,q}[T]\#_qW$ be the 
algebra of $q$-difference reflection operators with coefficients
in $\mathbb{C}_{\sigma,\nabla}^{k,q}[T]$.
Then
\[\mathcal{A}_{\sigma,\nabla}^{k,q}=
\mathbb{C}_{\sigma,\nabla}^{k,q}[T]\#_qW\otimes_{\mathbb{C}} H(k)
\]
and $\mathcal{A}_\sigma^{k,q},\mathcal{A}_{\nabla}^{k,q}\subseteq
\mathcal{A}_{\sigma,\nabla}^{k,q}\subseteq \mathcal{A}^{k,q}$
as algebras.
\begin{defi}
For a left $\mathcal{A}_{\sigma,\nabla}^{k,q}$-module $M$
we write $M_\sigma=F_\sigma\bigl(M|_{\mathcal{A}_{\sigma}^{k,q}}\bigr)$
and $M_\nabla=F_\nabla\bigl(M|_{\mathcal{A}_{\nabla}^{k,q}}\bigr)$
for the associated $H(k)$-module and $\mathbb{C}[W]$-module, respectively.
\end{defi}

For a left $\mathcal{A}_{\sigma,\nabla}^{k,q}$-module $M$,
for a subalgebra $A\subseteq H(k)$ and
for a subgroup $G\subseteq W$ we now write
\begin{equation*}
\begin{split}
M_\sigma^A&:=\{m\in M \,\, | \,\, \sigma^{k,q}(a)m=\epsilon_+^{k}(a)m\quad
\forall\, a\in A\},\\
M_\nabla^G&:=\{m\in M \,\, | \,\, \nabla^{k,q}(g)m=m\quad
\forall\, g\in G\}.
\end{split}
\end{equation*}

Let $J_k: H(k^{-1})\rightarrow H(k)$ be the unique unital
algebra anti-involution satisfying $J_k(T_j)=T_j^{-1}$
($0\leq j\leq n$) and 
$J_k(T_\omega)=T_{\omega^{-1}}$ ($\omega\in\Omega$).
Note that $J_k$ restricts to an algebra anti-involution
$J_k: H_0(k^{-1})\rightarrow H_0(k)$.

\begin{prop}\label{aa}
Let $M$ be a left $\mathcal{A}_{\sigma,\nabla}^{k,q}$-module.\\
{\bf (i)} We have
\begin{equation}\label{a1}
M_\sigma^{H_0(k)}=M_\nabla^{W_0}=
\{m\in M \,\, | \,\, \pi^{k^{-1},q}(h)m=J_k(h)m
\quad \forall\, h\in H_0(k^{-1})\}.
\end{equation}
{\bf (ii)} We have
\begin{equation}\label{a2}
M_\sigma^{H(k)}=M_\nabla^W=
\{m\in M \,\, | \,\, \pi^{k^{-1},q}(h)m=J_k(h)m
\quad \forall\, h\in H(k^{-1}) \}.
\end{equation}
\end{prop}
\begin{proof}
Observe that in $\mathcal{A}_{\sigma,\nabla}^{k,q}$,
\begin{equation*}
\begin{split}
\sigma^{k,q}(T_j)&=k_j+c_{j}^{k,q}(\nabla^{k,q}(s_j)-1),\\
\sigma^{k,q}(T_{\omega})&=\nabla^{k,q}(\omega)
\end{split}
\end{equation*}
for $0\leq j\leq n$ and $\omega\in\Omega$.
Since $c_j^{k,q}$ is invertible in 
$\mathcal{A}_{\sigma,\nabla}^{k,q}$,
it implies for $m\in M$,
\begin{equation*}
\begin{split}
\sigma^{k,q}(T_j)m=k_jm \,\, &\Leftrightarrow\,\, \nabla^{k,q}(s_j)m=m,\\
\sigma^{k,q}(T_{\omega})m=m \,\, &\Leftrightarrow\,\, \nabla^{k,q}(\omega)m=m
\end{split}
\end{equation*}
for $0\leq j\leq n$ and $\omega\in\Omega$. This implies the first equalities
in \eqref{a1} and \eqref{a2}.
For the second equalities in \eqref{a1} and \eqref{a2} it suffices to show,
for $m\in M$,
\begin{equation*}
\begin{split}
\nabla^{k,q}(s_j)m=m\,\, &\Leftrightarrow\,\, \pi^{k^{-1},q}(T_j)m=J_k(T_j)m,\\
\nabla^{k,q}(\omega)m=m\,\, &\Leftrightarrow\,\, \pi^{k^{-1},q}(T_{\omega})m=
J_k(T_{\omega})m
\end{split}
\end{equation*}
for $0\leq j\leq n$ and $\omega\in\Omega$.
The second equivalence is immediate. It thus remains to prove 
the first equivalence. 

Note that $\nabla^{k,q}(s_j)m=m$ is equivalent to
\[s_jT_jm+(c_j^{k,q}-k_j)s_jm=c_j^{k,q}m.
\]
We now act by $s_j\in\mathcal{A}_{\sigma,\nabla}^{k,q}$ on both sides,
and pull the action of $s_j$ to the right. Using the fact that
\[ w_q\bigl(c_a^{k,q}\bigr)=c_{wa}^{k,q},\qquad
c_{-a}^{k,q}=c_a^{k^{-1},q}
\]
in $\mathbb{C}(T)$ for $w\in W$ and $a\in R$, it follows that
$\nabla^{k,q}(s_j)m=m$ is equivalent to
\[ T_jm+(c_j^{k^{-1},q}-k_j)m=c_j^{k^{-1},q}s_jm.
\]
Since $T_j^{-1}=T_j-k_j+k_j^{-1}$ in $H(k)$, we conclude
that $\nabla^{k,q}(s_j)m=m$ if and only if
\[T_j^{-1}m=(k_j^{-1}+c_j^{k^{-1},q}(s_j-1))m.
\]
The left hand side equals $J_k(T_j)m$ and the right hand
side equals $\pi^{k^{-1},q}(T_j)m$, hence the result.
\end{proof}
\begin{rema}
The third form of the space of invariants (the far right side
of \eqref{a2}) is used in the analysis of special solutions of 
quantum Knizhnik-Zamolodchikov equations in the context of the
Razumov-Stroganov conjectures, see, e.g., \cite[\S 4.1]{Pas}. 
\end{rema}

\subsection{Quantum affine Knizhnik-Zamolodchikov equations}\label{qaKZ}

In this subsection we recall Cherednik's \cite{CInd} construction of the
quantum affine KZ equations.

Observe that for $w\in W$,
\begin{equation}\label{connmatrix}
\nabla^{k,q}(w)=C_w^{k,q}w\in\mathcal{A}_{\nabla}^{k,q},
\end{equation}
with $C_w^{k,q}$ an element in the subalgebra 
$\mathbb{C}_{\nabla}^{k,q}[T]\otimes_{\mathbb{C}}H(k)$ of
$\mathcal{A}_{\nabla}^{k,q}$. It follows that the $C_w^{k,q}$
satisfies the cocycle conditions
\begin{equation}\label{cocycle}
C_{ww^\prime}^{k,q}=C_w^{k,q}w_q\bigl(C_{w^\prime}^{k,q}\bigr),\qquad
w,w^\prime\in W,
\end{equation}
where $w_q$ acts on the first tensor component of
$C_{w^\prime}^{k,q}\in\mathbb{C}_{\nabla}^{k,q}[T]\otimes H(k)$.
In view of the cocycle condition \eqref{cocycle}, the $C_w^{k,q}$
are uniquely determined by $C_{s_j}^{k,q}$ ($0\leq j\leq n$)
and $C_\omega^{k,q}$ ($\omega\in\Omega$). By \eqref{nablarel}, 
for $0\leq j\leq n$ and $\omega\in\Omega$,
\begin{equation}
\begin{split}
C_{s_j}^{k,q}(t)&=c_j^{k,q}(t)^{-1}T_j+\frac{c_j^{k,q}(t)-k_j}{c_j^{k,q}(t)}
=\frac{T_j^{-1}-t_q^{a_j}T_j}{k_j^{-1}-k_jt_q^{a_j}},\\
C_\omega^{k,q}(t)&=T_\omega
\end{split}
\end{equation}
as rational $H(k)$-valued functions
in $t\in T$.

For a left $\mathbb{C}_{\sigma,\nabla}^{k,q}[T]\#_qW$-module $L$
and a left $H(k)$-module $L$ we write
\[
\Gamma_L^{k,q}(N):=L\otimes_{\mathbb{C}}N
\]
for the associated $\mathcal{A}_{\sigma,\nabla}^{k,q}$-module.
Typically, $L$ is some field of functions on $T$
(for example, $L=\mathbb{C}(T)$,
or $\mathcal{M}(T)$) with 
$\mathbb{C}_{\sigma,\nabla}^{q,k}[T]\#_qW$ acting
by $q$-difference reflection operators, in which case it is convenient
to think of the resulting
$\mathcal{A}_{\sigma,\nabla}^{k,q}$-module $\Gamma_L^{k,q}(N)$ 
as some space of global sections of a trivial vector bundle over $T$ with
fiber $N$. We call $\Gamma_L^{k,q}(N)$ the space of $N$-valued functions on $T$
of class $L$.

The action of $\tau(P^\vee)$ on $\Gamma_L^{k,q}(N)_\nabla$
then gets the interpretation of an integrable $q$-connection on 
$\Gamma_L^{k,q}(N)$;
in this interpretation the cocycle values $C_{\tau(\lambda)}^{k,q}$
($\lambda\in P^\vee$), acting on $\Gamma_L^{k,q}(N)$, play the role of the
$q$-connection matrices, while the integrability is captured by the cocycle
condition \eqref{cocycle}.
The $W_0$-submodule
$\Gamma_L^{k,q}(N)_\nabla^{\tau(P^\vee)}$ 
of $\Gamma_L^{k,q}(N)_\nabla$ then plays the role of the 
subspace of flat $q$-sections. 
\begin{defi}
The system $\nabla^{k,q}(\tau(P^\vee))$
of holonomic $q$-difference equations on $\Gamma_L^{k,q}(N)$ is the
quantum affine Knizhnik-Zamolodchikov equations for $N$-valued 
functions on $T$ of class $L$. 
\end{defi}
Note that the assignment
\begin{equation}\label{functoriality}
L\times N\rightarrow \Gamma_L^{k,q}(N)_\nabla
\end{equation}
defines a covariant functor 
$\textup{Mod}_{\mathbb{C}_{\sigma,\nabla}^{k,q}[T]\#_qW}\times\textup{Mod}_{H(k)}
\rightarrow \textup{Mod}_{\mathbb{C}[W]}$.
In particular, if $L\rightarrow L^\prime$ and 
$N\rightarrow N^\prime$ are morphisms of $\mathbb{C}_{\sigma,\nabla}^{k,q}[T]\#_qW$
and $H(k)$-modules respectively, it gives rise to canonical linear maps
\[\Gamma_L^{k,q}(N)_\nabla^{\tau(P^\vee)}\rightarrow
\Gamma_{L^\prime}^{k,q}(N^\prime)_\nabla^{\tau(P^\vee)},\qquad
\Gamma_L^{k,q}(N)_\nabla^W\rightarrow \Gamma_{L^\prime}^{k,q}(N^\prime)_\nabla^W.
\]

\section{Spectral problem of the Cherednik-Dunkl operators}
For a left $\mathbb{C}_{\sigma,\nabla}^{k,q}[T]\#_qW$-module $L$
and a left $H(k)$-module of the form $M^{k,\pm,I}(\gamma)$ 
($\gamma\in T_I^{k^{\pm 1}}$)
we relate in this section the space 
$\Gamma_L^{k,q}\bigl(M^{k,\pm,I}(\gamma)\bigr)_\nabla^{W}$ of $W_0$-invariant
flat $q$-sections of the quantum affine KZ equations
to a suitable space of 
common eigenfunctions of the Cherednik-Dunkl $q$-difference
reflection operators
$\pi^{k^{-1},q}(f(Y))\in\mathbb{C}_{\sigma,\nabla}^{k,q}[T]\#_qW$ 
($f\in\mathbb{C}[T]$) acting on $L$. 
 
\subsection{$W_0$-invariants} \label{KZsection}

In this subsection we analyze
$\Gamma_L^{k,q}(M^{k,\pm,I}(\gamma))_\nabla^{W_0}$ ($\gamma\in T_I^{k^{\pm 1}}$). 
In the following subsection we extend the
analysis to its subspace $\Gamma_L^{k,q}(M^{k,\pm,I}(\gamma))_\nabla^W$
of $W_0$-invariant flat $q$-sections.

If $L$ is a left $\mathbb{C}[T]_{\sigma,\nabla}^{k,q}[T]\#_qW$-module,
then we write
\begin{equation}\label{LpiIpm}
L_\pi^{I,\pm}:=\{\phi\in L \,\, | \,\, \pi^{k^{-1},q}(h)\phi=
\epsilon_{\pm}^{k^{-1}}(h)\phi\quad \forall\, h\in H_{0,I}(k^{-1}) \}.
\end{equation}
We give first an alternative description of the spaces $L_\pi^{I,\pm}$.
Set $\rho^\vee:=\frac{1}{2}\sum_{\alpha\in R_0^+}\alpha^\vee\in P^\vee$
and define $G^{k,\pm}\in \mathbb{C}[T]$ by
\begin{equation*}
G^{k,\epsilon}(t):=
\begin{cases}
1&\qquad \hbox{ if } \epsilon=+,\\ 
t^{\rho^\vee}\prod_{\alpha\in R_0^+}\bigl(
k_{\alpha^\vee}^{-1}-k_{\alpha^\vee}t^{-\alpha^\vee}\bigr),&\qquad
\hbox{ if } \epsilon=-. 
\end{cases}
\end{equation*}
Note that $G^{k,\pm}\in\mathbb{C}_{\sigma,\nabla}^{k,q}[T]^\times$.
Define for $w\in W_0$,
\begin{equation}\label{wplusmin}
w_{\pm}:=\frac{G^{k,\pm}}{w(G^{k,\pm})}w\in 
\mathbb{C}_{\sigma,\nabla}^{k,q}[T]\#_qW.
\end{equation}
It is just $w$ in the symmetric case ($+$), but it is convenient
to write it as $w_+$ to maintain an uniform treatment of the 
symmetric and antisymmetric theory.
\begin{lem}\label{W0IinvL}
Let $L$ be a left $\mathbb{C}[T]_{\sigma,\nabla}^{k,q}\#_qW$-module.
Then $L_\pi^{I,\pm}=L^{W_{0,I,\pm}}$, where
\[
L^{W_{0,I,\pm}}:=\{\phi\in L \,\, | \,\, w_{\pm}\phi=\phi\quad \forall\,
w\in W_{0,I}\}.
\]
\end{lem}
\begin{proof}
For all $i\in\{1,\ldots,n\}$ and $\phi\in L$ we have 
\begin{equation*}
\begin{split}
(\pi^{k^{-1},q}(T_i)-k_i^{-1})\phi&=c_i^{k^{-1},q}(s_i-1)\phi,\\
(\pi^{k^{-1},q}(T_i)+k_i)\phi&=\left(\frac{k_i^{-1}-k_it^{\alpha_i^\vee}}
{1-t^{\alpha_i^\vee}}\right)\left(1-\frac{G^{k,-}}{s_i(G^{k,-})}s_i\right)\phi
\end{split}
\end{equation*}
The lemma follows now immediately.
\end{proof}
\begin{prop}\label{KZW0}
Fix $\gamma\in T_I^{k^{\pm 1}}$. 
Let $L$ be a left $\mathbb{C}_{\sigma,\nabla}^{k,q}[T]\#_qW$-module
and let $\psi\in\Gamma_L^{k,q}\bigl(M^{k,\pm,I}(\gamma)\bigr)$. Then
$\psi\in\Gamma_L^{k,q}\bigl(M^{k,\pm,I}(\gamma)\bigr)_\nabla^{W_0}$ if and only
if 
\[
\psi=\sum_{w\in W_0^I}\pi^{k^{-1},q}(T_{w^{-1}}^{-1})\phi\otimes
v_w^{k,\pm,I}(\gamma)
\]
for some $\phi\in L_\pi^{I,\pm}$.
\end{prop}
\begin{proof}
Any $\psi\in \Gamma_L^{k,q}\bigl(M^{k,\pm,I}(\gamma)\bigr)$ has a unique expansion
of the form
\[\psi=\sum_{w\in W_0^I}\psi_w\otimes v_w^{k,\pm,I}(\gamma)
\]
with $\psi_w\in L$. By \eqref{a1} the condition 
$\psi\in \Gamma_L^{k,q}\bigl(M^{k,\pm,I}(\gamma)\bigr)_\nabla^{W_0}$ is 
equivalent to 
\begin{equation}\label{todo}
\sum_{w\in W_0^I}\pi^{k^{-1},q}(T_i)\psi_w\otimes v_w^{k,\pm,I}(\gamma)=
\sum_{w\in W_0^I}\psi_w\otimes T_i^{-1}v_w^{k,\pm,I}(\gamma)
\end{equation}
for all $i\in\{1,\ldots,n\}$, in view of \eqref{a1}.
Recasting \eqref{todo} as explicit recursion relations for the
$\psi_w$ using Lemma \ref{ABC}
and the explicit formulas \eqref{vaction} for the action of
the $T_i$ on the standard basis of $M^{k,\pm,I}(\gamma)$, 
implies that \eqref{todo} holds 
for $1\leq i\leq n$ if and only if
\begin{equation}\label{recurrpsi}
\pi^{k^{-1},q}(T_i)\psi_w=
\begin{cases}
\psi_{s_iw} &\,\,\hbox{ if }\,\, w\in A_i,\\
\psi_{s_iw}+(k_i^{-1}-k_i)\psi_w &\,\,\hbox{ if }\,\,
w\in B_i,\\
\pm k_{i}^{\mp 1}\psi_{w} &\,\,\hbox{ if }\,\,
w\in C_i
\end{cases}
\end{equation}
for $1\leq i\leq n$. 

Suppose that we have a solution
$\{\psi_w\}_{w\in W_0^I}\subset L$ of the recursion relations
\eqref{recurrpsi} for $1\leq i\leq n$.
Consider first the recursion \eqref{recurrpsi} for $w=e\in W_0^I$
the unit element of $W_0$ and for $i\in I$. Then $e\in C_i$,
$\overline{s_ie}=e$ and $i_e=i$, hence 
$\pi^{k^{-1},q}(T_i)\psi_e=\pm k_i^{\mp 1}\psi_e$. It follows
that $\psi_e\in L_\pi^{I,\pm}$.
If $w\in W_0^I$ with $l(w)>0$ and $w=s_{i_1}s_{i_2}\cdots s_{i_{l(w)}}$
is a reduced expression ($1\leq i_j\leq n$), then repeated application
of the first recursion in \eqref{recurrpsi} shows that
\[\psi_w=\pi^{k^{-1},q}(T_{i_1}^{-1})\psi_{s_{i_2}\cdots s_{i_{l(w)}}}=
\cdots=\pi^{k^{-1},q}(T_{w^{-1}}^{-1})\psi_e.
\]
Hence a solution $\{\psi_w\}_{w\in W_0^I}\subset L$ of the recursion relations
\eqref{recurrpsi} for $1\leq i\leq n$ is uniquely determined
by $\psi_e$, and $\psi_e$ must be an element from the
subspace $L_\pi^{I,\pm}$ of $L$.

On the other hand, let $\phi\in L_\pi^{I,\pm}$ and define
\[\psi_w:=\pi^{k^{-1},q}(T_{w^{-1}}^{-1})\phi\in L,\qquad w\in W_0^I.
\]
Then $\{\psi_w\}_{w\in W_0^I}$ satisfies the recursion relations
\eqref{recurrpsi} for $1\leq i\leq n$ due to the following
identities in $H_0(k^{-1})$,
\begin{equation*}
T_iT_{w^{-1}}^{-1}=
\begin{cases}
T_{(s_iw)^{-1}}^{-1}\,\, &\hbox{ if }\,\, w\in A_i,\\
T_{(s_iw)^{-1}}^{-1}+(k_i^{-1}-k_i)T_{w^{-1}}^{-1}\,\, &\hbox{ if }\,\, w\in B_i,\\
T_{w^{-1}}^{-1}T_{i_w}^{-1}+(k_i^{-1}-k_i)T_{w^{-1}}^{-1}\,\,
&\hbox{ if }\,\, w\in C_i
\end{cases}
\end{equation*}
for $1\leq i\leq n$. The verification of these identities 
is straightforward.
\end{proof}

\begin{cor}\label{corrLW0}
Let $\gamma\in T_I^{k^{\pm 1}}$. We have
a complex linear isomorphism
\[L_\pi^{I,\pm}
\overset{\sim}{\longrightarrow}
\Gamma_L^{k,q}\bigl(M^{k,\pm,I}(\gamma)\bigr)_\nabla^{W_0},
\]
defined by
\begin{equation}\label{isoY}
\phi\mapsto 
\sum_{w\in W_0^I}\pi^{k^{-1},q}\bigl(T_{w^{-1}}^{-1}\bigr)\phi
\otimes v_w^{k,\pm,I}(\gamma).
\end{equation}
\end{cor}

\subsection{$W$-invariants} 

For the analysis of the $W$-invariants in 
$\Gamma_L^{k,q}(M^{k,\pm,I}(\gamma))_\nabla$
it is convenient to reformulate Corollary
\ref{corrLW0} in the following way.

For $i\in I$ let $i_I^*\in\{1,\ldots,n\}$
such that $\overline{w_0}(\alpha_i)=\alpha_{i_I^*}$, and set $I^*=
\{i_I^*\}_{i\in I}$. 
It follows from 
the identities
\[
T_{\overline{w_0}^{-1}}^{-1}T_iT_{\overline{w_0}^{-1}}=
T_{i_I^*},\qquad \forall\, i\in I
\]
in $H_0(k^{-1})$ that 
$L_\pi^{I^*,\pm}=\pi^{k^{-1},q}\bigl(T_{\overline{w_0}^{-1}}^{-1}\bigr)
L_\pi^{I,\pm}$.
Thus Corollary \ref{corrLW0} can be reformulated 
as follows.
\begin{cor}\label{coreq}
Let $L$ be a left $\mathbb{C}_{\sigma,\nabla}^{k,q}[T]\#_qW$-module
and $\gamma\in T_I^{k^{\pm 1}}$.
We have a linear isomorphism
\[
L_\pi^{I^*,\pm}
\overset{\sim}{\longrightarrow} \Gamma_L^{k,q}\bigl(M^{k,\pm,I}(\gamma)
\bigr)_{\nabla}^{W_0},
\]
defined by
\begin{equation}\label{isoY2}
\phi\mapsto \psi_\phi:=\sum_{w\in W_0^I}\pi^{k^{-1},q}
\bigl(T_{w\overline{w_0}^{-1}}\bigr)\phi\otimes
v_w^{k,\pm,I}(\gamma).
\end{equation}
\end{cor}
\begin{proof}
It follows from Corollary \ref{corrLW0} that 
we have a linear isomorphism
\[
L_\pi^{I^*,\pm}
\overset{\sim}{\longrightarrow} \Gamma_L^{k,q}\bigl(M^{k,\pm,I}(\gamma)
\bigr)_{\nabla}^{W_0}
\]
given by
\[
\phi\mapsto \sum_{w\in W_0^I}\pi^{k^{-1},q}
\bigl(T_{w^{-1}}^{-1}T_{\overline{w_0}^{-1}}\bigr)\phi\otimes
v_w^{k,\pm,I}(\gamma).
\]
It thus suffices to show that
$T_{w^{-1}}^{-1}T_{\overline{w_0}^{-1}}=T_{w\overline{w_0}^{-1}}$
in $H_0(k^{-1})$ if $w\in W_0^I$.
This follows from the fact that for $w\in W_0^I$,
\begin{equation*}
\begin{split}
l(w\overline{w_0}^{-1})&=l(w\underline{w_0}w_0)=l(w_0)-l(w\underline{w_0})\\
&=l(w_0)-l(\underline{w_0})-l(w)=l(\overline{w_0}^{-1})-l(w).
\end{split}
\end{equation*}
\end{proof}
The following lemma should be compared to Lemma \ref{chilemma}{\bf (i)}.
\begin{lem}\label{dualgamma}
If $\gamma\in T_I^k$ then $\overline{w_0}\gamma^{-1}\in T_{I^*}^{k^{-1}}$.
In particular, for $\gamma\in T_I^{k^{\pm 1}}$ we have a well
defined character $\chi_{\overline{w_0}\gamma^{-1}}^{k^{-1},\pm,I^*}:
H_{I^*}(k^{-1})\rightarrow\mathbb{C}$.
\end{lem}
\begin{proof}
Let $\gamma\in T_I^k$. Since $\overline{w_0}(\alpha_i)=\alpha_{i_I^*}$
($i\in I$) we have, for $i\in I$,
\[(\overline{w_0}\gamma^{-1})^{\alpha_{i_I^*}^\vee}=
\gamma^{-\alpha_i^\vee}=k_i^{-2}=k_{i_I^*}^{-2}.
\]
Hence $\overline{w_0}\gamma^{-1}\in T_{I^*}^{k^{-1}}$.

The second statement follows from Lemma \ref{chilemma}{\bf (i)}.
\end{proof}

\begin{defi}
Let $\gamma\in T_I^{k^{\pm 1}}$. For 
a left $\mathbb{C}_{\sigma,\nabla}^{k,q}[T]\#_qW$-module $L$ we
define
\begin{equation*}
\begin{split}
L_{\pi,a}^{I^*,\pm}[\overline{w_0}\gamma^{-1}]&:=\{\phi\in L \,\, | \,\, 
\pi^{k^{-1},q}(h)\phi=\chi_{\overline{w_0}\gamma^{-1}}^{k^{-1},\pm,I^*}(h)\phi
\quad\forall h\in H_{I^*}^a(k^{-1})\},\\
L_{\pi}^{I^*,\pm}[\overline{w_0}\gamma^{-1}]&:=\{\phi\in L \,\, | \,\, 
\pi^{k^{-1},q}(h)\phi=\chi_{\overline{w_0}\gamma^{-1}}^{k^{-1},\pm,I^*}(h)\phi
\quad\forall h\in H_{I^*}(k^{-1})\}.
\end{split}
\end{equation*}
\end{defi}
Note the alternative description  
\[
L_{\pi}^{I^*,\pm}[\overline{w_0}\gamma^{-1}]=
\{\phi\in L_\pi^{I^*,\pm} \,\, | \,\,
\pi^{k^{-1},q}(f(Y))\phi=f(\overline{w_0}\gamma^{-1})\phi\quad
\forall f\in\mathbb{C}[T] \},
\]
which emphasizes the fact that $L_\pi^{I^*,\pm}[\overline{w_0}\gamma^{-1}]$
consists of
common eigenfunctions within $L_\pi^{I^*,\pm}$
of the commuting Cherednik-Dunkl operators
$\pi^{k^{-1},q}(f(Y))$ ($f\in\mathbb{C}[T]$), 
with associated spectral
point $\overline{w_0}\gamma^{-1}$. 
It is convenient to make explicit 
contact with the notations of Subsection \ref{Intersection}.
We write $L_\pi$ for a left $\mathbb{C}_{\sigma,\nabla}^{k,q}[T]\#_qW$-module $L$
when we view $L$ as a $H(k^{-1})$-module via the algebra map
$\pi^{k^{-1},q}: H(k^{-1})\rightarrow \mathbb{C}_{\sigma,\nabla}^{k,q}[T]\#_qW$.
Then we can alternatively write for $\gamma\in T_I^{k^{\pm 1}}$,
\begin{equation*}
L_\pi^{I^*,\pm}[\overline{w_0}\gamma^{-1}]
=L_\pi^{I^*,\pm}\cap L_{\pi,\overline{w_0}\gamma^{-1}},
\end{equation*}
where, recall,  
$L_{\pi,\overline{w_0}\gamma^{-1}}$
is the $\mathcal{A}_Y^{k^{-1}}$-weight space of the $H(k^{-1})$-module
$L_\pi$ 
with weight $\overline{w_0}\gamma^{-1}\in T$,
\[L_{\pi,\overline{w_0}\gamma^{-1}}=
\{\phi\in L \,\, | \,\, \pi^{k^{-1},q}(f(Y))\phi=f(\overline{w_0}\gamma^{-1})\phi
\quad \forall\, f\in\mathbb{C}[T] \}.
\]
In particular, for $I^*=\emptyset$
we have
\[L_\pi^{\emptyset,\pm}[\overline{w_0}\gamma^{-1}]=L_{\pi,\overline{w_0}\gamma^{-1}}.
\]
Similar alternative descriptions
can be given of the space $L_{\pi,a}^{I^*,\pm}[\overline{w_0}\gamma^{-1}]$.

\begin{prop}\label{Wacase}
Let $\gamma\in T_I^{k^{\pm 1}}$. Let $L$
be a left $\mathbb{C}_{\sigma,\nabla}^{k,q}[T]\#_qW$-module. The
map $\phi\mapsto \psi_\phi$, given by \eqref{isoY2}, restricts
to an isomorphism
\[
L_{\pi,a}^{I^*,\pm}[\overline{w_0}\gamma^{-1}]
\overset{\sim}{\longrightarrow}
\Gamma_L^{k,q}\bigl(M^{k,\pm,I}(\gamma)\bigr)_\nabla^{W^a}.
\]
\end{prop}
\begin{proof}
Let $\gamma\in T_I^{k^{\pm 1}}$ and
$\psi\in\Gamma_L^{k,q}\bigl(M^{k,\pm,I}(\gamma)\bigr)_\nabla^{W_0}$,
written as
\[
\psi=\sum_{w\in W_0^I}\psi_w\otimes v_w^{k,\pm,I}(\gamma),
\]
with $\psi_w=\pi^{k^{-1},q}\bigl(T_{w\overline{w}_0^{-1}}\bigr)\phi$
and $\phi\in L_\pi^{I^*,\pm}$, cf. Corollary \ref{coreq}.
Then we have $\psi\in\Gamma_L^{k,q}\bigl(M^{k,\pm,I}(\gamma)\bigr)_\nabla^{W^a}$
if and only if
\[
\sum_{w\in W_0^I}\pi^{k^{-1},q}(T_0)\psi_w\otimes v_w^{k,\pm,I}(\gamma)=
\sum_{w\in W_0^I}\psi_w\otimes T_0^{-1}v_w^{k,\pm,I}(\gamma)
\]
in view of Proposition \ref{aa}.
For $w\in W_0^I$ we write $s_\varphi w=\overline{s_{\varphi}w}w_\varphi$
with, as usual, $\overline{s_{\varphi}w}\in W_0^I$ the minimal coset
representative of $s_{\varphi}wW_{0,I}$, and $w_\varphi
=\underline{s_{\varphi}w}\in W_{0,I}$.
By \eqref{YT0} and the definition of $v_w^{k,\pm,I}(\gamma)$ we have
for $w\in W_0^I$,
\begin{equation*}
T_0^{-1}v_w^{k,\pm,I}(\gamma)=
\begin{cases}
\epsilon_{\pm}^k(T_{w_{\varphi}})\gamma^{-w^{-1}(\varphi^\vee)}
v_{\overline{s_{\varphi}w}}^{k,\pm,I}(\gamma) &\hbox{if } w^{-1}\varphi\in
R_0^+,\\
\epsilon_{\pm}^k(T_{w_{\varphi}})\gamma^{-w^{-1}(\varphi^\vee)}
v_{\overline{s_{\varphi}w}}^{k,\pm,I}(\gamma)+(k_0^{-1}-k_0)v_w^{k,\pm,I}(\gamma)
&\hbox{if } w^{-1}\varphi\in R_0^-.
\end{cases}
\end{equation*}
Note that $w\mapsto \overline{s_{\varphi}w}$ defines an involution
$W_0^I\overset{\sim}{\longrightarrow}W_0^I$, and that
$(\overline{s_{\varphi}w})_\varphi=w_\varphi^{-1}$ for all $w\in W_0^I$.
It follows that 
$\psi\in\Gamma_L^{k,q}\bigl(M^{k,\pm,I}(\gamma)\bigr)_\nabla^{W^a}$
if and only if
\begin{equation*}
\begin{split}
\sum_{w\in W_0^I}\pi^{k^{-1},q}(T_0)\psi_w\otimes v_w^{k,\pm,I}(\gamma)&=
\sum_{w\in W_0^I}\epsilon_{\pm}^k(T_{w_{\varphi}})
\gamma^{w_{\varphi}w^{-1}(\varphi^\vee)}\psi_{\overline{s_{\varphi}w}}\otimes
v_w^{k,\pm,I}(\gamma)\\
&+(k_0^{-1}-k_0)\sum_{w\in W_0^I: w^{-1}\varphi\in R_0^-}
\psi_w\otimes v_w^{k,\pm,I}(\gamma),
\end{split}
\end{equation*}
which holds true if and only if
\[
\pi^{k^{-1},q}\bigl(T_0^{\sigma(w^{-1}\varphi)}\bigr)\psi_w=
\epsilon_{\pm}^k(T_{w_{\varphi}})\gamma^{w_{\varphi}w^{-1}(\varphi^\vee)}
\psi_{\overline{s_{\varphi}w}},\qquad \forall\, w\in W_0^I.
\]
By \eqref{YT0} and using $T_{u^{-1}}^{-1}T_{w_0}=T_{uw_0}$ for
$u\in W_0$,
\begin{equation*}
\begin{split}
T_0^{\sigma(w^{-1}\varphi)}T_{w^{-1}}^{-1}T_{\overline{w_0}^{-1}}&=
T_0^{\sigma((s_\varphi ww_0)^{-1}\varphi)}T_{ww_0}T_{w_0}^{-1}T_{\overline{w_0}^{-1}}\\
&=T_{w^{-1}s_{\varphi}}^{-1}T_{w_0}Y^{-w_0w^{-1}(\varphi^\vee)}
T_{w_0}^{-1}T_{\overline{w_0}^{-1}}\\
&=
T_{(\overline{s_{\varphi}w})^{-1}}^{-1}T_{w_{\varphi}^{-1}}^{-1}
T_{\underline{w_{0}}}T_{\overline{w_0}^{-1}}
Y^{-w_0w^{-1}(\varphi^\vee)}
T_{\overline{w_0}^{-1}}^{-1}T_{\underline{w_{0}}}^{-1}T_{\overline{w_0}^{-1}}.
\end{split}
\end{equation*}
Since $\pi^{k^{-1},q}(T_{\overline{w_0}^{-1}})\phi\in L_\pi^{I,\pm}$,
we obtain for $w\in W_0^I$,
\begin{equation*}
\begin{split}
\pi^{k^{-1},q}\bigl(T_0^{\sigma(w^{-1}\varphi)}\bigr)\psi_w&=
\pi^{k^{-1},q}\bigl(
T_0^{\sigma(w^{-1}\varphi)}T_{w^{-1}}^{-1}T_{\overline{w_0}^{-1}}\bigr)\phi\\
&=\epsilon_{\pm}^k(T_{\underline{w_{0}}})\pi^{k^{-1},q}\bigl(
T_{(\overline{s_{\varphi}w})^{-1}}^{-1}T_{w_{\varphi}^{-1}}^{-1}
T_{\underline{w_{0}}}T_{\overline{w_0}^{-1}}
Y^{-w_0w^{-1}(\varphi^\vee)}\bigr)\phi,
\end{split}
\end{equation*}
while
\[
\epsilon_{\pm}^k(T_{w_{\varphi}})\gamma^{w_{\varphi}w^{-1}(\varphi^\vee)}
\psi_{\overline{s_{\varphi}w}}=
\epsilon_{\pm}^k(T_{w_{\varphi}})\gamma^{w_{\varphi}w^{-1}(\varphi^\vee)}
\pi^{k^{-1},q}\bigl(T_{(\overline{s_{\varphi}w})^{-1}}^{-1}T_{\overline{w_0}^{-1}}
\bigr)\phi.
\]
Hence $\psi\in \Gamma_L^{k,q}\bigl(M^{k,\pm,I}(\gamma)\bigr)_{\nabla}^{W^a}$
if and only if, for all $w\in W_0^I$,
\[
\pi^{k^{-1},q}
\bigl(Y^{-w_0w^{-1}(\varphi^\vee)}\bigr)\phi=
\gamma^{w_{\varphi}w^{-1}(\varphi^\vee)}\phi.
\]
Here we have used repeatedly that
$\pi^{k^{-1},q}(T_{\overline{w_0}^{-1}})\phi\in L_\pi^{I,\pm}$.

Note that $\{ww_0\}_{w\in W_0^I}$ is a complete set of coset
representatives of $W_0/W_{0,I^*}$. In view of Lemma \ref{generateI}
it thus remains to show that for $\gamma\in T_I^k$,
\[\gamma_I^{-w_0w^{-1}(\varphi^\vee)}=\gamma^{w_{\varphi}w^{-1}(\varphi^\vee)}
\qquad (w\in W_0^I), 
\]
with $\gamma_I\in T_{I^*}^{k^{-1}}$ given by
\begin{equation}\label{gammaI}
\gamma_I:=\overline{w_0}\gamma^{-1}=
w_0(\rho_I^{k}\gamma)^{-1}
\end{equation}
(see Lemma \ref{dualgamma}),
where the last equality follows from \eqref{W0Iactcons}.
Hence it remains to show for $\gamma\in T_I^k$,
\[(\rho_I^k)^{w^{-1}(\varphi^\vee)}\gamma^{w^{-1}(\varphi^\vee)}=
\gamma^{w_{\varphi}w^{-1}(\varphi^\vee)}\qquad (w\in W_0^I).
\]
Let $w\in W_0^I$. Since $\gamma\in T_I^k$ and $w_{\varphi}\in W_{0,I}$
we have 
\[w_{\varphi}^{-1}\gamma=\gamma\prod_{\alpha\in R_0^{I,+}\cap
w_{\varphi}^{-1}R_0^{I,-}}k_{\alpha}^{-2\alpha}
\]
in $T$ by \eqref{W0Iact}, hence it suffices to show that for all $w\in W_0^I$,
\[
(\rho_I^k)^{w^{-1}(\varphi^\vee)}=\prod_{\alpha\in R_0^{I,+}\cap 
w_{\varphi}^{-1}R_0^{I,-}}k_{\alpha^\vee}^{-2\langle\alpha,w^{-1}\varphi^\vee\rangle},
\]
i.e. that $\langle \alpha,w^{-1}\varphi^\vee\rangle=0$ if 
$\alpha\in R_0^{I,+}\cap w_{\varphi}^{-1}R_0^{I,+}$ and
$w\in W_0^I$. This follows
from Lemma \ref{following} below.
\end{proof}
\begin{lem}\label{following}
Let $w\in W_0^I$ and write $s_{\varphi}w=\overline{s_{\varphi}w}w_{\varphi}$
with $\overline{s_{\varphi}w}\in W_0^I$ and $w_{\varphi}=
\underline{s_{\varphi}w}\in W_{0,I}$.
Then
\begin{equation*}
\langle w\alpha,\varphi^\vee\rangle=
\begin{cases}
1\quad &\hbox{ if }\,\, \alpha\in R_0^{I,+}\cap w_{\varphi}^{-1}R_0^{I,-},\\
0\quad &\hbox{ if }\,\, \alpha\in R_0^{I,+}\cap w_{\varphi}^{-1}R_0^{I,+}.
\end{cases}
\end{equation*}
\end{lem}
\begin{proof}
Let $\alpha\in R_0^{I,+}$. Then $w\alpha\in R_0^+$ since $w\in W_0^I$.
Using
\[s_{\varphi}(w\alpha)=w\alpha-\langle w\alpha,\varphi^\vee\rangle\varphi
\]
and using the fact that
$\varphi\in R_0^+$ is the longest root, we conclude that
$\langle w\alpha,\varphi^\vee\rangle$ is $0$ or $1$.
It is $0$ if $s_{\varphi}(w\alpha)\in R_0^+$ and $1$ if
$s_{\varphi}(w\alpha)\in R_0^-$.
Furthermore,
\[R_0^+\cap (s_{\varphi}w)^{-1}R_0^-=
w_{\varphi}^{-1}\bigl(R_0^+\cap (\overline{s_{\varphi}w})^{-1}R_0^-\bigr)
\cup \bigl(R_0^{I,+}\cap w_{\varphi}^{-1}R_0^{I,-}\bigr)
\]
since $l(\overline{s_{\varphi}w}w_{\varphi})=l(\overline{s_{\varphi}w})+
l(w_{\varphi})$ and $w_{\varphi}\in W_{0,I}$.

Suppose that $\alpha\in R_0^{I,+}\cap w_{\varphi}^{-1}R_0^{I,-}$.
Then $\alpha\in R_0^+\cap (s_{\varphi}w)^{-1}R_0^-$, hence
$s_{\varphi}(w\alpha)\in R_0^-$. Consequently, 
$\langle w\alpha,\varphi^\vee\rangle=1$.

Suppose that $\alpha\in R_0^{I,+}\cap w_{\varphi}^{-1}R_0^{I,+}$.
Then 
\[
s_{\varphi}(w\alpha)=\overline{s_{\varphi}w}w_{\varphi}(\alpha)
\in \overline{s_{\varphi}w}\bigl(R_0^{I,+}\bigr)
\subseteq R_0^+,
\]
hence $\langle w\alpha,\varphi^\vee\rangle=0$.
\end{proof}
The following theorem provides an explicit bridge between the theory
of quantum affine KZ equations and the Cherednik-Macdonald theory.
\begin{thm}\label{mainthmY}
Let $\gamma\in T_I^{k^{\pm 1}}$.
Let $L$ be a left $\mathbb{C}_{\sigma,\nabla}^{k,q}[T]\#_qW$-module.
The map $\phi\mapsto \psi_\phi$, given by \eqref{isoY2},
restricts to a linear isomorphism
\[L_\pi^{I^*,\pm}[\overline{w_0}\gamma^{-1}]\overset{\sim}{\longrightarrow}
\Gamma_L^{k,q}\bigl(M^{k,\pm,I}(\gamma)\bigr)_\nabla^{W}.
\]
\end{thm}
\begin{proof}
Let $\gamma\in T_I^{k^{\pm 1}}$ and
$\psi\in\Gamma_L^{k,q}\bigl(M^{k,\pm,I}(\gamma)\bigr)_\nabla^{W^a}$,
written as 
\[
\psi=\sum_{w\in W_0^I}\psi_w\otimes v_w^{k,\pm,I}(\gamma)
\]
with $\psi_w\in \pi^{k^{-1},q}\bigl(T_{w\overline{w_0}^{-1}}\bigr)\phi$
and $\phi\in L_{\pi,a}^{I^*,\pm}[\overline{w_0}\gamma^{-1}]$, 
cf. Proposition \ref{Wacase}.
Then $\psi\in\Gamma_L^{k,q}\bigl(M^{k,\pm,I}(\gamma)\bigr)_{\nabla}^W$ if and only
if
\begin{equation}\label{todoto}
\sum_{w\in W_0^I}\pi^{k^{-1},q}(T_{\omega})\psi_w\otimes
v_w^{k,\pm,I}(\gamma)=\sum_{w\in W_0^I}\psi_w\otimes 
T_{\omega}^{-1}v_w^{k,\pm,I}(\gamma)
\end{equation}
for all $\omega\in\Omega$ in view of Proposition \ref{aa}.

In order to analyze \eqref{todoto}, we first need to give some
additional properties of the abelian subgroup $\Omega\subset W$
(we refer to \cite{M} for detailed proofs of the following  
facts).
For $\lambda\in P^\vee$ we write $v(\lambda)$ for the unique
element of minimal length in $W_0$ such that 
$v(\lambda)\lambda\in w_0P_+^\vee$. We furthermore set $u(\lambda)\in W$
such that $t(\lambda)=u(\lambda)v(\lambda)$ in $W$. Set
$\varpi_0^\vee\in 0$. Then we write, for $0\leq j\leq n$,
$u_j=u(\varpi_j^\vee)\in W$ and $v_j=v(\varpi_j^\vee)$. In particular,
$u_0=e$ and $v_0=e$.
Set 
\[
J:=\{0\}\cup \{j\in \{1,\ldots,n\} \,\, | \,\,
\langle \varpi_j^\vee,\varphi\rangle=1\}.
\]
Then $\Omega=\{u_j\}_{j\in J}$. We write $\{U_j\}_{j\in J}$ for the
corresponding elements in the extended affine Hecke algebra $H(k)$. Then
in $H(k)$,
\begin{equation}\label{Uj}
U_j=T_wY^{w^{-1}(\varpi_j^\vee)}T_{v_jw}^{-1}\qquad (j\in J,\,
w\in W_0),
\end{equation}
cf. \cite[(3.3.3)]{M}.

Choose $\sigma_j\in W_0$ ($j\in J$) arbitrarily.
Then $\{\sigma_j(\varpi_j^\vee)\}_{j\in J}$
is a complete set of representatives of $P^\vee/Q^\vee$.
In particular, the standard parabolic subalgebra $H_I(k)$ of $H(k)$
is generated by $H_I^a(k)$ and the $Y^{\sigma_j(\varpi_j^\vee)}$ ($j\in J$).

We are now set to analyze \eqref{todoto}. Let $j\in J$ and $w\in W_0^I$,
then it follows from \eqref{Uj} that
\[U_j^{-1}v_w^{k,\pm,I}(\gamma)=\epsilon_{\pm}^k(T_{w_j})
\gamma^{-w^{-1}(\varpi_j^\vee)}v_{\overline{v_jw}}^{k,\pm,I}(\gamma),
\]
where we have written $v_jw=(\overline{v_jw})w_j$ with 
$\overline{v_jw}\in W_0^I$ and $w_j:=\underline{v_jw}\in W_{0,I}$.
Note that $w\mapsto \overline{v_jw}$ defines a bijection
$W_0^I\rightarrow W_0^I$ with inverse given by
$w\mapsto \overline{v_j^{-1}w}$ ($w\in W_0^I$).
We write $w_j^\prime:=\underline{v_j^{-1}w}\in W_{0,I}$, such that
$v_j^{-1}w=(\overline{v_j^{-1}w})w_j^\prime$. Note that
\begin{equation}\label{rell}
(\overline{v_j^{-1}w})_j=w_j^\prime{}^{-1},\qquad w\in W_0^I
\end{equation}
since $v_j(\overline{v_j^{-1}w})=ww_j^\prime{}^{-1}$.
Thus we conclude that
\[\sum_{w\in W_0^I}\psi_w\otimes U_j^{-1}v_w^{k,\pm,I}(\gamma)=
\sum_{w\in W_0^I}\epsilon_{\pm}^k\bigl(T_{w_j^\prime}\bigr)
\gamma^{-(\overline{v_j^{-1}w})^{-1}(\varpi_j^\vee)}
\psi_{\overline{v_j^{-1}w}}\otimes v_w^{k,\pm,I}(\gamma).
\]
Hence \eqref{todoto} holds true for all $\omega\in\Omega$ if and only
if, for all $w\in W_0^I$ and $j\in J$,
\begin{equation}\label{equival}
\pi^{k^{-1},q}(U_j)\psi_w=
\epsilon_{\pm}^k(T_{w_j^\prime})
\gamma^{-(\overline{v_j^{-1}w})^{-1}(\varpi_j^\vee)}
\psi_{\overline{v_j^{-1}w}}.
\end{equation}
As in the proof of Proposition \ref{Wacase}, it follows from
\eqref{Uj} and $\psi_w=\pi^{k^{-1},q}(T_{w^{-1}}^{-1}T_{\overline{w_0}^{-1}})\phi$
that for all $w\in W_0^I$,
\[\pi^{k^{-1},q}(U_j)\psi_w=
\pi^{k^{-1},q}\bigl(T_{(\overline{v_j^{-1}w})^{-1}}^{-1}T_{w_j^\prime{}^{-1}}^{-1}
T_{\underline{w_{0}}}T_{\overline{w_0}^{-1}}
Y^{w_0w^{-1}v_j(\varpi_j^\vee)}T_{\overline{w_0}^{-1}}^{-1}T_{\underline{w_{0}}}^{-1}
T_{\overline{w_0}^{-1}}\bigr)\phi.
\]
Using that $\phi\in L_{\pi,a}^{I^*,\pm}[\overline{w_0}\gamma^{-1}]$, in particular
$\pi^{k^{-1},q}\bigl(T_{\overline{w_0}^{-1}}\bigr)\phi\in L_\pi^{I,\pm}$,
we get $\psi\in \Gamma_L^{k,q}\bigl(M^{k,\pm,I}(\gamma)\bigr)_\nabla^W$
iff \eqref{equival} holds true for all $w\in W_0^I$ and $j\in J$,
iff
\[\pi^{k^{-1},q}\bigl(Y^{w_0w^{-1}v_j(\varpi_j^\vee)}\bigr)\phi=
\gamma^{-(\overline{v_j^{-1}w})^{-1}(\varpi_j^\vee)}\phi
\quad \forall\, w\in W_0^I,\,\, \forall\, j\in J.
\]
It thus remains to prove that for $\gamma\in T_I^k$,
\begin{equation}\label{equivall}
\gamma^{-(\overline{v_j^{-1}w})^{-1}(\varpi_j^\vee)}=
\gamma_I^{w_0w^{-1}v_j(\varpi_j^\vee)}\quad \forall\, w\in W_0^I,\,\,
\forall\, j\in J,
\end{equation}
with $\gamma_I$ given by \eqref{gammaI}.

Now \eqref{equivall} for $\gamma\in T_I^k$ is equivalent to
\[\gamma^{(\overline{v_j^{-1}w})^{-1}(\varpi_j^\vee)}=
(\rho_I^k)^{w^{-1}v_j(\varpi_j^\vee)}\gamma^{w^{-1}v_j(\varpi_j^\vee)}
\quad \forall\, w\in W_0^I,\,\, \forall\, j\in J.
\]
Since $\gamma\in T_I^k$ and $w^{-1}v_j=
w_j^\prime{}^{-1}(\overline{v_j^{-1}w})^{-1}$ for $w\in W_0^I$
we have by \eqref{W0Iact},
\[\gamma^{w^{-1}v_j(\varpi_j^\vee)}=
\gamma^{(\overline{v_j^{-1}w})^{-1}(\varpi_j^\vee)}
\prod_{\alpha\in R_0^{I,+}\cap w_j^\prime R_0^{I,-}}
k_{\alpha^\vee}^{-2\langle\alpha,(\overline{v_j^{-1}w})^{-1}(\varpi_j^\vee)\rangle},
\]
while, by the definition of $\rho_I^k$,
\[(\rho_I^k)^{w^{-1}v_j(\varpi_j^\vee)}=
\prod_{\alpha\in R_0^{I,+}}k_{\alpha^\vee}^{-2\langle\alpha,w^{-1}v_j(\varpi_j^\vee)
\rangle}=
\prod_{\alpha\in w_j^\prime(R_0^{I,-})}k_{\alpha^\vee}^{2\langle
\alpha,(\overline{v_j^{-1}w})^{-1}(\varpi_j^\vee)\rangle}.
\]
Hence \eqref{equivall} holds true iff for all $w\in W_0^I$ and $j\in J$,
\begin{equation}\label{last}
\langle\alpha,(\overline{v_j^{-1}w})^{-1}(\varpi_j^\vee)\rangle=0,
\quad \forall\, \alpha\in R_0^{I,-}\cap w_j^\prime(R_0^{I,-}).
\end{equation}
Fix $w\in W_0^I$ and $j\in J$. We show in fact that
\begin{equation}\label{lastl}
\langle \alpha,(\overline{v_j^{-1}w})^{-1}(\varpi_j^\vee)\rangle
\begin{cases}
\leq 0,\quad &\forall\, \alpha\in R_0^{I,-},\\
\geq 0,\quad &\forall\, \alpha\in w_j^\prime(R_0^{I,-}),
\end{cases}
\end{equation}
which implies \eqref{last}. The first inequality is immediate, since
$u(R_0^{I,-})\subseteq R_0^-$ if $u\in W_0^I$. For the second equality,
let $\alpha=w_j^\prime(\beta)\in w_j^\prime(R_0^{I,-})$.
Then, since $(\overline{v_j^{-1}w})w_j^\prime=v_j^{-1}w$,
\[\langle \alpha,(\overline{v_j^{-1}w})^{-1}(\varpi_j^\vee)\rangle=
\langle w(\beta),v_j(\varpi_j^\vee)\rangle.
\]
But $w\in W_0^I$ and $\beta\in R_0^{I,-}$ hence $w(\beta)\in R_0^-$,
and $v_j(\varpi_j^\vee)\in w_0P_+^\vee$, hence
$\langle w(\beta),v_j(\varpi_j^\vee)\rangle\geq 0$. This concludes
the proof.
\end{proof}

\begin{cor}\label{alright}
Let $\gamma\in T_I^{k^{\pm 1}}$ and let $L$ a left 
$\mathbb{C}_{\sigma,\nabla}^{k,q}[T]\#_qW$-module.
The we have an injective linear map
\[
\alpha_{L,\gamma}^{k,q,\pm,I}: \Gamma_L^{k,q}\bigl(M^{k,\pm,I}(\gamma)\bigr)_\nabla^W
\hookrightarrow \Gamma_L^{k,q}\bigl(M^k(\underline{w_0}\gamma)\bigr)_\nabla^W,
\]
defined by
\[
\alpha_{L,\gamma}^{k,q,\pm,I}\bigl(\sum_{w\in W_0^I}\psi_w\otimes
v_w^{k,\pm,I}(\gamma)\bigr)=
\sum_{u\in W_0^I}\psi_u\otimes\Bigl(\sum_{v\in W_{I,0}}
\epsilon_{\pm}^k(T_v)v_{uv}^k(\underline{w_0}\gamma)\Bigr).
\]
\end{cor}
\begin{proof}
We claim that 
$\epsilon_{\pm}^k(T_{\underline{w_0}})^{-1}\alpha_{L,\gamma}^{k,q,\pm,I}$
is equal to the composition of maps
\[
\Gamma_L^{k,q}\bigl(M^{k,\pm,I}(\gamma)\bigr)_\nabla^W
\overset{\sim}{\longrightarrow}
L_\pi^{I^*,\pm}[\overline{w_0}\gamma^{-1}]\hookrightarrow
L_{\pi,\overline{w_0}\gamma^{-1}}\overset{\sim}{\longrightarrow}
\Gamma_L^{k,q}\bigl(M^k(\underline{w_0}\gamma)\bigr)_\nabla^W,
\]
with the second the trivial inclusion, and the first and third isomorphisms
obtained from Theorem \ref{mainthmY}. This can be proved by a direct 
computation, using the fact that
\[\pi^{k^{-1},q}(T_{w^{-1}}^{-1}T_{w_0})\phi=
\epsilon_{\pm}^k(T_{\underline{w_0}})^{-1}\epsilon_{\pm}^k(T_{\underline{w}})
\pi^{k^{-1},q}(T_{\overline{w}^{-1}}^{-1}T_{\overline{w_0}^{-1}})\phi
\]
for $\phi\in L_\pi^{I^*,\pm}$ and $w\in W_0$. 
\end{proof}
Let $\gamma\in T_I^{k^{\pm 1}}$ and $L$ a left 
$\mathbb{C}_{\sigma,\nabla}^{k,q}[T]\#_qW$-module.
The functoriality of the assignment \eqref{functoriality} can be
applied to the surjective morphism $M^{k}(\gamma)\rightarrow
M^{k,\pm,I}(\gamma)$ of $H(k)$-modules mapping $v_e^{k}(\gamma)$
to $v_e^{k,\pm,I}(\gamma)$. In fact, it maps $v_{w}^{k}(\gamma)$
to $\epsilon_{\pm}^k(T_{\underline{w}})v_{\overline{w}}^{k,\pm,I}(\gamma)$ 
for $w\in W_0$. It thus gives rise to a $\mathbb{C}[W]$-linear map
\begin{equation}\label{funtapp}
\begin{split}
\Gamma_L^{k,q}(M^{k}(\gamma))_\nabla&\rightarrow 
\Gamma_L^{k,q}(M^{k,\pm,I}(\gamma))_\nabla,\\
\sum_{w\in W_0}\psi_w\otimes v_w^{k}(\gamma)&\mapsto
\sum_{u\in W_0^I}\Bigl(\sum_{v\in W_{0,I}}\epsilon_{\pm}^k(T_v)\psi_{uv}\Bigr)
\otimes v_u^{k,\pm,I}(\gamma)
\end{split}
\end{equation}
for $\gamma\in T_I^{k^{\pm 1}}$.
Combined with Theorem \ref{mainthmY} we obtain
\begin{cor}
Let $L$ be a left $\mathbb{C}_{\sigma,\nabla}^{k,q}[T]\#_qW$ and
let $\gamma\in T_I^{k^{\pm 1}}$. The assignment
\[
\phi\mapsto 
\sum_{v\in W_{0,I^*}}\epsilon_{\pm}^{k^{-1}}(T_v)\pi^{k^{-1},q}(T_v)\phi
\]
defines a linear map
\[
L_{\pi,w_0\gamma^{-1}}\rightarrow L_\pi^{I^*,\pm}[\overline{w_0}\gamma^{-1}].
\]
\end{cor}
\begin{proof}
Let $\gamma\in T_I^{k^{\pm 1}}$. Denote by
\[\eta_\gamma^I: 
L_{\pi,w_0\gamma^{-1}}\rightarrow L_\pi^{I^*,\pm}[\overline{w_0}\gamma^{-1}]
\]
the composition of the linear maps
\[L_{\pi,w_0\gamma^{-1}}\overset{\sim}{\longrightarrow}
\Gamma_L^{k,q}(M^{k}(\gamma))_\nabla^W\rightarrow
\Gamma_L^{k,q}(M^{k,\pm,I}(\gamma))_\nabla^W\overset{\sim}{\longrightarrow}
L_\pi^{I^*,\pm}[\overline{w_0}\gamma^{-1}],
\]
where the isomorphisms are from Theorem \ref{mainthmY}, and the
second map is the restriction of \eqref{funtapp} to 
$\Gamma_L^{k,q}(M^{k}(\gamma))_\nabla^W$.
Then
\[\eta_\gamma^I(\phi)=\sum_{v\in W_{0,I}}\epsilon_{\pm}^k(T_v)
\pi^{k^{-1},q}(T_{\overline{w_0}^{-1}}^{-1}T_{v^{-1}}^{-1}T_{w_0})\phi
\]
for $\phi\in L_{\pi,w_0\gamma^{-1}}$.
It suffices to show that
\begin{equation}\label{todoYY}
\eta_\gamma^I(\phi)=\epsilon_{\pm}^k(T_{\underline{w_0}})
\sum_{v\in W_{0,I^*}}\epsilon_{\pm}^{k^{-1}}(T_v)\pi^{k^{-1},q}(T_v)\phi
\qquad (\phi\in L_{\pi,w_0\gamma^{-1}}).
\end{equation}
Let $\phi\in L_{\pi,w_0\gamma^{-1}}$.
Note that $W_{0,I}=\overline{w_0}^{-1}W_{0,I^*}\overline{w_0}$,
and
\[T_{\overline{w_0}^{-1}v\overline{w_0}}=T_{\overline{w_0}^{-1}}T_v
T_{\overline{w_0}^{-1}}^{-1}\qquad (v\in W_{0,I^*}).
\]
It follows that
\[\eta_\gamma^I(\phi)=\sum_{v\in W_{0,I^*}}
\epsilon_{\pm}^k(T_v)\pi^{k^{-1},q}(T_{v^{-1}}^{-1}T_{\overline{w_0}^{-1}}^{-1}T_{w_0})
\phi.
\]
Denote $\underline{w_0}^*$ for the longest Weyl group element of $W_{0,I^*}$.
Then $\underline{w_0}^*=\overline{w_0}\underline{w_0}\overline{w_0}^{-1}$, hence
\[T_{\overline{w_0}^{-1}}^{-1}T_{w_0}=T_{\overline{w_0}w_0}=
T_{\underline{w_0}^*}.
\]
It follows that 
\[\eta_\gamma^I(\phi)=\sum_{v\in W_{0,I^*}}\epsilon_{\pm}^k(T_v)
\pi^{k^{-1},q}(T_{v\underline{w_0}^*})\phi.
\]
Using that $T_{v\underline{w_0}^*}=T_{v^{-1}}^{-1}T_{\underline{w_0}^*}$
for $v\in W_{0,I^*}$, we get \eqref{todoYY}.
\end{proof}

\subsection{$\textup{GL}_m$ case}\label{GL}

It is instructive to consider the $\textup{GL}_{m}$ case of 
Theorem \ref{mainthmY}, since it has a simpler proof.
We keep the notations and definitions as before with $R_0$
the root system of type $A_{m-1}$ ($m\geq 2$). 
We redefine first, for the duration of this subsection, 
those notions from the previous subsections which need
slight modifications in the $\textup{GL}_m$ setup.

Let $\{\epsilon_i\}_{i=1}^m$ be an orthonormal basis of $\mathbb{R}^m$. 
Take $\{\epsilon_i-\epsilon_j\}_{1\leq i<j\leq m}$
as the realization of the root system $R_0$, and take
$\alpha_j=\epsilon_j-\epsilon_{j+1}$ ($1\leq j\leq n=m-1$)
as a basis of $R_0$. The affine root system is $R=\mathbb{Z}c+R_0$
with additional affine simple root $a_0=c-\epsilon_1+\epsilon_m$.

Set $T=\bigl(\mathbb{C}\setminus\{0\}\bigr)^m$
and define for $\lambda\in\mathbb{Z}^m$ and $r\in\mathbb{Z}$,
the monomial $e_q^{rc+\lambda}\in\mathbb{C}[T]$ by
\[e_q^{rc+\lambda}(t)=q^rt^\lambda,\qquad t\in T
\]
(where we use the usual multi-index notations).
The extended affine Weyl group is $W=S_m\ltimes\mathbb{Z}^m$.
A multiplicity function $R\rightarrow \mathbb{C}^\times$
(i.e. $W$-invariant)
takes on a constant value, which we denote by $k\in\mathbb{C}^\times$.
We let $W$ act on $\mathbb{C}[T]$ by $q$-difference reflection operators
(see \eqref{qaction}).

The affine Weyl group $W$ is generated by $s_i=s_{\alpha_i}$
($1\leq i<m$) and $\zeta$,
where 
\[
\zeta=\sigma\tau(\epsilon_m)
\]
with $\sigma=s_1s_2\cdots s_{m-1}\in S_m$. It thus acts on
$f\in\mathbb{C}[T]$ by
\[(\zeta_qf)(t)=f(t_2,\ldots,t_m,q^{-1}t_1).
\]
Let $\Omega\subset W$ be
the subgroup of $W$ generated by $\zeta$. It is the subgroup 
consisting of the affine Weyl group elements $w\in W$ of length zero.

The extended affine Hecke algebra $H(k)$ is generated by 
$T_\zeta$ and $T_1,\ldots,T_{m-1}$ with, besides the familiar
relations for the finite Hecke algebra generators
$\{T_i\}_{1\leq i<m}$ (including the quadratic relations
$(T_i-k)(T_i+k^{-1})=0$), the relations that $T_\zeta^m$ is central in $H(k)$
and that $T_\zeta T_i=T_{i+1}T_\zeta$ for $1\leq i<m-1$. The
commuting $Y_i=Y^{\epsilon_i}\in H(k)$ ($1\leq i\leq m$) are given by
\[
Y_i=T_{i-1}^{-1}\cdots T_2^{-1}T_1^{-1}T_\zeta T_{m-1}T_{m-2}\cdots T_i.
\]
Let $\mathcal{A}_Y^k$ be the commutative subalgebra of $H(k)$
generated by the $Y_i^{\pm 1}$ ($1\leq i\leq m$). 
The extended affine Hecke algebra is generated
by the finite Hecke algebra
$H_0(k)$ (which is generated by $T_1,\ldots,T_{m-1}$) and $\mathcal{A}_Y^k$, 
with defining relations as in Theorem \ref{Hchar}.

The algebra map $\pi^{k,q}: H(k)\rightarrow \mathbb{C}_\sigma^q[T]\#_qW$
is now given by
\[\pi^{k,q}(T_i)=k+c_i^{k,q}(s_i-1),\qquad 
\pi^{k,q}(T_\zeta)=\zeta
\]
for $1\leq i<m$. The algebra map $\nabla^{k,q}: \mathbb{C}[T]\#_qW
\rightarrow \mathcal{A}_\nabla^{k,q}$ is 
\begin{equation*}
\begin{split}
\nabla^{k,q}(f)&=f,\\
\nabla^{k,q}(s_i)&=(c_i^{k,q})^{-1}s_iT_i+\frac{c_i^{k,q}-k}{c_i^{k,q}}s_i,\\
\nabla^{k,q}(\zeta)&=\zeta T_\zeta
\end{split}
\end{equation*}
for $f\in\mathbb{C}[T]$ and $1\leq i<m$.

For $I\subseteq \{1,\ldots,m-1\}$ we write $H_I(k)$ for the subalgebra
generated by $T_i$ ($i\in I$) and $\mathcal{A}_Y^k$.
For $\gamma\in T_I^{k^{\pm 1}}$ we have a character
$\chi_\gamma^{k,\pm,I}: H_I(k)\rightarrow \mathbb{C}$ mapping $T_i$ to 
$\pm k^{\pm 1}$
for $i\in I$ and $f(Y)$ to $f(\gamma)$ for $f\in\mathbb{C}[T]$.
We have the associated affine Hecke algebra module $M^{k,\pm,I}(\gamma)=
\textup{Ind}_{H_I(k)}^{H(k)}\bigl(\chi_\gamma^{k,\pm,I}\bigr)$
with basis $v_w^{k,\pm,I}(\gamma)$ ($w\in S_m^I$), where $S_m^I$
is the set of minimal coset representatives of 
$S_m/S_{m,I}$ (with $S_{m,I}$ the subgroup of $S_m$ generated
by $s_i$ ($i\in I$)). 

Let $w_0\in S_m$ be the longest Weyl group element, mapping
$j$ to $m+1-j$ for $1\leq j\leq m$,
and set
\[\varpi_j=\epsilon_1+\epsilon_2+\cdots+\epsilon_j,\qquad 1\leq j\leq m.
\]
Corollary \ref{coreq} now holds true in the present $\textup{GL}_m$ 
setup. 
\begin{prop}\label{explicitprop}
Let $\gamma\in T_I^{k^{\pm 1}}$ and let
$L$ be a left $\mathbb{C}_{\sigma,\nabla}^{k,q}[T]\#_qW$-module.
Suppose that $\phi\in L_\pi^{I^*,\pm}$ and let $\psi_\phi\in
\Gamma_L^{k,q}\bigl(M^{k,\pm,I}(\gamma)\bigr)_\nabla^{S_m}$ be defined
by \eqref{isoY2}. Then
$\nabla^{k,q}(\tau(\varpi_j))\psi_\phi$ is equal to
\[
\sum_{w\in S_m^I}\epsilon_{\pm}^k(T_{\underline{w_{0}}}T_{w_j^\prime})
\gamma^{w_j^\prime w^{-1}(\varpi_j)}\pi^{k^{-1},q}\bigl(T_{ww_j^\prime{}^{-1}w_0}
Y^{w_0w_j^\prime w^{-1}(\varpi_j)}\bigr)\phi\otimes v_w^{k,I}(\gamma)
\]
for $1\leq j\leq m$, where $w_j^\prime\in S_{m,I}$ is such that
$\sigma^{-j}ww_j^\prime{}^{-1}\in S_m^I$.
\end{prop}
\begin{proof}
Note that $\tau(\varpi_j)=\zeta^j\sigma^{-j}\in W$, hence
\[\nabla^{k,q}(t(\varpi_j))\psi_\phi=\nabla^{k,q}(\zeta^j)\psi_\phi.
\]
For $\phi\in L_\pi^{I^*,\pm}$ we have 
$\pi^{k^{-1},q}(T_{\overline{w_0}^{-1}})\phi\in
L_\pi^{I,\pm}$, hence for all $w\in S_m^I$,
\[\pi^{k^{-1},q}(T_{w\overline{w_0}^{-1}})\phi=
\pi^{k^{-1},q}(T_{w^{-1}}^{-1}T_{\overline{w_0}^{-1}})\phi=
\epsilon_{\pm}^k(T_{\underline{w_{0}}})\pi^{k^{-1},q}(T_{ww_0})\phi.
\]
Thus $\psi_\phi=\sum_{w\in S_m^I}\psi_w\otimes v_w^{k,\pm,I}(\gamma)$ with
$\psi_w=\epsilon_{\pm}^k(T_{\underline{w_{0}}})\pi^{k^{-1},q}(T_{ww_0})\phi$.
Using the fact that 
\begin{equation}\label{Hzeta}
T_\zeta T_u=T_{\sigma u}Y_{u^{-1}(m)},\qquad u\in S_m
\end{equation}
in $H(k)$ (cf., e.g., the proof of \cite[Lemma 4.1]{vMS}), we obtain
\begin{equation*}
\begin{split}
\nabla^{k,q}(\tau(\varpi_j))\psi_\phi&=\nabla^{k,q}(\zeta^j)\psi_\phi\\
&=\sum_{w\in S_m^I}\zeta^j\psi_w\otimes T_\zeta^jv_w^{k,\pm,I}(\gamma)\\
&=\sum_{w\in S_m^I}\epsilon_{\pm}^k(T_{\underline{w_{0}}}T_{w_j})
\gamma^{w^{-1}w_0(\varpi_j)}\pi^{k^{-1},q}(T_\zeta^jT_{ww_0})\phi\otimes
v_{\overline{\sigma^jw}}^{k,\pm,I}(\gamma),
\end{split}
\end{equation*}
where $w_j:=\underline{\sigma^jw}\in S_{m,I}$, such that
$\sigma^jw=(\overline{\sigma^jw})w_j$.

For $w\in S_m^I$ we write 
$w_j^\prime:=\underline{\sigma^{-j}w}\in S_{m,I}$, such that $\sigma^{-j}w=
(\overline{\sigma^{-j}w})w_j^\prime$. Then $w\mapsto \overline{\sigma^jw}$
defines a bijection $S_m^I\overset{\sim}{\longrightarrow} S_m^I$,
with inverse $w\mapsto \overline{\sigma^{-j}w}$ ($w\in S_m^I$).
Furthermore, $(\overline{\sigma^{-j}w})_j=w_j^\prime{}^{-1}$ for 
$w\in S_m^I$. Thus
\begin{equation*}
\begin{split}
&\nabla^{k,q}(\tau(\varpi_j))\psi_\phi=\\
&\qquad=
\sum_{w\in S_m^I} \epsilon_{\pm}^k(T_{\underline{w_{0}}}T_{w_j^\prime})
\gamma^{w_j^\prime w^{-1}\sigma^jw_0(\varpi_j)}\pi^{k^{-1},q}(T_\zeta^j
T_{\sigma^{-j}ww_j^\prime{}^{-1}w_0})\phi\otimes v_w^{k,\pm,I}(\gamma)\\
&\qquad=\sum_{w\in S_m^I} \epsilon_{\pm}^k(T_{\underline{w_{0}}}T_{w_j^\prime})
\gamma^{w_j^\prime w^{-1}(\varpi_j)}\pi^{k^{-1},q}\bigl(
T_{ww_j^\prime{}^{-1}w_0}Y^{w_0w_j^\prime w^{-1}(\varpi_j)}\bigr)\phi\otimes
v_w^{k,\pm,I}(\gamma),
\end{split}
\end{equation*}
where we have used \eqref{Hzeta} and $\sigma^jw_0(\varpi_j)=\varpi_j$
to obtain the last equality.
\end{proof}
\begin{rema}
The classical ($q=1$) analogue of Proposition \ref{explicitprop}
is \cite[Lemma 3.2]{O}. In contrast to the present $q$-setup,
it is for arbitrary root systems. 
\end{rema}

We write $\rho_I^k=(\rho_{I,1}^k,\ldots,\rho_{I,m}^k)\in T$ with
\[
\rho_{I,j}^k=\prod_{\alpha\in R_0^{I,+}}k^{-2\langle\epsilon_j,\alpha\rangle}.
\]
For $\gamma\in T_I^k$ we have $\overline{w_0}\gamma^{-1}=
(w_0\rho_I^k)^{-1}(w_0\gamma)^{-1}\in T_{I^*}^{k^{-1}}$ 
(cf. Lemma \ref{dualgamma}
and \eqref{gammaI}). As before, we define
for $\gamma\in T_I^{k^{\pm 1}}$ and for a left 
$\mathbb{C}_{\sigma,\nabla}^{k,q}[T]\#_qW$-module $L$,
\[L_\pi^{I^*,\pm}[\overline{w_0}\gamma^{-1}]=
\{\phi\in L_\pi^{I^*,\pm} \,\, | \,\, 
\pi^{k^{-1},q}(f(Y))\phi=f(\overline{w_0}\gamma^{-1})\phi\quad 
\forall\, f\in\mathbb{C}[T]\}.
\]
The $\textup{GL}_m$ version of Theorem \ref{mainthmY} which we now
formulate, is essentially due to Kasatani and Takeyama \cite{KT}
(the present version is more general since
we do not need to impose any parameter restraints). 
We give a proof 
based on Proposition \ref{explicitprop}. 
\begin{thm}
Let $\gamma\in T_I^{k^{\pm 1}}$ and let $L$ be a left 
$\mathbb{C}_{\sigma,\nabla}^{k,q}[T]\#_qW$-module.
Then
$\phi\mapsto \psi_\phi$ defines a linear bijection
$L_\pi^{I^*,\pm}[\overline{w_0}\gamma^{-1}]\overset{\sim}{\longrightarrow}
\Gamma_L^{k,q}\bigl(M^{k,\pm,I}(\gamma)\bigr)_\nabla^{W}$.
\end{thm}
\begin{proof}
Let $\gamma\in T_I^{k^{\pm 1}}$ and 
$\phi\in L_\pi^{I^*,\pm}$. Since for $w\in S_m^I$ and $1\leq j\leq m$,
\[T_{ww_j^\prime{}^{-1}w_0}=
T_{w^{-1}}^{-1}T_{w_j^\prime}^{-1}T_{\underline{w_{0}}}T_{\overline{w_0}^{-1}},
\]
it follows from Proposition \ref{explicitprop} and 
$\pi^{k^{-1},q}(T_{\overline{w_0}^{-1}})\phi\in L_\pi^{I,\pm}$
that $\psi_\phi\in\Gamma_L\bigl(M^{k,\pm,I}(\gamma)\bigr)_\nabla^{S_m}$ 
is $W$-invariant if and only of
\[
\pi^{k^{-1},q}\bigl(Y^{w_0w_j^\prime w^{-1}(\varpi_j)}\bigr)\phi=
\epsilon_{\pm}^{k}(T_{w_j^\prime})^{-2}\gamma^{-w_j^\prime 
w^{-1}(\varpi_j)}\phi
\]
for all $w\in S_m^I$ and $1\leq j\leq m$. Since
$\overline{w_0}\gamma^{-1}=w_0(\rho_I^{k^{\pm 1}}\gamma)^{-1}$,
it thus suffices to show that
\[
(\rho_I^k)^{w_j^\prime w^{-1}(\varpi_j)}=\epsilon_+^k(T_{w_j^\prime})^2
\]
for all $1\leq j\leq m$ and $w\in S_m^I$.
Fix $w\in S_m^I$ and $1\leq j\leq m$. Then
\[
(\rho_I^k)^{w_j^\prime w^{-1}(\varpi_j)}=
\prod_{\alpha\in w_j^\prime{}^{-1}(R_0^{I,-})}
k^{2\langle \alpha,w^{-1}(\varpi_j)\rangle},\quad
\epsilon_+^k(T_{w_j^\prime})^2=\prod_{\alpha\in R_0^{I,+}\cap
w_j^\prime{}^{-1}R_0^{I,-}}k^2,
\]
so it remains to show that 
\begin{equation}\label{cl}
\langle \alpha,w^{-1}(\varpi_j)\rangle=
\begin{cases}
1\quad &\hbox{ if }\, \alpha\in R_0^{I,+}\cap w_j^\prime{}^{-1}R_0^{I,-},\\
0\quad &\hbox{ if }\, \alpha\in R_0^{I,-}\cap w_j^\prime{}^{-1}R_0^{I,-}.
\end{cases}
\end{equation}
First note that if $\alpha=w_j^\prime{}^{-1}\beta$ ($\beta\in R_0^{I,-}$)
then 
\[\langle \alpha,w^{-1}(\varpi_j)\rangle=
\langle (\overline{\sigma^{-j}w})(\beta),\sigma^{-j}(\varpi_j)\rangle\geq 0
\]
since $\sigma^{-j}(\varpi_j)=w_0(\varpi_j)$ and
$u(R_0^{I,-})\subseteq R_0^-$ for $u\in S_m^I$.
But $w_0(\varpi_j)$ is miniscule (i.e.
$|\langle \alpha,w_0(\varpi_j)\rangle|\leq 1$ for all $\alpha\in R_0$),
hence $\langle \alpha,w^{-1}(\varpi_j)\rangle\in\{0,1\}$ if
$\alpha\in w_0^\prime{}^{-1}(R_0^{I,-})$.

Fix now $\alpha\in w_j^\prime{}^{-1}(R_0^{I,-})$. Suppose first
that $\alpha\in R_0^{I,-}$.
Then $\langle\alpha,w^{-1}(\varpi_j)\rangle=
\langle w(\alpha),\varpi_j\rangle\leq 0$, since $w(\alpha)\in R_0^-$.
On the other hand, we have already observed that 
the scalar product is $0$ or $1$, hence this forces
$\langle \alpha,w^{-1}(\varpi_j)\rangle=0$.
If on the other hand $\alpha\in R_0^{I,+}$, then, since
$\sigma^{-j}w=(\overline{\sigma^{-j}w})w_j^\prime$ and
$l(\sigma^{-j}w)=l(\overline{\sigma^{-j}w})+l(w_j^\prime)$,
\[
R_0^+\cap (\sigma^{-j}w)^{-1}R_0^-=w_j^\prime{}^{-1}\bigl(
R_0^+\cap (\overline{\sigma^{-j}w})^{-1}R_0^-)\cup
(R_0^{I,+}\cap w_j^\prime{}^{-1}R_0^{I,-})
\]
(disjoint union), thus $\alpha\in R_0^+\cap (\sigma^{-j}w)^{-1}R_0^-$.
Since furthermore
$\alpha\in R_0^{I,+}$ and $w\in S_m^I$ we have $w(\alpha)\in R_0^+$,
hence $w(\alpha)\in R_0^+\cap \sigma^jR_0^-$. But
\begin{equation*}
\begin{split}
R_0^+\cap \sigma^jR_0^-&=R^+\cap \sigma^jR^-\\
&=R^+\cap \sigma^j\zeta^{-j}R^-\\
&=R^+\cap \tau(-\varpi_j)R^-.
\end{split}
\end{equation*}
In other words,
\begin{equation}\label{lastcall}
\tau(\varpi_j)(w\alpha)=w\alpha-\langle\varpi_j,w\alpha\rangle c\in R^-.
\end{equation}
We already observed that $w\alpha\in R_0^+$ and that
$\langle\varpi_j,w\alpha\rangle$ is $0$ or $1$, hence \eqref{lastcall}
can only hold true if $\langle\varpi_j,w\alpha\rangle=1$. This
completes the proof of \eqref{cl}, and hence of the theorem.
\end{proof}
\section{Cherednik-Matsuo type correspondences}
We return now again to arbitrary root systems.
Cherednik \cite{CInd} related solutions of quantum affine KZ equations
with values in principal series modules to common eigenfunctions
of the commuting Cherednik-Macdonald scalar $q$-difference operators.
We discuss and deepen this result using the
results of the previous sections.
\subsection{(Anti)symmetrization}

By Lemma \ref{tttt} we have for 
$\gamma\in T_I^{k^{\pm 1}}$
a linear map 
\begin{equation}\label{symmetmap}
\pi^{k^{-1},q}(C_{\pm}^{I^*}(k^{-1})): L_\pi^{I^*,\pm}[\overline{w_0}\gamma^{-1}]
\rightarrow
\bigl(L_\pi^{W_0\gamma^{-1}}\bigr)^{\pm},
\end{equation}
which we call the symmetrization ($+$), respectively antisymmetrization
($-$), map. Recall here that
$L_\pi^{W_0\gamma^{-1}}$ is the $H(k^{-1})$-submodule 
\[L_\pi^{W_0\gamma^{-1}}=\{\phi\in L \,\, | \,\, \pi^{k^{-1},q}(f(Y))\phi=
f(\gamma^{-1})\phi\quad
\forall f\in\mathbb{C}[T]^{W_0}\}
\]
of $L_\pi$. Define
\begin{equation}\label{TIreg}
T_{I,reg}^k:=\{\gamma\in T_I^k \,\, | \,\, k_{\alpha^\vee}^2\not=
\gamma^{\alpha^\vee}\not=1 \quad \forall \alpha\in R_0^+\setminus R_0^{I,+} \}.
\end{equation}

\begin{prop}\label{repappl}
Let $L$ be a left $\mathbb{C}_{\sigma,\nabla}^{k,q}[T]\#_qW$-module.\\
{\bf (a)} If $\gamma\in T_{I,reg}^{k^{\pm 1}}$ 
then the (anti)symmetrization map \eqref{symmetmap} is injective,
\[
\pi^{k^{-1},q}(C_{\pm}^{I^*}(k^{-1})): L_\pi^{I^*,\pm}[\overline{w_0}\gamma^{-1}]
\hookrightarrow
\bigl(L_\pi^{W_0\gamma^{-1}}\bigr)^{\pm}.
\]
{\bf (b)} Let $\gamma\in T$ such that
$k_{\alpha^\vee}^2\not=\gamma^{\alpha^\vee}\not=1$ for all $\alpha\in R_0$.
Then  
\[\pi^{k^{-1},q}(C_{\pm}(k^{-1})): L_{\pi,w_0\gamma^{-1}}
\overset{\sim}{\longrightarrow} 
\bigl(L_\pi^{W_0\gamma^{-1}}\bigr)^{\pm}
\]
is a linear isomorphism.
\end{prop}
\begin{proof}
{\bf (a)} For $\gamma\in T_I^{k}$ we have 
$\gamma_I:=\overline{w_0}\gamma^{-1}
\in T_{I^*}^{k^{-1}}$, cf. Lemma \ref{dualgamma}. Furthermore,
\[
\overline{w_0}(R_0^-\setminus R_0^{I,-})=
R_0^+\setminus R_0^{I^*,+}.
\]
Hence, for $\gamma\in T$, we have 
$\gamma\in T_{I,reg}^k$ if and only if $\gamma_I\in T_{I^*}^{k^{-1}}$
and 
$k_{\alpha^\vee}^{2}\not=
\gamma_I^{\alpha^\vee}\not=1$ for all
$\alpha\in R_0^+\setminus R_0^{I^*,+}$.

Let $\gamma\in T_{I,reg}^{k^{\pm 1}}$
and $\phi\in L_\pi^{I^*,\pm}[\gamma_I]$, and suppose that
$\pi^{k^{-1},q}(C_{\pm}^{I^*}(k^{-1}))\phi=0$. 
Since $\phi\in L_{\pi,\gamma_I}$ we have $I_w(k^{-1})\phi\in L_{\pi,w\gamma_I}$
for $w\in W_0^{I^*}$. The condition 
$\gamma_I^{\alpha^\vee}\not=1$
for all $\alpha\in R_0^+\setminus R_0^{I^*,+}$ implies that
the $w\gamma_I$ ($w\in W_0^{I^*}$) are pairwise different
(see Proposition \ref{cal}{\bf (i)}). 
Hence $\pi^{k^{-1},q}(C_{\pm}^{I^*}(k^{-1}))\phi=0$ implies, in view of
Theorem \ref{1exp}, that
\[\Bigl(\prod_{\alpha\in R_0^+\setminus R_0^{I^*,+}}
(\pm k_{\alpha^\vee}^{\pm 1}\mp
k_{\alpha^\vee}^{\mp 1}\gamma_I^{\alpha^\vee})\Bigr)\phi=0.
\]
Since $\gamma_I^{\alpha^\vee}\not=k_{\alpha^\vee}^{\pm 2}$ for all
$\alpha\in R_0^+\setminus R_0^{I^*,+}$, we conclude that $\phi=0$.\\
{\bf (b)} Fix $\gamma\in T$ satisfying
$k_{\alpha^\vee}^2\not=\gamma^{\alpha^\vee}\not=1$ for all $\alpha\in R_0$.
By {\bf (a)}, $\pi^{k^{-1},q}(C_{\pm}(k^{-1}))$
defines an injective linear map
\begin{equation}\label{mapinj}
\pi^{k^{-1},q}(C_{\pm}(k^{-1})): L_{\pi,w_0\gamma^{-1}}\hookrightarrow
\bigl(L_\pi^{W_0\gamma^{-1}}\bigr)^{\pm}.
\end{equation}
It remains to show that 
\eqref{mapinj} is surjective. For the remainder of the proof
we leave out the map $\pi^{k^{-1},q}$ from the notations.

Choose a nonzero element $v\in
\bigl(L_\pi^{W_0\gamma^{-1}}\bigr)^{\pm}$ and set
$M:=H(k^{-1})v$ for the associated
cyclic $H(k^{-1})$-submodule of $L_\pi^{W_0\gamma^{-1}}$.
By a result of Steinberg \cite[Thm. 2.2]{Stein}, 
$H(k^{-1})\simeq \mathcal{K}\otimes_{\mathbb{C}} 
\bigl(\mathcal{A}_Y^{k^{-1}}\bigr)^{W_0}\otimes_{\mathbb{C}}H_0(k^{-1})$ 
as complex vector spaces by the multiplication map,
with $\mathcal{K}\subseteq \mathcal{A}_Y^{k^{-1}}$ a complex subspace
of dimension $|W_0|$. It follows that 
$\textup{Dim}_{\mathbb{C}}\bigl(M\bigr)\leq |W_0|$.

For all $w\in W_0$, the intertwiner $I_{ww_0}(k^{-1})\in H(k^{-1})$ defines a 
linear bijection
\[I_{ww_0}(k^{-1}): M_{w_0\gamma^{-1}}\overset{\sim}{\longrightarrow}
M_{w\gamma^{-1}}
\]
with inverse $c_w^{-1}I_{w_0w^{-1}}(k^{-1})$, where (cf. 
Corollary \ref{productY}),
\[c_w=d_{ww_0}(w\gamma^{-1})d_{w_0w^{-1}}(w_0\gamma^{-1}),
\]
which is nonzero since $\gamma^{\alpha^\vee}\not=
k_{\alpha^\vee}^2$
for all $\alpha\in R_0$.
Since $M$ has central character $W_0\gamma^{-1}$, there
exists a $u\in W_0$ such that $M_{u\gamma^{-1}}\not=0$. 
We conclude that $M_{w\gamma^{-1}}\not=0$ for all $w\in W_0$.
Furthermore, $w\gamma^{-1}=w^\prime\gamma^{-1}$ for $w,w^\prime\in W_0$
iff $w=w^\prime$, since $\gamma^{\alpha^\vee}\not=1$
for all $\alpha\in R_0$ (see Proposition \ref{cal}{\bf (i)}). 
Combined with $\textup{Dim}_{\mathbb{C}}(M)\leq |W_0|$
we thus conclude that
\begin{equation}\label{Li}
M=\bigoplus_{w\in W_0}M_{w\gamma^{-1}},\qquad
\textup{Dim}_{\mathbb{C}}\bigl(M_{w\gamma^{-1}}\bigr)=1 \,\,\,
\forall\, w\in W_0.
\end{equation}

It follows from \eqref{Li} and the conditions on $\gamma$ that
$M\simeq M^{k^{-1}}(\gamma^{-1})$ as $H(k^{-1})$-modules. Hence,
$M^{\pm}=\mathbb{C}v$ (cf. Lemma \ref{1spher}) and $C_\pm(k^{-1})$ 
defines a linear isomorphism
\begin{equation}\label{isoi}
C_\pm(k^{-1}): 
M_{w_0\gamma^{-1}}\overset{\sim}{\longrightarrow} M^{\pm}
\end{equation}
between the one-dimensional complex vector spaces $M_{w_0\gamma^{-1}}$
and $M^{\pm}=\mathbb{C}v$ of $M\subseteq L_\pi^{W_0\gamma^{-1}}$. 
We conclude that there exists
a $m\in M_{w_0\gamma^{-1}}\subseteq
L_{\pi,w_0\gamma^{-1}}$ such that $C_{\pm}(k^{-1})m=v$. Hence 
the map \eqref{mapinj} is surjective.
\end{proof}
\subsection{The Cherednik-Matsuo map}

\begin{defi}\label{CMdef}
Let $L$ be a left $\mathbb{C}_{\sigma,\nabla}^{k,q}[T]\#_qW$
and $\gamma\in T_I^{k^{\pm 1}}$. The Cherednik-Matsuo map
$\xi_{L,\gamma}^{k,q,\pm,I}: 
\Gamma_L^{k,q}\bigl(M^{k,\pm,I}(\gamma)\bigr)\rightarrow L$
is defined by
\[
\xi_{L,\gamma}^{k,q,\pm,I}\bigl(\sum_{w\in W_0^I}\psi_w\otimes 
v_w^{k,\pm,I}(\gamma)\bigr)=\sum_{w\in W_0^I}\epsilon_{\pm}^k(T_w)\psi_w.
\]
If $I=\emptyset$ then we 
write $\xi_{L,\gamma}^{k,q,\pm}=\xi_{L,\gamma}^{k,q,\pm,\emptyset}$.
\end{defi}
The map $\xi_{L,\gamma}^{k,q,+}$ was considered
by Cherednik \cite{CQKZ,CInd}. Its classical analog appears in
the work of Matsuo \cite{Mat}. Observe that the Cherednik-Matsuo maps
are compatible with the functorial maps 
$\Gamma_L^{k,q}(M^{k,\pm,I}(\gamma))\rightarrow 
\Gamma_L^{k,q}(M^{k,\pm,J}(\gamma))$
for $I\subseteq J$ and $\gamma\in T_J^{k^{\pm 1}}$
(cf. \eqref{funtapp} in the special case 
that $I=\emptyset$).

Note that the maps $\xi_{L,\gamma}^{k,q,\pm,I}$ for $\gamma\in T_I^{k^{\pm 1}}$
are $W_0$-equivariant in the following sense,
\begin{equation}\label{xiequiv}
\xi_{L,\gamma}^{k,q,\pm,I}\bigl(\nabla^{k,q}(w)\psi\bigr)=
w_{\pm}\bigl(\xi_{L,\gamma}^{k,q,\pm,I}(\psi)\bigr)\qquad 
(\psi\in\Gamma^{k,q}_L(M^{k,\pm,I}(\gamma)), w\in W_0),
\end{equation}
with $w_{\pm}\in\mathbb{C}_{\sigma,\nabla}^{k,q}[T]\#_qW$ given by
\eqref{wplusmin}.
This follows directly from the explicit expression \eqref{nablarel}
of $\nabla^{k,q}(s_j)$ ($1\leq j\leq n$) and \eqref{vaction}. 
Recall the set $T_{I,reg}^{k}$ (see \eqref{TIreg}) of regular
elements in $T_I^k$.
\begin{prop}\label{CMcorrsymprop}
Let $L$ be a left $\mathbb{C}_{\sigma,\nabla}^{k,q}[T]\#_qW$-module.\\
{\bf (a)} Let $\gamma\in T_I^{k^{\pm 1}}$. Then $\xi_{L,\gamma}^{k,q,\pm,I}$
restricts to a linear map
\[
\xi_{L,\gamma}^{k,q,\pm,I}: \Gamma_L^{k,q}\bigl(M^{k,\pm,I}(\gamma)\bigr)_\nabla^W
\rightarrow \bigl(L_\pi^{W_0\gamma^{-1}}\bigr)^{\pm}.
\]
The map is injective if $\gamma\in T_{I,reg}^{k^{\pm 1}}$.\\
{\bf (b)} Let $\gamma\in T$ such that $k_{\alpha^\vee}^2\not=\gamma^{\alpha^\vee}
\not=1$ for all $\alpha\in R_0$. Then $\xi_{L,\gamma}^{k,q,\pm}$ restricts
to a linear isomorphism
\[
\xi_{L,\gamma}^{k,q,\pm}: \Gamma_L^{k,q}\bigl(M^{k}(\gamma)\bigr)_\nabla^W
\overset{\sim}{\longrightarrow} \bigl(L_\pi^{W_0\gamma^{-1}}\bigr)^{\pm}.
\]
\end{prop}
\begin{proof}
By Theorem \ref{mainthmY} and \eqref{symmetmap}
we have for 
$\gamma\in T_I^{k^{\pm 1}}$ a linear map
\begin{equation}\label{tildexi}
\tilde{\xi}_\gamma^{I,\pm}: \Gamma_L^{k,q}\bigl(M^{k,\pm,I}(\gamma)\bigr)_\nabla^W
\rightarrow \bigl(L_\pi^{W_0\gamma^{-1}}\bigr)^{\pm}
\end{equation}
given by
\[
\tilde{\xi}_\gamma^{I,\pm}
\bigl(\sum_{w\in W_0^I}\psi_w\otimes v_w^{k,\pm,I}(\gamma)\bigr)=
\sum_{w\in W_0^{I^*}}\epsilon_{\pm}^{k^{-1}}(T_w)\pi^{k^{-1},q}(T_w)
\psi_{\overline{w_0}}.
\]
By Theorem \ref{mainthmY} and Proposition \ref{repappl}{\bf (a)},
the map $\tilde{\xi}_\gamma^{I,\pm}$ 
is injective if $\gamma\in T_{I,reg}^{k^{\pm 1}}$.
For $I=\emptyset$, by Theorem \ref{mainthmY} and Proposition
\ref{repappl}{\bf (b)},
$\tilde{\xi}_\gamma^{\emptyset,\pm}$ is a linear isomorphism
if $k_{\alpha^\vee}^2\not=\gamma^{\alpha^\vee}\not=1$ for all $\alpha\in R_0$.
To complete the proof of the theorem, it thus suffices to show that
\[
\tilde{\xi}_\gamma^{I,\pm}(\psi)=\epsilon_{\pm}^{k^{-1}}(T_{\overline{w_0}})
\xi_{L,\gamma}^{k,q,\pm,I}(\psi),\qquad \forall\, \psi\in 
\Gamma_L^{k,q}\bigl(M^{k,\pm,I}(\gamma)\bigr)_\nabla^W.
\]
Let $\psi=\sum_{w\in W_0^I}\psi_w\otimes v_w^{k,\pm,I}(\gamma)\in
\Gamma_L^{k,q}\bigl(M^{k,\pm,I}(\gamma)\bigr)_\nabla^W$. Then
$\psi_w=\pi^{k^{-1},q}(T_{w\overline{w_0}^{-1}})\phi$ 
for all $w\in W_0^I$ with
$\phi\in L_\pi^{I^*,\pm}[\overline{w_0}\gamma^{-1}]$.

Since $\overline{w_0}(\alpha_i)=\alpha_{i_I^*}$ for
$i\in I$, we have $W_{0,I^*}=\overline{w_0}W_{0,I}\overline{w_0}^{-1}$
and $W_0^{I^*}=W_0^I\overline{w_0}^{-1}$. 
Furthermore, for $w\in W_0^I$
we have $T_{w\overline{w_0}^{-1}}=T_{w^{-1}}^{-1}T_{\overline{w_0}^{-1}}$,
see the proof of Corollary \ref{coreq}.
Hence we compute,
\begin{equation*}
\begin{split}
\tilde{\xi}_\gamma^{I,\pm}(\psi)&=\sum_{w\in W_0^{I^*}}
\epsilon_{\pm}^{k^{-1}}(T_w)\pi^{k^{-1},q}(T_w)\phi\\
&=\sum_{w\in W_0^I}\epsilon_{\pm}^{k^{-1}}(T_{w\overline{w_0}^{-1}})
\pi^{k^{-1},q}(T_{w\overline{w_0}^{-1}})\phi\\
&=\sum_{w\in W_0^I}\epsilon_{\pm}^{k^{-1}}(T_{w^{-1}}^{-1}T_{\overline{w_0}^{-1}})
\psi_w\\
&=\epsilon_{\pm}^{k^{-1}}(T_{\overline{w_0}})\sum_{w\in W_0^I}\epsilon_{\pm}^k(T_w)
\psi_w\\
&=\epsilon_{\pm}^{k^{-1}}(T_{\overline{w_0}})\xi_{L,\gamma}^{k,\gamma,\pm,I}(\psi),
\end{split}
\end{equation*}
as desired.
\end{proof}

\subsection{Spectral problem of families of 
commuting $q$-difference operators}
For a left $\mathbb{C}_{\sigma,\nabla}^{k,q}[T]\#_qW$-module $L$
and $\gamma\in T$, we will identify, following Cherednik \cite{CInd},
the subspace $\bigl(L_\pi^{W_0\gamma^{-1}}\bigr)^\pm$,
which appears in Proposition \ref{CMcorrsymprop},
with a space of common eigenfunctions
of a commuting family of $q$-difference operators. In case of the
symmetric theory ($+$),
they are the Cherednik-Macdonald $q$-difference operators. 
\begin{defi}
The subalgebra
\[\mathbb{D}_{\sigma,\nabla}^{k,q}:=\mathbb{C}_{\sigma,\nabla}^{k,q}[T]\#_q
\tau(P^\vee)
\]
of $\mathbb{C}_{\sigma,\nabla}^{k,q}[T]\#_qW$ is called 
the algebra of $q$-difference
operators on $T$ with coefficients from $\mathbb{C}_{\sigma,\nabla}^{k,q}[T]$.
\end{defi}

Recall the elements $w_{\pm}\in\mathbb{C}_{\sigma,\nabla}^{k,q}[T]\#_qW$
from \eqref{wplusmin}.
A (twisted) $W_0$-action by algebra
automorphisms on $\mathbb{D}_{\sigma,\nabla}^{k,q}\subset
\mathbb{C}_{\sigma,\nabla}^{k,q}[T]\#_qW$ is given by
\[
(w,D)\mapsto w_{\pm}Dw_{\pm}^{-1}\qquad (w\in W_0,\,\, D\in
\mathbb{D}_{\sigma,\nabla}^{k,q}).
\]
We set
\[
\mathbb{D}_{\sigma,\nabla}^{k,q,\pm}:=\{D\in\mathbb{D}_{\sigma,\nabla}^{k,q}
\,\, | \,\, w_{\pm}Dw_{\pm}^{-1}=D\quad \forall\, w\in W_0\}
\]
the subalgebra of $\mathbb{D}_{\sigma,\nabla}^{k,q}$ consisting of
$W_{0,\pm}$-invariant $q$-difference operators. 

The following result is due to Cherednik \cite{CInd}.
\begin{prop}\label{CMDEprop}
For $f\in\mathbb{C}[T]^{W_0}$ write $D_f^{k,q,\pm}:=\sum_{w\in W_0}D_{f,w}^{\pm}$ 
with $D_{f,w}^{\pm}$ the unique elements from $\mathbb{D}_{\sigma,\nabla}^{k,q}$
such that
\[\pi^{k^{-1},q}(f(Y))=\sum_{w\in W_0}D_{f,w}^{\pm}w_{\pm}
\]
in $\mathbb{C}_{\sigma,\nabla}^{k,q}[T]\#_qW$.
Then the map $f\mapsto D_f^{k,q,\pm}$ defines an algebra map
\begin{equation}\label{algmap}
\mathbb{C}[T]^{W_0}\rightarrow \mathbb{D}_{\sigma,\nabla}^{k,q,\pm}.
\end{equation}
In particular, the $q$-difference operators $D_f^{k,q,\pm}$ 
($f\in\mathbb{C}[T]^{W_0}$) pairwise commute.
\end{prop}
For the proof of Proposition \ref{CMDEprop}
one first shows that 
$D_f^{k,q,\pm}$ ($f\in\mathbb{C}[T]^{W_0}$) is $W_{0,\pm}$-invariant
(for which one uses the fact that $f(Y)$ lies in the center of $H(k^{-1})$
if $f\in\mathbb{C}[T]^{W_0}$, as well as the expressions of $\pi^{k^{-1},q}(T_i)$
($1\leq i\leq n$) from
the proof of Lemma \ref{W0IinvL}). It then follows that 
\eqref{algmap} is an algebra homomorphism
(see \cite[Lemma 2.7]{LS} for a detailed proof
in the symmetric case ($+$)).

\begin{defi}
The $D_f^{k,q,+}\in \mathbb{D}_{\sigma,\nabla}^{k,q,+}$ 
($f\in\mathbb{C}[T]^{W_0}$) are the
Cherednik-Macdonald $q$-difference operators.
\end{defi}
The Macdonald $q$-difference operators correspond to $D_f^{k,q,+}$ with
$f\in\mathbb{C}[T]^{W_0}$ given by
$f(t)=\sum_{\mu\in W_0\lambda}t^\mu$ and $w_0\lambda$ 
a (quasi-)miniscule coweight. They can be written down explicitly (cf., e.g., 
\cite[\S 4.4]{M}).

\begin{defi}\label{SPM}
Let $L$ be a left $\mathbb{C}_{\sigma,\nabla}^{k,q}[T]\#_qW$-module
and $\gamma\in T$.
We call 
\[
\textup{SpM}_L^{k,q,\pm}(W_0\gamma):=\{\phi\in L \,\, | \,\, 
D_f^{k,q,\pm}\phi=f(\gamma)\phi\quad \forall\, f\in\mathbb{C}[T]^{W_0}\}
\]
the solution space of the spectral problem of the commuting operators
$D_f^{k,q,\pm}$ ($f\in\mathbb{C}[T]^{W_0}$) on $L$, 
with spectral parameter $W_0\gamma\in T/W_0$.
\end{defi}
By the previous proposition, $\textup{SpM}_L^{k,q,\pm}(W_0\gamma)$
is a $W_{0,\pm}$-invariant subspace of $L$.
We write $\textup{SpM}_L^{k,q,\pm}(W_0\gamma)^{W_{0,\pm}}$ 
for the corresponding subspace
of $W_{0,\pm}$-invariant elements.

\begin{prop}\label{propsame}
Let $L$ be a $\mathbb{C}_{\sigma,\nabla}^{k,q}[T]\#_qW$-module and
$\gamma\in T$. 
Then
\begin{equation}\label{same}
\bigl(L_\pi^{W_0\gamma}\bigr)^{\pm}=\textup{SpM}_L^{k,q,\pm}(W_0\gamma)^{W_{0,\pm}}.
\end{equation}
\end{prop}
\begin{proof}
We have $L_\pi^{\pm}=L^{W_{0,\pm}}$ by Lemma \ref{W0IinvL} (where, recall,
$L_\pi^{\pm}$ is the space of (anti)spherical vectors in $L_\pi$,
see Subsection \ref{spherical}).
Furthermore, 
for $\phi\in L^{W_{0,\pm}}$ we have, for all $f\in\mathbb{C}[T]^{W_0}$,
\[D_f^{k,q,\pm}\phi=\sum_{w\in W_0}D_{f,w}^{\pm}\phi=
\sum_{w\in W_0}D_{f,w}^{\pm}w_{\pm}\phi=\pi^{k^{-1},q}(f(Y))\phi.
\]
This implies now immediately \eqref{same}.
\end{proof}

\begin{rema}\label{antisymmpol}
In this remark we explain the link to the theory
of (anti)symmetric Macdonald polynomials (see \cite{C,M}).
Fix $q\in\mathbb{C}^\times$, not a root of unity.
Fix furthermore a multiplicity function 
$k$, sufficiently generic so that the theory of 
nonsymmetric Macdonald polynomials with respect to the parameters
$(k^{-1},q)$ is to our proposal.

Set $L:=\mathbb{C}_{\sigma,\nabla}^{k,q}[T]$ (viewed as left
$\mathbb{C}_{\sigma,\nabla}^{k,q}[T]\#_qW$-module) and consider 
the linear map
\[\pi^{k^{-1},q}(C_{\pm}(k^{-1})): L_\pi\rightarrow L_\pi^{\pm}
\]
(see Lemma \ref{tttt}). It restricts, for each $\gamma\in T$,
to a linear map 
\[\pi^{k^{-1},q}(C_{\pm}(k^{-1})): L_\pi^{W_0\gamma^{-1}}
\rightarrow \bigl(L_\pi^{W_0\gamma^{-1}}\bigr)^{\pm}.
\]
Let $\lambda\in P^\vee$. Then there exists, up to a scalar multiple, a unique
$0\not=E_\lambda(k^{-1},q)\in\mathbb{C}[T]\subset L$ satisfying
\[
\pi^{k^{-1},q}(f(Y))E_\lambda(k^{-1},q)=
f(\gamma_\lambda(k^{-1},q)^{-1})E_\lambda(k^{-1},q),\qquad
\forall f\in \mathbb{C}[T],
\]
where
\[
\gamma_\lambda(k^{-1},q)=q^{\lambda}\prod_{\alpha\in R_0^+}
k_{\alpha^\vee}^{-\eta(\langle\lambda,\alpha\rangle)\alpha}
\]
with $\eta(x)=1$ if $x>0$ and $=-1$ if $x\leq 0$.
The Laurent polynomial $E_\lambda(k^{-1},q)$
is the nonsymmetric Macdonald polynomial of degree $\lambda$. Let
$P_+^\vee$ be the cone of dominant coweights. Then,
writing $\gamma_\lambda=\gamma_\lambda(k^{-1},q)$, the symmetric and
antisymmetric Macdonald polynomials 
are defined, for $\lambda\in P_+^\vee$, by
\[
P_\lambda^{(\pm)}(k^{-1},q):=\pi^{k^{-1},q}(C_{\pm}(k^{-1}))E_\lambda(k^{-1},q),
\]
which is an element in $\bigl(L_\pi^{W_0\gamma_\lambda^{-1}}\bigr)^{\pm}=
\textup{SpM}_L^{k,q,\pm}(W_0\gamma_\lambda^{-1})^{W_{0,\pm}}$.
Under generic conditions on the parameters we have 
$P_\lambda^{(+)}(k^{-1},q)\not=0$ for all $\lambda\in P_+^\vee$,
and $P_\lambda^{(-)}(k^{-1},q)\not=0$ unless 
$\lambda\in P_+^\vee\setminus (\rho^\vee+P_+^\vee)$, in which case
$P_\lambda^{(-)}(k^{-1},q)=0$
(cf., e.g., \cite[\S 5.7]{M}).
\end{rema}

From Proposition \ref{CMcorrsymprop} we now immediately obtain
the following important intermediate result.
\begin{prop}\label{CMcorrsym}
Let $L$ be a left $\mathbb{C}_{\sigma,\nabla}^{k,q}[T]\#_qW$-module.\\
{\bf (a)} Let $\gamma\in T_I^{k^{\pm 1}}$. Then $\xi_{L,\gamma}^{k,q,\pm,I}$
restricts to a linear map
\[
\xi_{L,\gamma}^{k,q,\pm,I}: \Gamma_L^{k,q}\bigl(M^{k,\pm,I}(\gamma)\bigr)_\nabla^W
\rightarrow \textup{SpM}_L^{k,q,\pm}(W_0\gamma^{-1})^{W_{0,\pm}}.
\]
The map is injective if $\gamma\in T_{I,reg}^{k^{\pm 1}}$ (see \eqref{TIreg}).\\
{\bf (b)} Let $\gamma\in T$ such that $k_{\alpha^\vee}^2\not=\gamma^{\alpha^\vee}
\not=1$ for all $\alpha\in R_0$. Then $\xi_{L,\gamma}^{k,q,\pm}$ restricts
to a linear isomorphism
\[
\xi_{L,\gamma}^{k,q,\pm}: \Gamma_L^{k,q}\bigl(M^{k}(\gamma)\bigr)_\nabla^W
\overset{\sim}{\longrightarrow} 
\textup{SpM}_L^{k,q,\pm}(W_0\gamma^{-1})^{W_{0,\pm}}.
\]
\end{prop}
\begin{rema}
Let $L$ be a left $\mathbb{C}_{\sigma,\nabla}^{k,q}[T]\#_qW$-module.
Multiplcation by $G^{k,-}$ defines a linear isomorphism
$L_\pi^+\overset{\sim}{\longrightarrow} L_\pi^-$.
Through this map, in case $L=\mathbb{C}_{\sigma,\nabla}^{k,q}[T]$
and the parameters $k,q$ are generic, symmetric Macdonald polynomials
are mapped to antisymmetric Macdonald polynomials with respect to a 
shifted multiplicity function, see, e.g., \cite[(5.8.9)]{M}.   
\end{rema}

\subsection{The correspondences}

In this subsection we prove a nonsymmetric version of Proposition
\ref{CMcorrsym} (Proposition \ref{CMcorrsym} then is reobtained 
by restriction to the subspace of $W_0$-invariant
elements).
We will prove the nonsymmetric version of the theorem using a 
construction which is
motivated by Opdam's \cite[\S3]{O} analysis of the trigonometric KZ equation
(see also Cherednik and Ma \cite[\S 3.5]{CM} for a different but closely
related treatment).

Let $A$ and $B$ be unital associative $\mathbb{C}$-algebras.
We write 
$\textup{BiMod}_{(A,B)}$ for the category of left $(A,B)$-bimodules over
$\mathbb{C}$. 

Let $A$ be an an unital associative $\mathbb{C}$-algebra, endowed with a left
action of a group $G$ by unital algebra automorphisms. Write $A\#G$ for the
associated smashed product algebra 
(so $A\#G\simeq A\otimes_{\mathbb{C}}\mathbb{C}[G]$
as vector spaces, with multiplication law $(a\otimes g)
(a^\prime\otimes g^\prime)=ag(a^\prime)\otimes gg^\prime$).
We then have a covariant functor
\begin{equation}\label{spinor}
F_A^G: \textup{Mod}_{A\#G}\rightarrow \textup{BiMod}_{(A\#G,\mathbb{C}[G])}
\end{equation}
defined as follows. Let $M$ be a left $A\#G$-module. Then
$F_A^G(M)$ as complex vector space is the space of functions
$f: G\rightarrow M$. It is viewed as left $A\#G$-module by
\begin{equation*}
\begin{split}
(a\cdot f)(g^\prime)&:=a\cdot f(g^\prime),\\
(g\cdot f)(g^\prime)&:=g\cdot f(g^\prime g)
\end{split}
\end{equation*}
for $a\in A$, $g,g^\prime\in G$ and $f\in F_A^G(M)$
where, on the right hand side, the dot stands for the $A\#G$-action
on $M$ (from now on we leave out the dot from the notations). 
The left $\mathbb{C}[G]$-action on $F_A^G(M)$, commuting with the
above $A\#G$-action, is defined by
\[(\mu(g)f)(g^\prime):=f(g^{-1}g^\prime)
\]
for $f\in F_A^G(M)$ and $g,g^\prime\in G$. If $\phi$ is a morphism of
$A\#G$-modules then $F_{M}^G(\phi)$ is set to be
$(F_{A}^G(\phi)f)(g):=\phi(f(g))$.

Let $M$ be a left $A\#G$-module. Then we have a linear
isomorphism of $A\#G$-modules
\begin{equation}\label{identpure}
F_A^G(M)^{\mu(G)}\overset{\sim}{\longrightarrow} M,\qquad
f\mapsto f(e),
\end{equation}
where $e\in G$ is the identity 
element (we could as well evaluate $f$ at any other
element $g\in G$). 
We will freely use this identification in the remainder of this subsection.
If $M$ is a $(A\#G,\mathbb{C}[G])$-bimodule, then we write $M^G$ for
the subspace of $G$-invariants in $M$, with the $G$-action coming from
the action of $A\#G$ on $M$.
\begin{lem}\label{otheridentification}
Let $A$ be a unital associative $\mathbb{C}$-algebra endowed with an action
of a group $G$ by unital algebra automorphisms. Let $M$ be a left $A\#G$-module.
\begin{enumerate}
\item[{\bf (i)}] We have a linear isomorphism
\[\Xi_M: M\overset{\sim}{\longrightarrow} F_A^G(M)^{G}
\]
defined by $(\Xi_Mm)(g):=g^{-1}m$ for $g\in G$ and $m\in M$.
Furthermore, $\Xi_M^{-1}(f)=f(e)$ for $f\in F_A^G(M)^G$. 
\item[{\bf (ii)}] For $m\in M$ and $g\in G$ we have $\Xi_M(gm)=
\mu(g)(\Xi_Mm)$. In particular, $\Xi_M|_{M^G}=\textup{id}$ if we take
the identification \eqref{identpure} into account.
\end{enumerate}
\end{lem}
\begin{proof}
We write $\Xi=\Xi_M$ for the duration of the proof.\\
{\bf (i)} Let $m\in M$ and $g,g^\prime\in G$, then
\begin{equation*}
\begin{split}
\bigl(g^\prime(\Xi m)\bigr)(g)&=
g^\prime\bigl((\Xi m)(gg^\prime)\bigr)\\
&=g^\prime(g^\prime{}^{-1}g^{-1})m\\
&=g^{-1}m\\
&=(\Xi m)(g).
\end{split}
\end{equation*}
This shows that $\Xi m\in F_A^G(M)$
is $G$-invariant. 

Define now $\Xi^\prime: F_A^G(M)^{G}\rightarrow
M$ by $\Xi^\prime(f)=f(e)$.
Note that $\Xi^\prime\circ\Xi=\textup{id}$.
On the other hand, if $f\in F_A^G(M)^{G}$, then
for $g\in G$,
\[\Xi(\Xi^\prime(f))(g)=g^{-1}(f(e))=
g^{-1}g(f(g))=f(g),
\]
where we used that $f$ is $G$-invariant in the second equality.\\
{\bf (ii)} Let $m\in M$ and $g,g^\prime\in G$. Then
\[(\Xi(gm))(g^\prime)=g^\prime{}^{-1}gm=
(\mu(g)(\Xi m))(g^\prime).
\]
The last statement of {\bf (ii)} is obvious.
\end{proof}

Since $W_0$ acts by conjugation
on the subalgebra $\mathbb{D}_{\sigma,\nabla}^{k,q}$ of
$\mathbb{C}_{\sigma,\nabla}^{k,q}[T]\#_qW$, 
we have an isomorphism
\[\mathbb{C}_{\sigma,\nabla}^{k,q}[T]\#_qW\simeq 
\mathbb{D}_{\sigma,\nabla}^{k,q}\#W_0
\]
of algebras, cf. Proposition \ref{CMDEprop}. 
We can thus apply the above constructions with
$A=\mathbb{D}_{\sigma,\nabla}^{k,q}$ the algebra of $q$-difference operators
viewed as a $G=W_0$-module algebra.

If $N$ is a left $\mathbb{C}_{\sigma,\nabla}^{k,q}[T]\#_qW$-module,
then we write
\[
\widehat{N}:=F_{\mathbb{D}_{\sigma,\nabla}^{k,q}}^{W_0}\bigl(N\bigr)
\]
for the corresponding $(\mathbb{C}_{\sigma,\nabla}^{k,q}[T]\#_qW,
\mathbb{C}[W_0])$-bimodule. 
Concretely, $\widehat{N}$ consists
of the space of functions $f: W_0\rightarrow N$, with actions
\begin{equation*}
\begin{split}
(af)(w^\prime)&:=a\bigl(f(w^\prime)\bigr),\\
(wf)(w^\prime)&:=w\bigl(f(w^\prime w)\bigr),\\
(\mu(w)f)(w^\prime)&:=f(w^{-1}w^\prime)
\end{split}
\end{equation*}
for $f\in\widehat{N}$,
$a\in\mathbb{D}_{\sigma,\nabla}^{k,q}$ and $w,w^\prime\in W_0$.

Let $L$ be a left $\mathbb{C}_{\sigma,\nabla}^{k,q}[T]\#_qW$-module and
$M$ a left $H(k)$-module. Then $\Gamma_{L}^{k,q}(M)_\nabla$ is a 
$\mathbb{C}_{\sigma,\nabla}^{k,q}[T]\#_qW$-module via the algebra homomorphism
\[\nabla^{k,q}: 
\mathbb{C}_{\sigma,\nabla}^{k,q}[T]\#_qW\rightarrow \mathcal{A}^{k,q}
\]
(see Corollary \ref{cor1}). Hence we can form the corresponding
$\bigl(\mathbb{C}_{\sigma,\nabla}^{k,q}[T]\#_qW,\mathbb{C}[W_0]\bigr)$-bimodule
$\widehat{\Gamma_L^{k,q}(M)}_{\nabla}$. 
On the other hand, we have the left 
$\mathbb{C}_{\sigma,\nabla}^{k,q}[T]\#_qW$-module
$\Gamma_{\widehat{L}}^{k,q}\bigl(M\bigr)_\nabla$, with the action again obtained
from the algebra map $\nabla^{k,q}$.
The $\mu(W_0)$-action on $\widehat{L}$ 
naturally extends to an $W_0$-action
on $\Gamma_{\widehat{L}}^{k,q}\bigl(M\bigr)_\nabla$ (using
$\Gamma_{\widehat{L}}^{k,q}\bigl(M\bigr)=\widehat{L}\otimes M$ as vector
spaces, it is the $\mu(W_0)$-action on the first tensor
leg).  We will use the same notation
$\mu$ for this extended action. With these two actions, 
$\Gamma_{\widehat{L}}^{k,q}\bigl(M\bigr)_\nabla$ becomes a 
$(\mathbb{C}_{\sigma,\nabla}^{k,q}[T]\#_qW,
\mathbb{C}[W_0])$-bimodule. 
Using \eqref{connmatrix}, in particular using that
$\nabla^{k,q}(\mathbb{D}_{\sigma,\nabla}^{k,q})\subseteq
\mathbb{D}_{\sigma,\nabla}^{k,q}\otimes_{\mathbb{C}}H(k)$,
it follows that the canonical complex linear isomorphism
\[
\widehat{\Gamma_L^{k,q}(M)}_{\nabla}\simeq 
\Gamma_{\widehat{L}}^{k,q}\bigl(M\bigr)_\nabla
\]
is an isomorphism of
$(\mathbb{C}_{\sigma,\nabla}^{k,q}[T]\#_qW,\mathbb{C}[W_0])$-bimodules.
We will use this isomorphism to
identify $\widehat{\Gamma_L^{k,q}(M)}_\nabla$ and 
$\Gamma_{\widehat{L}}^{k,q}\bigl(M\bigr)_\nabla$
in the remainder of the text. Lemma \ref{otheridentification} then gives
\begin{cor}\label{identcorr}
Let $L$ be a left $\mathbb{C}_{\sigma,\nabla}^{k,q}[T]\#_qW$-module
and $M$ a left $H(k)$-module. 
\begin{enumerate}
\item[{\bf (i)}] We have a linear isomorphism
\[
\Xi: \Gamma_{L}^{k,q}(M)_\nabla\overset{\sim}{\longrightarrow}
\Gamma_{\widehat{L}}^{k,q}(M)_\nabla^{W_{0}}
\]
defined by $(\Xi f)(w):=\nabla^{k,q}(w^{-1})f$ for $w\in W_0$ and
$f\in \Gamma_L(M)$. Furthermore, $\Xi^{-1}(h)=h(e)$ for
$h\in\Gamma_{\widehat{L}}^{k,q}(M)_\nabla^{W_{0}}$.
The map $\Xi$  restricts to a linear isomorphism
\[\Gamma_{L}^{k,q}(M)_\nabla^{\tau(P^\vee)}\overset{\sim}{\longrightarrow} 
\Gamma_{\widehat{L}}^{k,q}(M)_\nabla^W.
\]
\item[{\bf (ii)}] For $f\in \Gamma_L^{k,q}(M)_\nabla$ and $w\in W_0$
we have $\Xi (\nabla^{k,q}(w)f)=\mu(w)(\Xi f)$.
In particular, $\Xi|_{\Gamma_L^{k,q}(M)_\nabla^{W_0}}=\textup{id}$ if we take
the identification \eqref{identpure} into account.
\end{enumerate}
\end{cor}
\begin{proof}
Apply Lemma \ref{otheridentification} to the
$\mathbb{C}_{\sigma,\nabla}^{k,q}[T]\#_qW$-module
$\Gamma_L^{k,q}(M)_\nabla$ and set $\Xi=\Xi_{\Gamma_L^{k,q}(M)_\nabla}$. 
This gives all the results besides 
the second statement of {\bf (i)}, which though follows by
a direct computation.
\end{proof}
The corollary can be used to formulate nonsymmetric versions of 
Theorem \ref{mainthmY} and of Proposition \ref{CMcorrsym}.
We start with the nonsymmetric version of Theorem \ref{mainthmY}.
\begin{cor}\label{corspin}
Let $\gamma\in T_I^{k^{\pm 1}}$.
Let $L$ be a left $\mathbb{C}_{\sigma,\nabla}^{k,q}[T]\#_qW$-module.
We have a linear isomorphism
\[
\eta_{\pm}: \widehat{L}_\pi^{I^*,\pm}[\overline{w_0}\gamma^{-1}]
\overset{\sim}{\longrightarrow}
 \Gamma_L^{k,q}\bigl(M^{k,\pm,I}(\gamma)\bigr)_{\nabla}^{\tau(P^\vee)},
\]
given by 
\[
\eta_{\pm}(\phi):=\sum_{w\in W_0^I}\bigl(\pi^{k^{-1},q}(
T_{w\overline{w_0}^{-1}})\phi\bigr)(e)\otimes v_w^{k,\pm,I}(\gamma),\qquad
\phi\in \widehat{L}_\pi^{I^*,\pm}[\overline{w_0}\gamma^{-1}].
\]
Furthermore, $\widehat{L}_\pi^{I^*,\pm}[\overline{w_0}\gamma^{-1}]
\subseteq \widehat{L}$
is a $\mu(W_0)$-submodule, and, 
with the identification \eqref{identpure},
the map $\eta_{\widehat{L}_\pi^{I^*,\pm}[\overline{w_0}\gamma^{-1}]^{\mu(W_0)}}$
coincides with the isomorphism 
\[
L_\pi^{I^*,\pm}[\overline{w_0}\gamma^{-1}]\overset{\sim}{\longrightarrow}
\Gamma_L^{k,q}\bigl(M^{k,\pm,I}(\gamma)\bigr)_\nabla^W
\]
from Theorem \ref{mainthmY}, given by $\phi\mapsto \psi_\phi$
(see \eqref{isoY2}).
\end{cor}
\begin{proof}
The first statement follows from the chain of isomorphisms
\[
\widehat{L}_\pi^{I^*,\pm}[\overline{w_0}\gamma^{-1}]
\overset{\sim}{\longrightarrow}
\Gamma_{\widehat{L}}^{k,q}\bigl(M^{k,\pm,I}(\gamma)\bigr)_\nabla^W
\overset{\sim}{\longrightarrow} 
\Gamma_L^{k,q}\bigl(M^{k,\pm,I}(\gamma))_\nabla^{\tau(P^\vee)},
\]
with the first isomorphism given by
\[\widehat{L}_\pi^{I^*,\pm}[\overline{w_0}\gamma^{-1}]\ni \phi\mapsto \psi_\phi=
\sum_{w\in W_0^I}\pi^{k^{-1},q}(T_{w\overline{w_0}^{-1}})\phi\otimes
v_w^{k,\pm,I}(\gamma)
\] 
(cf. Theorem \ref{mainthmY}), and the second isomorphism given by $\Xi^{-1}$,
which maps $f\in \Gamma_{\widehat{L}}^{k,q}\bigl(M^{k,\pm,I}(\gamma)\bigr)_\nabla^W$
to $f(e)\in\Gamma_L^{k,q}\bigl(M^{k,\pm,I}(\gamma)\bigr)_\nabla^{\tau(P^\vee)}$.

It is clear that $\widehat{L}_\pi^{I^*,\pm}[\overline{w_0}\gamma^{-1}]
\subseteq \widehat{L}$ is a 
$\mu(W_0)$-submodule. Let 
$w^\prime\in W_0$ and $\phi\in 
\widehat{L}_\pi^{I^*,\pm}[\overline{w_0}\gamma^{-1}]$.
Then 
\[\eta_{\pm}(\mu(w^\prime)\phi)=
\sum_{w\in W_0^I}\bigl(\pi^{k^{-1},q}(
T_{w\overline{w_0}^{-1}})\phi\bigr)(w^\prime{}^{-1})\otimes v_w^{k,\pm,I}(\gamma).
\]
Thus, with the identification 
$\widehat{L}^{\mu(W_0)}\simeq L$
as $\mathbb{C}_{\sigma,\nabla}^{k,q}[T]\#_qW$-modules 
(see \eqref{identpure}), we get
for $\phi\in \widehat{L}_\pi^{I^*,\pm}[\overline{w_0}\gamma^{-1}]^{\mu(W_0)}
\simeq L_\pi^{I^*,\pm}[\overline{w_0}\gamma^{-1}]$,
\[\eta_{\pm}(\phi)=
\sum_{w\in W_0^I}\pi^{k^{-1},q}(T_{w\overline{w_0}^{-1}})\phi
\otimes v_w^{k,\pm,I}(\gamma)=\psi_\phi.
\]
\end{proof}
Similarly we are now going to prove the 
nonsymmetric version of Proposition \ref{CMcorrsym}
and Proposition \ref{CMcorrsymprop}. We need the following
prepatory lemma. Recall that $W_{0,\pm}=\{w_{\pm}\}_{w\in W_0}\subset
\bigl(\mathbb{C}_{\sigma,\nabla}^{k,q}[T]\#_qW\bigr)^\times$,
with $w_{\pm}$ given by \eqref{wplusmin}.
\begin{lem}\label{515}
Let $L$ be a left $\mathbb{C}_{\sigma,\nabla}^{k,q}[T]\#_qW$-module.\\
{\bf (a)} We have a linear isomorphism
\[
\Xi_L^{\pm}: L\overset{\sim}{\longrightarrow} \widehat{L}^{W_{0,\pm}}
\]
defined by $(\Xi_L^{\pm}\phi)(w):=w_{\pm}^{-1}\phi$ for
$w\in W_0$ and $\phi\in L$. Furthermore, $\bigl(\Xi_L^{\pm}\bigr)^{-1}(h)=h(e)$
for $h\in\widehat{L}^{W_{0,\pm}}$. The map $\Xi_L^{\pm}$ restricts
to a linear isomorphism
\[ \textup{SpM}_L^{k,q,\pm}(W_0\gamma^{-1})\overset{\sim}{\longrightarrow}
\textup{SpM}_{\widehat{L}}^{k,q,\pm}(W_0\gamma^{-1})^{W_{0,\pm}}.
\]
{\bf (b)} For $\phi\in L$ and $w\in W_0$ we have 
$\Xi_L^{\pm}(w_{\pm}\phi)=\mu(w)(\Xi_L^{\pm}\phi)$. In particular,
$\Xi_L^{\pm}|_{L^{W_{0,\pm}}}=\textup{id}$ if we take
the identification \eqref{identpure} into account.
\end{lem}
\begin{proof}
This is a straightforward adjust of the proof of
Corollary \ref{identcorr}. The second part of {\bf (a)}
uses the fact that the $D_f^{k,q,\pm}$ ($f\in \mathbb{C}[T]^{W_0}$)
are $W_{0,\pm}$-equivariant, cf. Proposition \ref{CMDEprop}.
\end{proof}

\begin{thm}\label{CMcorr}
Let $L$ be a left $\mathbb{C}_{\sigma,\nabla}^{k,q}[T]\#_qW$-module.\\
{\bf (a)} Let $\gamma\in T_I^{k^{\pm 1}}$. The Cherednik-Matsuo map
$\xi_{L,\gamma}^{k,q,\pm,I}$ (see Definition \ref{CMdef})
restricts to a linear map $W_0$-equivariant (in the sense of
\eqref{xiequiv}) linear map
\[
\xi_{L,\gamma}^{k,q,\pm,I}: 
\Gamma_L^{k,q}\bigl(M^{k,\pm,I}(\gamma)\bigr)_\nabla^{\tau(P^\vee)}
\rightarrow \textup{SpM}_L^{k,q,\pm}(W_0\gamma^{-1}).
\]
The map is injective if $\gamma\in T_{I,reg}^{k^{\pm 1}}$ (see \eqref{TIreg}).\\
{\bf (b)} Let $\gamma\in T$ such that $k_{\alpha^\vee}^2\not=\gamma^{\alpha^\vee}
\not=1$ for all $\alpha\in R_0$. Then $\xi_{L,\gamma}^{k,q,\pm}$ restricts
to a $W_0$-equivariant (in the sense of
\eqref{xiequiv}) linear isomorphism
\[
\xi_{L,\gamma}^{k,q,\pm}: 
\Gamma_L^{k,q}\bigl(M^{k}(\gamma)\bigr)_\nabla^{\tau(P^\vee)}
\overset{\sim}{\longrightarrow} \textup{SpM}_L^{k,q,\pm}(W_0\gamma^{-1}).
\]
\end{thm}
\begin{proof}
Let $\gamma\in T_I^{k^{\pm 1}}$ and let
\[\xi: \Gamma_L^{k,q}\bigl(M^{k,\pm,I}(\gamma)\bigr)_\nabla^{\tau(P^\vee)}
\rightarrow \textup{SpM}_L^{k,q,\pm}(W_0\gamma^{-1})
\]
be the linear map such that the following diagram\\ 
\vspace{.3cm}
$\qquad\qquad\qquad\qquad$
\xymatrix{ 
\Gamma_L^{k,q}\bigl(M^{k,\pm,I}(\gamma)\bigr)_\nabla^{\tau(P^\vee)}
\ar[r]_{\Xi}^{\sim}
\ar[d]^{\xi} & 
\Gamma_{\widehat{L}}^{k,q}\bigl(M^{k,\pm,I}(\gamma)\bigr)_\nabla^W
\ar[d]^{\xi_{\widehat{L},\gamma}^{k,q,\pm,I}}\\
\textup{SpM}_L^{k,q,\pm}(W_0\gamma^{-1})
\ar[r]_{\Xi_L^{\pm}\,\,\,\,\,\,\,\,}^{\sim\,\,\,\,\,\,\,\,}& 
\textup{SpM}_{\widehat{L}}^{k,q,\pm}(W_0\gamma^{-1})^{W_{0,\pm}}
}\\
is commutative. 
In view of Proposition \ref{CMcorrsym}
(applied to the $\mathbb{C}_{\sigma,\nabla}^{k,q}[T]\#_qW$-module $\widehat{L}$)
it then suffices to show that
\[\xi=\xi_{L,\gamma}^{k,q,\pm,I}.
\]
Let $\phi\in\Gamma_L^{k,q}\bigl(M^{k,\pm,I}(\gamma)\bigr)_\nabla^{\tau(P^\vee)}$
and write $\Xi(\phi)=
\sum_{w\in W_0^I}\widehat{\psi}_w\otimes v_w^{k,\pm,I}(\gamma)$
with $\widehat{\psi}_w\in\widehat{L}$. Then by Corollary \ref{corspin}
and Lemma \ref{515},
\begin{equation*}
\begin{split}
\xi_{L,\gamma}^{k,q,\pm,I}(\phi)&=\xi_{L,\gamma}^{k,q,\pm,I}((\Xi \phi)(e))\\
&=\sum_{w\in W_0^I}\epsilon_{\pm}^k(T_w)\widehat{\psi}_w(e)\\
&=\bigl(\xi_{\widehat{L},\gamma}^{k,q,\pm,I}(\Xi \phi)\bigr)(e)\\
&=\bigl(\Xi_L^{\pm}\bigr)^{-1}\xi_{\widehat{L},\gamma}^{k,q,\pm,I}\Xi(\phi)\\
&=\xi(\phi),
\end{split}
\end{equation*}
as desired.
\end{proof}
\begin{rema}
Cherednik used different methods in \cite{CInd} to
show that 
$\xi_{L,\gamma}^{k,q,+}$ 
defines a linear map $\xi_{L,\gamma}^{k,q,+}: 
\Gamma_L^{k,q}(M^k(\gamma))_\nabla^{\tau(P^\vee)}
\rightarrow \textup{SpM}_L^{k,q,+}(W_0\gamma^{-1})$. 
The advantage of the present techniques is that they lead
to precise conditions on $\gamma$ and $k$
to ensure that $\xi_{L,\gamma}^{k,q,+}$ 
is a linear isomorphism. Such properties of the map $\xi_{L,\gamma}^{k,q,+}$
(for special choices of $L$ and $q$) were also discussed in
\cite[Thm. 4.3]{CInd}. It is though not clear to the author that 
the suggested proof of \cite[Thm. 4.3]{CInd} (based on
classical techniques from, e.g., \cite[Part I, Section 4.1]{HS}) 
works out in the present $q$-setup.
\end{rema}
\begin{rema}
The classical analogue of Theorem \ref{CMcorr}{\bf (b)}
is due to Matsuo \cite[Thm. 5.4.1]{Mat}
(in the symmetric ($+$) case) and Cherednik \cite[Thm. 4.7]{CInt}.
The arguments leading to Theorem \ref{CMcorr}{\bf (b)}
is motivated by Opdam's \cite[\S3 ]{O} approach to this
classical correspondence. 
It is likely that Theorem \ref{CMcorr}{\bf (b)}
can be strenghened a bit, in the sense that the genericity
conditions $k_{\alpha^\vee}^2\not=\gamma^{\alpha^\vee}\not=1$ ($\alpha\in R_0$)
can be weaked further
(so that they match the genericity conditions
in \cite[Thm. 5.4.1]{Mat}, \cite[Thm. 4.7]{CInt} and 
\cite[Cor. 3.12]{O} for the classical correspondence). Note that
the sharpened genericity conditions in \cite[Cor. 3.12]{O}, compared
to e.g. the genericity conditions in \cite[Cor. 3.11]{O}, are justified
by the fact that the analogue of 
$\textup{SpM}_L^{k,q,+}(W_0\gamma^{-1})$ 
is known to be of dimension $\#W_0$;
such a type of result is not known in the present setup, as far as we know.
\end{rema}
\begin{rema}
Using the notations from Subsection \ref{GL}, Theorem \ref{CMcorr}
holds true in the $\textup{GL}_m$-case. Here $P^\vee$ should be taken
to be the lattice $\mathbb{Z}^m$, the complex torus is
$T=\bigl(\mathbb{C}^\times\bigr)^m$ and $W_0=S_m$.
In the symmetric case ($+$) the commuting $q$-difference operators
$D_f^{k,q,+}$ ($f\in\mathbb{C}[T]^{S_m}$) were obtained for the first
time by Ruijsenaars \cite{R}. Concretely,
for the elementary symmetric functions $e_i\in\mathbb{C}[T]^{S_m}$
($1\leq i\leq m$) defined by
\[e_i(t)=\sum_{\stackrel{I\subseteq\{1,\ldots,m\}}{\#I=i}}
\prod_{j\in I}t_j,
\]
we have the explicit expressions
\[
D_{e_i}^{k,q,+}=
\sum_{\stackrel{I\subseteq\{1,\ldots,m\}}{\#I=i}}
\left(
\prod_{\stackrel{r\in I}{s\not\in I}}\frac{kt_r-k^{-1}t_s}{t_r-t_s}
\right)\tau\bigl(\sum_{r\in I}\epsilon_r\bigr)\in
\mathbb{C}[T]_{\sigma,\nabla}^{k,q}\#_q\mathbb{Z}^m,
\] 
see, e.g., \cite[\S 1.3.5]{C} and \cite{KN}.
\end{rema}
Recall the morphism \eqref{funtapp} of $\mathbb{C}[W]$-modules,
valid for $\gamma\in T_I^{k^{\pm 1}}$.
Combined with Theorem \ref{CMcorr} we obtain
\begin{cor}\label{injectivitycor}
Let $\gamma\in T_I^{k^{\pm 1}}$ and let $L$ be a left 
$\mathbb{C}_{\sigma,\nabla}^{k,q}[T]\#_qW$-module. We have a commutative
diagram of linear maps\\
\vspace{.3cm}
$\qquad\qquad\qquad\qquad$
\xymatrix{\Gamma_L^{k,q}\bigl(M^{k}(\gamma)\bigr)_\nabla^{\tau(P^\vee)}
\ar[r]
\ar[d]^{\xi_{L,\gamma}^{k,q,\pm}} & 
\Gamma_{L}^{k,q}\bigl(M^{k,\pm,I}(\gamma)\bigr)_\nabla^{\tau(P^\vee)}
\ar[ld]^{\xi_{L,\gamma}^{k,q,\pm,I}}\\
\textup{SpM}_L^{k,q,\pm}(W_0\gamma^{-1})
}\\
where the horizontal arrow is the restriction of the map
\eqref{funtapp} to $\Gamma_L^{k,q}(M^{k}(\gamma))_\nabla^{\tau(P^\vee)}$.
The southwest arrow represents an injective map
if $\gamma\in T_{I,reg}^{k^{\pm 1}}$ (see \eqref{TIreg}).
\end{cor}
\subsection{An application}
The Cherednik-Matsuo correspondence (Theorem \ref{CMcorr}) 
with $L=\mathcal{M}(T)$ and $0<|q|<1$ is instrumental for 
quantum (noncompact) harmonic analysis; the author will discuss this
in the forthcoming second part \cite{Sprep} of this paper.
In this subsection we consider consequences of the Cherednik-Matsuo
correspondence when $q=1$; this case actually also deserves much more
attention in view of the potential applications to spin chains and
the Razumov-Stroganov conjectures, see, e.g., \cite{RS, Pas, dFZ, KP}.
We hope to return to this in more detail in future work.

For $f\in\mathbb{C}[T]^{W_0}$ we write the
$q$-difference operator $D_f^{k,q,\pm}\in\mathbb{D}_{\sigma,\nabla}^{k,q}$
as
\[D_f^{k,q,\pm}=\sum_{\lambda\in P^\vee}u_{f,\lambda}^{k,q,\pm}\tau(\lambda)
\]
with $u_{f,\lambda}^{k,q,\pm}\in\mathbb{C}_{\sigma,\nabla}^{k,q}[T]$. 

The proof of the following lemma hinges on the theory of (anti)symmetric
Macdonald polynomials (cf. Remark \ref{antisymmpol}). Recall the
definition \eqref{deltapm} of $\delta_{\pm}^k\in T$. 
\begin{lem}
Let $f\in\mathbb{C}[T]^{W_0}$. Then
\[
\sum_{\lambda\in P^\vee}u_{f,\lambda}^{k,1,\pm}=
f\bigl(\delta_+^k\bigr)
\]
as identity in $\mathbb{C}_{\sigma,\nabla}^{k,1}[T]$.
\end{lem}
\begin{proof}
We use the notations from Remark \ref{antisymmpol}.
Suppose for the moment that we have fixed $q$ and $k$
satisfying the generic conditions as in Remark \ref{antisymmpol}.
Then it is known that, up to a nonzero constant,
\[P_{0}^{(+)}(k^{-1},q)=1,\qquad P_{\rho^\vee}^{(-)}(k^{-1},q)=G^{k,-},
\]
where $1\in \mathbb{C}[T]$ is the constant function one,
cf., e.g., \cite[(5.8.10)]{M} for the second equality. Thus
for all possible values of $q$ and $k$,
\begin{equation}\label{klk}
\begin{split}
1&\in \textup{SpM}_L^{k,q,+}(W_0\gamma_0(k^{-1},q)^{-1})^{W_{0,+}},\\
G^{k,-}&\in 
\textup{SpM}_L^{k,q,-}(W_0\gamma_{\rho^\vee}(k^{-1},q)^{-1})^{W_{0,-}}
\end{split}
\end{equation}
with $L=\mathbb{C}_{\sigma,\nabla}^{k,q}[T]$.
Since
\[\gamma_0(k^{-1},q)^{-1}=\prod_{\alpha\in R_0^+}k_{\alpha^\vee}^{-\alpha}=
w_0(\delta_+^k)
\]
(independent of $q$) and 
\[
\gamma_{\rho^\vee}(k^{-1},q)^{-1}=q^{-\rho^\vee}\prod_{\alpha\in R_0^+}
k_{\alpha^\vee}^\alpha=q^{-\rho^\vee}\delta_+^k,
\]
we get for all $f\in\mathbb{C}[T]^{W_0}$,
\[\sum_{\lambda\in P^\vee}u_{f,\lambda}^{k,q,+}=D_f^{k,q,+}(1)=
f\Bigl(\gamma_0(k^{-1},q)^{-1}\Bigr)=
f\bigl(\delta_+^k\bigr)
\]
in $\mathbb{C}_{\sigma,\nabla}^{k,q}[T]$
(this is valid for all $q\in\mathbb{C}^\times$, in particular for $q=1$),
as well as
\[
\sum_{\lambda\in P^\vee}u_{f,\lambda}^{k,1,-}=
\bigl(G^{k,-}\bigr)^{-1}D_f^{k,1,-}\bigl(G^{k,-}\bigr)
=f\Bigl(\gamma_{\rho^\vee}(k^{-1},1)^{-1}\Bigr)
=f\bigl(\delta_+^k\bigr)
\]
in $\mathbb{C}_{\sigma,\nabla}^{k,1}[T]$.
\end{proof}

\begin{prop}\label{TM}
Let $L$ be a 
left $\mathbb{C}_{\sigma,\nabla}^{k,1}[T]\#W_0$
module. 
Turn $L$ into a $\mathbb{C}_{\sigma,\nabla}^{k,1}[T]\#_1W$
module by letting $\tau(P^\vee)$ act trivially on $L$.
Suppose that $\gamma\in T_{I,reg}^{k^{\pm 1}}$ (see \eqref{TIreg})
and $\gamma\not\in W_0(\delta_+^k)$.
Then
$\Gamma_L^{k,1}\bigl(M^{k,\pm,I}(\gamma)\bigr)_\nabla^{\tau(P^\vee)}=\{0\}$. 
\end{prop}
\begin{proof}
Suppose that $\gamma\in T_{I,reg}^{k^{\pm 1}}$ and 
$\gamma\not\in W_0(\delta_+^k)$.
In view of Theorem \ref{CMcorr}{\bf (a)} 
it suffices to prove that
$\textup{SpM}_L^{k,1,\pm}(W_0\gamma^{-1})=\{0\}$. 

Since $L=L^{\tau(P^\vee)}$, 
the previous lemma shows that
\[
\textup{SpM}_L^{k,1,\pm}(W_0\gamma^{-1})=
\{\phi\in L \,\,\, | \,\,\, f\bigl(\delta_+^k\bigr)\phi=
f(\gamma^{-1})\phi\qquad 
\forall\, f\in \mathbb{C}[T]^{W_0}\}.
\]
Consequently
\begin{equation*}
\textup{SpM}_L^{k,1,\pm}(W_0\gamma^{-1})=
\begin{cases}
L\quad &\hbox{ if }\, \gamma\in 
W_0\bigl(\delta_+^k\bigr),\\
\{0\}\quad &\hbox{ if }\, \gamma\not\in W_0\bigl(\delta_+^k\bigr),
\end{cases}
\end{equation*}
hence the result.
\end{proof}

\begin{rema} 
Let $L$ be field and
a left $\mathbb{C}_{\sigma,\nabla}^{k,q}[T]\#_qW$
module such that $W$ acts by field automorphisms on $L$.
Let $M$ be a finite dimensional $H(k)$-module over $\mathbb{C}$.
Then it can be shown that
\begin{equation}\label{dimension}
\textup{Dim}_{L^{\tau(P^\vee)}}
\bigl(\Gamma_L^{k,q}\bigl(M\bigr)_\nabla^{\tau(P^\vee)}
\bigr)\leq \textup{Dim}_{\mathbb{C}}(M).
\end{equation}
Since the quantum affine KZ equations are holonomic it is natural
to expect equality in \eqref{dimension}.
This is for instance the case if $0<|q|<1$ and if $L=\mathcal{M}(T)$ with
the left action of $\mathbb{C}_{\sigma,\nabla}^{k,q}[T]\#_qW$ on $\mathcal{M}(T)$
by $q$-difference reflection operators (see \cite{Et}). In general
equality does not hold true though, as Proposition \ref{TM} shows.\\
\end{rema}


\end{document}